\documentclass[11pt, reqno]{amsart} 
\usepackage{amssymb, amsmath, amsfonts}
\usepackage{mathrsfs}
\usepackage{upgreek}
\usepackage{bbm, bm}

\usepackage[utf8]{inputenc}
\usepackage{BOONDOX-calo}

\DeclareFontFamily{OT1}{pzc}{}
\DeclareFontShape{OT1}{pzc}{m}{it}{<-> s * [1.200] pzcmi7t}{}
\DeclareMathAlphabet{\mathpzc}{OT1}{pzc}{m}{it}

\usepackage[left= 3.5 cm, right= 3.5 cm, top= 3.5 cm, bottom=3.2 cm, foot= 1.1 cm, marginparwidth=2.7 cm, marginparsep=0.5 cm]{geometry}
\usepackage{hyperref}
\hypersetup{colorlinks=true, citecolor=blue}
\usepackage{color}

\newcommand{\mb}{\mathbb}

\newcommand{\mf}{\mathfrak}
\newcommand{\mc}{\mathcal}

\newcommand{\ov}{\overline}
\newcommand{\ud}{\underline}

\newcommand{\wt}{\widetilde}

\newcommand{\arc}{A_R^{\sf C}}
\newcommand{\ev}{{\rm ev}}
\newcommand{\beqn}{\begin{equation*}}
\newcommand{\eeqn}{\end{equation*}}
\newcommand{\beq}{\begin{equation}}
\newcommand{\eeq}{\end{equation}}

\makeatletter 
\@addtoreset{equation}{section}
\makeatother  

\numberwithin{equation}{section}

\newtheorem{thm}{Theorem}[section]
\newtheorem{lemma}[thm]{Lemma}
\newtheorem{cor}[thm]{Corollary}
\newtheorem{prop}[thm]{Proposition}
\newtheorem{conj}[thm]{Conjecture}
	
\theoremstyle{definition}
\newtheorem{defn}[thm]{Definition}
	
\theoremstyle{remark}
\newtheorem{rem}[thm]{Remark}
\newtheorem{hyp}[thm]{Hypothesis}

\newtheorem{ques}[thm]{Question}

\title{Compactness in the adiabatic limit of disk vortices}

\author[Wang]{Dongning Wang}
\address{Center of Geometry and Physics\\
Institute for Basic Science, Pohang, Korea}
\email{dwang@math.wisc.edu}

\author[Xu]{Guangbo Xu}
\address{
Department of Mathematics \\
Princeton University\\
Fine Hall, Washington Road\\
Princeton, NJ 08544 USA
}
\email{guangbox@math.princeton.edu}
\thanks{G. X. is supported by AMS-Simons Travel Grant.}

\begin{document}

\begin{abstract}
This paper is the first input towards an open analogue of the quantum Kirwan map. We consider the adiabatic limit of the symplectic vortex equation over the unit disk for a Hamiltonian $G$-manifold with Lagrangian boundary condition, by blowing up the metric on the disk. We define an appropriate notion of stable solutions in the limit, and prove that any sequence of disk vortices with energy uniformly bounded has a subsequence converging to such a stable object. We also proved several analytical properties of vortices over the upper half plane, which are new type of bubbles appearing in our compactification.
\end{abstract}

\maketitle


\setcounter{tocdepth}{1}
\tableofcontents

\section{Introduction}

In this paper we initiate a project which intends to develop an open-string analogue of the quantum Kirwan map, proposed by Woodward (see \cite[Section 8]{Woodward_toric}). This paper specifically aims at the compactness problem in adiabatic limit for the symplectic vortex equation on disks, which is the open-string analogue of the ``bubbling'' part of Ziltener's work \cite{Ziltener_book}. In this introduction, we review works related to the closed quantum Kirwan map and briefly describe the idea of its open analogue as well as the main results of the current paper.

\subsection{The quantum Kirwan map}\label{subsection11}

Symplectic vortex equation, introduced by Cieliebak-Gaio-Salamon \cite{Cieliebak_Gaio_Salamon_2000} and Ignasi Mundet \cite{Mundet_thesis} \cite{Mundet_2003}, is the classical equation of motion of the gauged $\sigma$-model with targets in a Hamiltonian $G$-manifold (see Section \ref{section2} for a brief review of symplectic vortex equation). It is an equivariant generalization of the nonlinear Cauchy-Riemann equation (pseudoholomorphic curves) in an almost K\"ahler manifold, where the latter has been a fundamental tool in symplectic geometry since the pioneering works of Gromov and Floer. The study of symplectic vortex equation leads to the so-called gauged Gromov-Witten (or Hamiltonian Gromov-Witten) theory. Suppose $G$ is a compact Lie group and $(X, \omega)$ is a Hamiltonian $G$-manifold with a moment map $\mu: X \to {\mf g}^*$. The moduli space of solutions to the symplectic vortex equation gives rise to a correlation function 
\begin{align}\label{eqn11}
&\ \big\langle a_1, \ldots, a_n \big\rangle_{X}^G \in {\mb Q},\ &\ n \geq 0,\ a_1, \ldots, a_n \in H^*_G( X ).
\end{align}
This is called the gauged Gromov-Witten or Hamiltonian Gromov-Witten invariant of $X$. It has been rigorously defined in certain cases (see Mundet \cite{Mundet_2003}, Cieliebak-Gaio-Mundet-Salamon \cite{Cieliebak_Gaio_Mundet_Salamon_2002}, Mundet-Tian \cite{Mundet_Tian_2009} \cite{Mundet_Tian_Draft} in the symplectic setting and Gonzalez-Woodward \cite{Gonzalez_Woodward_abelian} in the algebraic setting).

There is an interesting correspondence between solutions to the vortex equation in $X$ and pseudoholomorphic curves in the symplectic quotient $\bar{X} = \mu^{-1}(0)/ G$. The symplectic vortex equation and the correlation \eqref{eqn11} depend on a scale parameter $\lambda>0$ which is the size of the domain curve $\Sigma$. By studying the $\lambda \to +\infty$ limit of the vortex equation, one may expect a relation
\beq\label{eqn12}
\big\langle \alpha_1, \ldots, \alpha_n \big\rangle_{X}^G \sim \big\langle \kappa^X (\alpha_1), \ldots, \kappa^X (\alpha_n) \big\rangle_{\bar{X}}.
\eeq
Here $\kappa^X : H^*_G(X) \to H^*( \bar{X} )$ is the (classical) Kirwan map and $\langle \ \rangle_{\bar{X}}$ is the Gromov-Witten correlator. Gaio-Salamon \cite{Gaio_Salamon_2005} then proved that in the context of \cite{Cieliebak_Gaio_Mundet_Salamon_2002}, for monotone manifolds \eqref{eqn12} is actually an equality when degrees of all $\alpha_i$'s are small. The failure of \eqref{eqn12} being an equality in general is because in the adiabatic limit, vortices don't converge to holomorphic curves in $\bar{X}$ uniformly, but will bubble off ``affine vortices''. These objects are solutions to the vortex equation on the complex plane. Salamon conjectured that counting affine vortices gives rise to a quantum deformation 
\beq\label{eqn13}
q\kappa^X: H^*_G ( X; \Lambda ) \to H^* ( \bar{X}; \Lambda ),\ q\kappa^X = \kappa^X + {\bm t} q\kappa^X_1 + {\bm t}^2 q\kappa^X_2 + \cdots 
\eeq
of the Kirwan map, where ${\bm t}$ is a formal variable in the coefficient ring $\Lambda$. \eqref{eqn12} should become an identity if $\kappa^X$ is replaced by the quantum Kirwan map $q\kappa^X$. In symplectic setting, certain preliminary results towards the proof of Salamon's conjecture have been obtained by F. Ziltener (\cite{Ziltener_Decay}\cite{Ziltener_book}). 

When $X$ is a smooth projective variety and the $G$-action extends to an action by its complexification, the quantum Kirwan map is constructed by Woodward \cite{Woodward_15}, using the algebraic construction of virtual cycles on the moduli space of affine vortices. An important perspective in \cite{Woodward_15} is that the equivariant quantum cohomology of $X$ (introduced by Givental \cite{Givental_96}) and the quantum cohomology of the quotient $\bar{X}$ are viewed as cohomological field theory algebras (CohFT algebras for short). Then the gauged Gromov-Witten invariants of $X$ and Gromov-Witten invariants of $\bar{X}$ are ``traces'' on these two CohFT algebras, and the quantum Kirwan map $q\kappa^X$ is a morphism between these two CohFT algebras that relates the two traces. To foreshadow our project, we remark that CohFT algebras can be viewed as complexifications of $A_\infty$ algebras.

Another important aspect of the quantum Kirwan map is related to mirror symmetry. In some sense, the quantum Kirwan map is essentially the mirror map obtained by Givental \cite{Givental_96} and Lian-Liu-Yau \cite{LLY_1} in proving Candelas-de la Ossa-Green-Parkes mirror formula (\cite{COGP}). The philosophy behind this relation is explained by Hori-Vafa \cite{Hori_Vafa} (cited from the introduction of \cite[Part I]{Woodward_15}): the mirror symmetry for vector spaces is trivial and the nontrivial change of coordinates as the mirror formula appears when passing from gauged linear $\sigma$-model to nonlinear sigma model. The quantum Kirwan map is also used to compute the quantum cohomology of symplectic quotients, such as in \cite{GW_Toric} \cite{CLLT}.

\subsection{The open quantum Kirwan map}\label{subsection12}

Now we describe our project on defining an open analogue of the quantum Kirwan map. Let $\bar{X}$ be symplectic (not necessarily obtained from symplectic reduction) and $\bar{L} \subset \bar{X}$ be an embedded Lagrangian submanifold. The Fukaya $A_\infty$ algebra $\mc{Fuk}(\bar{L})$ of $\bar{L}$ consist of a chain group $C( \bar{L}; \Lambda)$ of $\bar{L}$ with coefficients in a Novikov ring $\Lambda$ and a collection of higher compositions $\{ \bar{\mf m}_{\ud k}\}_{\ud k \geq 0}$. The compositions are defined by counting isomorphism classes of pseudoholomorphic disks $u: ({\mb D}, \partial {\mb D}) \to ( \bar{X}, \bar{L})$ with $\ud k + 1 $ marked points mapped into $\ud k + 1$ chains in $\bar{L}$, which are multilinear maps
\beq\label{eqn14}
\bar{\mf m}_{\ud k}:  C( \bar{L}; \Lambda)^{\otimes \ud k} \to C( \bar{L}; \Lambda),\ \ud k = 0, 1, \ldots
\eeq
Different choices of data such as the almost complex structures lead to $A_\infty$ algebras that are homotopy equivalent to each other (cf. Chapter 4 of \cite{FOOO_Book}). 

Another family of relevant objects are the correlations 
\begin{align}\label{eqn15}
&\ \bar{\tau}_{\ud k} (b_1, \ldots, b_{\ud k}) = \langle b_1, \ldots, b_{\ud k} \rangle\in \Lambda,&\ \ud k \geq 0,\ b_1, \ldots, b_{\ud k} \in C( \bar{L}; \Lambda)
\end{align}
defined by counting isolated {\it parametrized} marked holomorphic disks with constrain at boundary markings given by chains in $\bar{L}$. The collection $\{ \bar{\tau}_{\ud k} \}_{\ud k \geq 0}$ satisfies certain splitting axioms involving the $A_\infty$ structure of $\mc{Fuk}(\bar{L})$, which we will call an {\bf $A_\infty$ trace} on $\mc{Fuk}(\bar{L})$. More generally, for any quantum cohomology class $a\in QH(\bar{X})$, one can define the {\it bulk deformation} of the $A_\infty$ structure, or the correlations \eqref{eqn15}, by allowing interior markings on the disk. The deformed correlations give rise to  
\beq\label{eqn16}
\bar{\tau}_{\ud k}(a; b_1, \ldots, b_{\ud k}) = \sum_k \frac{1}{k!} \langle \underbrace{a, \ldots, a}_{k}; b_1, \ldots, b_{\ud k} \rangle.
\eeq

If $\bar{X}$ is obtained by symplectic reduction from a Hamiltonian $G$-manifold $(X, \omega, \mu)$ and $L \subset \mu^{-1}(0)$ is the $G$-invariant lift of $\bar{L}$, then in \cite{Woodward_toric} another $A_\infty$ algebra $\mc{QF}(L)$, called the {\bf quasimap $A_\infty$ algebra}, can be constructed counting holomorphic disks in $(X, L)$ modulo $G$-action. Since chains in $\bar{L}$ can be lifted to $G$-invariant chains in $L$, $\mc{QF}(L)$ has the same underlying chain group $(C(\bar{L}; \Lambda)$, though differs from $\mc{Fuk}(\bar{L})$ in the composition maps.

A {\bf disk vortex} is a pair $(A, u)$, where $A$ is a connection on the trivial $G$-bundle over ${\mb D}$, and $u$ is a map from ${\mb D}$ to $X$, satisfying the {\it vortex equation}:
\beq\label{eqn17}
\ov\partial_A u = 0,\ F_A + \lambda^2 \mu(u) \nu_0 = 0,\ u(\partial {\mb D}) \subset L.
\eeq
(See Section \ref{section2} for explanations for the terms in the equation.) The parameter $\lambda$ measures the scale. Since disk vortices may bubble off holomorphic disks or spheres, counting solutions to \eqref{eqn17} up to gauge equivalence should define an $A_\infty$ trace on $\mc{QF}(L)$, which is parametrized by $\lambda$ and denoted by
\beq\label{eqn18}
\tau^\lambda_{\ud k}(b_1, \ldots, b_{\ud k}) \in \Lambda,\ b_1, \ldots, b_k \in C(\bar{L}; \Lambda).
\eeq

Since CohFT algebras can be viewed as complexifications of $A_\infty$ algebras, it is natural to consider the open analogue of the quantum Kirwan map, as originally proposed in \cite[Section 8]{Woodward_toric}. We look for an $A_\infty$ morphism $q\kappa^{X, L}: \mc{QF}(L) \to \mc{Fuk}(\bar{L})$, which conjecturally is defined by counting affine vortices over the upper half plane ${\mb H}$. $q\kappa^{X, L}$ should intertwine the trace $\bar\tau$ on $\mc{Fuk}(\bar{L})$ (with some bulk deformation) with the $\lambda \to \infty$ limit of the trace $\tau^\lambda$ on $\mc{QF}(L)$. More precisely, we state the following conjecture. 

\begin{conj}\label{conj1}
Let $q\kappa_0^X: \Lambda \to QH(\bar{X})$ be the zeroth component of the closed quantum Kirwan map $q\kappa^X: QH^G(X) \to QH(\bar{X})$ and denote $a_X = q\kappa_0^X(1)\in QH(\bar{X})$. Then the counting of affine vortices over ${\mb H}$ defines an $A_\infty$ morphism 
\beqn
q\kappa^{X, L}: \mc{QF}(L) \to \mc{Fuk}(\bar{L}; a_X),
\eeqn
where $\mc{Fuk}(\bar{L}; a_X)$ is the bulk-deformation of $\mc{Fuk}(\bar{L})$ by $a_X$. The $A_\infty$ morphisms defined using different choice of data are $A_\infty$ homotopic to each other. 

Moreover, the gauged correlation function \eqref{eqn18} and the bulk-deformed correlation function \eqref{eqn16} on $\mc{Fuk}(\bar{L}; a_X)$ are related via the adiabatic limit (possibly up to $A_\infty$ homotopy)
\beq\label{eqn19}
\bar\tau^{X, L}_{\ud k} \big( a_X; q\kappa^{X, L}(b_1), \ldots, q\kappa^{X, L}(b_{\ud k}) \big) = \lim_{ \lambda \to +\infty} \tau^\lambda_{\ud k} \big( b_1, \ldots, b_{\ud k} \big).
\eeq
\end{conj}

This conjecture has important implications, which we will discuss later on. Towards the resolution of this conjecture, there are many tasks yet to be accomplished . In the forthcoming work \cite{Woodward_Xu}, Woodward and second named author plan to use the stablizing divisor technique of Cieliebak-Mohnke \cite{Cieliebak_Mohnke} to treat the related transversality problem and construct the $A_\infty$ morphism $q\kappa^{X, L}$. The stabilizing divisor approach has been used in \cite{Charest_Woodward_2} \cite{Charest_Woodward_2014} to construct Floer cohomology and Fukaya algebra. However, the corresponding version of $\mc{QF}(L)$ using the stabilizing divisor technique has not been written down. On the other hand, a crucial part in proving the $A_\infty$ relations for $q\kappa^{X, L}$ and the relation \eqref{eqn19} is certain gluing construction for affine vortices, which will be treated in \cite{Xu_glue}. The analytical framework for the gluing construction has been provided in \cite{Venugopalan_Xu}.

\subsection{Potential functions in the toric case}\label{subsection13}

Besides the generalization from the closed case, another motivation of studying the open quantum Kirwan map is to understand, in the toric case, the meaning of the coordinate change between the Lagrangian Floer potential and the Givental-Hori-Vafa potential. Here we briefly describe this application.

If $( \mc{Fuk}(\bar{L}), \{ \bar{\mf m}_{\ud k}\}_{\ud k \geq 0}\big)$ has a cohomological unit ${\bm 1}_L$, then a solution $b$ to the $A_\infty$ Maurer-Cartan equation
\beq\label{eqn110}
\sum_{\ud k=0}^\infty \bar{\mf m}_{\ud k} (b, \ldots, b) \equiv 0\ {\rm mod}\ {\bm 1}_L
\eeq
is called a weakly bounding cochain, which can be used to deform the map $\bar{\mf m}_1$, which gives a chain complex $( C(\bar{L}; \Lambda); \bar{\mf m}_1^b )$. Its homology gives obstructions of displacing the Lagrangian $\bar{L}$ by Hamiltonian diffeomorphisms of $\bar{X}$. In the view of mirror symmetry, an important object is the {\bf potential function}
\beqn
\mf{PO}: \hat{\mc M}_{weak}(\bar{L}) \to \Lambda,\ \  \sum_{k=0}^\infty {\mf m}_k(b, \ldots, b) = \mf{PO}(b)\cdot {\bm 1}_L.
\eeqn
Here $\hat{\mc M}_{weak}(\bar{L})$ is the set of all weakly bounding cochains. 

If $\bar{X}$ admits a special Lagrangian torus fibration (the SYZ picture), then $\mf{PO}$ induces a holomorphic function $W$ on the dual torus fibration $\bar{X}^\vee$ and mirror symmetry predicts relations between the Fukaya category of $\bar{X}$ and B-side theories of $(\bar{X}^\vee, W)$ (see more details in the expository article \cite{Auroux_09}). In the toric case, the potential $\mf{PO}$ is closely related to the Givental-Hori-Vafa potential, which coincides with the potential function on $\mc{QF}(L)$ (see \cite{Givental_potential}, \cite{Hori_Vafa}, \cite{Woodward_toric}). In the Fano case, the two potential functions are proven to be the same in \cite{Cho}, \cite{Cho_Oh}, \cite{FOOO_toric_1} by direct computation of the Lagrangian Floer potential\footnote{Their works are for different coefficients.}. For semi-Fano toric orbifolds, Chan-Lau-Leung-Tseng showed in \cite{CLLT} that the mirror map is the needed coordinate change, while for general toric manifolds, Fukaya-Oh-Ono-Ohta constructed a formal coordinate change by induction in \cite{FOOO_mirror}. 

Beyond the semi-Fano case, the geometric meaning of the coordinate change has not been understood clearly. However, if Conjecture \ref{conj1} is true, then $q\kappa^{X, L}$ sends solutions to the Maurer-Cartan equation of $\mc{QF}(L)$ to solutions to \eqref{eqn110}, and then induces a correspondence between the potential functions. Therefore, the coordinate change between this two potential functions essentially coincides with the open quantum Kirwan map.

\subsection{Main results of this paper}\label{subsection14}

In this paper we resolved a basic analytical problem in our project, i.e., the compactness of the $\lambda \to +\infty$ adiabatic limit of \eqref{eqn17}. An extension of the results to the case of polygons involving a collection of Lagrangians can be found in an earlier arXiv version of this paper. This compactness result is an important step towards proving relation \eqref{eqn19}. Meanwhile, we prove basic analytical properties of affine vortices over the upper half plane.

Consider a Hamiltonian $G$-manifold $X$ with symplectic quotient $\bar{X}$, satisfying Hypothesis \ref{hyp4}. Consider an embedded Lagrangian submanifold $\bar{L}$ contained in the smooth locus of $\bar{X}$. $\bar{L}$ lifts to $G$-invariant Lagrangian submanifolds $L\subset X$, contained in $\mu^{-1}(0)$, and $G$ acts freely on $L$.

Let $\wt{\mc M}_{\ud k}^\lambda(L)$ be the set of pairs $({\bf v}, {\bf w})$ with ${\bf v} = (A, u)$ solving \eqref{eqn17} and ${\bf w} \subset \partial{\mb D}$ is a set of $\ud k$ boundary marked points. Let ${\mc M}_{\ud k}^\lambda(L)$ be the set of gauge equivalence classes. A natural topology can be put on $\displaystyle \sqcup_{\lambda\geq 0} {\mc M}_{\ud k}^\lambda(L)$. Its compactification on the $\lambda = \infty$ side is reduced to the following question:

\begin{ques}
Given a sequence $\lambda_i \to \infty$ and $( {\bf v}^{(i)}, {\bf w}^{(i)})\in \wt{\mc M}{}^{\lambda_i}_{\ud k}(L)$. Suppose the energy of ${\bf v}^{(i)}$ is uniformly bounded. Up to choosing a subsequence and gauge transformation, to what geometric object does $({\bf v}^{(i)}, {\bf w}^{(i)})$ converge?
\end{ques}

The same as in the closed case (\cite{Ziltener_book}), we will prove that, away from a finite subset $Z\subset {\mb D}$, a subsequence converges uniformly on compact subsets to a holomorphic disk in the symplectic quotient, while near the points in $Z$, certain bubbling happens. The construction of the subsequence and $Z$ depends on the energy quantization property of the bubbles. In our case, a new type of bubble appears, which we call ${\mb H}$-vortices. They are solutions to the symplectic vortex equation over the upper half plane with Lagrangian boundary condition. A few basic properties of ${\mb H}$-vortices proved in this paper are summarized here.

\begin{thm}\label{thm2} Let $L$ be a $G$-Lagrangian submanifold of $X$.
\begin{enumerate}
\item If ${\bf v} = (A, u)$ is an ${\mb H}$-vortex with finite energy and $u({\mb H})$ has compact closure in $X$, then there exist a gauge transformation $g: {\mb H} \to G$ and a point $x \in L$ such that 
\beqn
\lim_{z \to \infty} g^{-1}(z) u(z) = x.
\eeqn
Moreover, the convergence is exponentially fast.

\item There exists a constant $\epsilon_{X, L}>0$ such that if ${\bf v} = (A, u)$ is an ${\mb H}$-vortex having positive finite energy and $u({\mb H})$ has compact closure, then 
\beqn
E({\bf v}) \geq \epsilon_{X, L}.
\eeqn
\end{enumerate}
\end{thm}

Together with quantization properties of other types of bubbles, this allows us to construct inductively the bubble trees attached at all the bubbling points. Modelling on specific type of trees described in Appendix \ref{appendixb}, we will define a notion of ``stable scaled holomorphic disk'' as the possible limits of sequences $({\bf v}^{(i)}, {\bf w}^{(i)})$ in Section \ref{section4}. Our main theorem can be roughly stated as follows (see Theorem \ref{thm33} for the precise version).
\begin{thm}\label{thm3}
Given a sequence $({\bf v}^{(i)}, {\bf w}^{(i)}) \in \wt{\mc M}{}_{\ud k}^{\lambda_i}(L)$ with $\lambda_i$ diverging to $\infty$ and uniformly bounded energy, there exists a subsequence which converges to a stable scaled holomorphic disk modulo gauge transformation.
\end{thm}

\subsection{Organization of this paper}

In Section \ref{section2} we briefly recall the basic setup of the symplectic vortex equation. We also introduce the central objects as examples of vortex equation. In Section \ref{section3} we introduce the adiabatic limit process of the symplectic vortex equation, and recall some known results about it. We also prove the compactness modulo bubbling result as a preparation of the main result. In Section \ref{section4} we introduce the stable objects which appear in the adiabatic limit, define the notion of convergence, and state our main theorem (Theorem \ref{thm33}). In Section \ref{section5} we prove several technical results which we need to construct the limiting object when proving the main theorem. In Section \ref{section6} we patch everything together and give a complete proof of the main theorem. In Appendix \ref{appendixa} we include useful analytical results about vortices. In Appendix \ref{appendixb} we define proper notion of trees which is used to define the stable objects in this paper. 

\subsection{Acknowledgements} We are very grateful to Chris Woodward for suggesting this problem, sharing many ideas and for his hospitality during our visits to Rutgers University. We would like to thank Prof. Yong-Geun Oh and Prof. Gang Tian for helpful discussions. G. X. would like to thank Institute for Basic Science (Pohang) and Institute for Advanced Study (Princeton) for hospitality and would like to thank Nick Sheridan for helpful discussion.

\section{Vortices with Lagrangian Boundary Condition}\label{section2}

In this section we give a brief review about symplectic vortex equation where the domain is a Riemann surface with possibly nonempty smooth boundary.

\subsection{The symplectic vortex equation}

\subsubsection*{The target space}

Let $G$ be a compact Lie group with Lie algebra ${\mf g}$. Let $(X, \omega, \mu)$ be a Hamiltonian $G$-manifold. This means that $(X, \omega)$ is a symplectic manifold and there is a $G$-action on $X$ with moment map $\mu: X \to {\mf g}^*$. Our convention is, for any $\xi\in {\mf g}$, the infinitesimal action of $\xi$ is the vector field ${\mc X}_\xi$ defined by 
\beqn
{\mc X}_\xi(x) = \left. \frac{d}{dt} \right|_{t=0} e^{t\xi} x.
\eeqn
In this convention the moment map satisfies
\beqn
d ( \mu(\xi) ) = \omega ( {\mc X}_\xi, \cdot ) \in \Omega^1(X),\ \forall \xi \in {\mf g}.
\eeqn
Throughout this paper, we make the following assumptions on $(X, \omega, \mu)$, which are usual assumptions in gauged Gromov-Witten theory for possibly noncompact targets (cf. \cite{Cieliebak_Gaio_Mundet_Salamon_2002}).

\begin{hyp}\label{hyp4}
$(X, \omega, \mu)$ satisfies the following conditions.
\begin{enumerate}
\item $\mu$ is proper, $0$ is a regular value of $\mu$ and the $G$-action on $\mu^{-1}(0)$ is free.

\item (cf. \cite[Definition 2.6]{Cieliebak_Gaio_Mundet_Salamon_2002}) If $X$ is noncompact, then there exist a smooth $G$-invariant proper function $f_X: X \to [0, +\infty)$, a $G$-invariant, $\omega$-compatible almost complex structure $J_X$, and a constant $c_X>0$ such that 
\beqn
{f_X (x) \geq c_X\atop \xi\in T_x X} \Longrightarrow \left\{ \begin{array}{c} \langle \nabla_\xi \nabla f_X(x), \xi \rangle + \langle \nabla_{J_X\xi} \nabla f_X(x), J_X \xi \rangle \geq 0,\\
 df_X(x) \cdot J_X {\mc X}_{\mu(x)} \geq 0, \end{array}	\right.
\eeqn
Here $\nabla$ is the Levi-Civita connection of the metric defined by $\omega$ and $J_X$. 
\end{enumerate}
\end{hyp}

\begin{defn}\label{defn5}
A {\bf $G$-Lagrangian} of $X$ is an embedded Lagrangian submanifold $L \subset X$ such that $L \subset \mu^{-1}(0)$ and $L$ is $G$-invariant.
\end{defn}

By Hypothesis \ref{hyp4}, the symplectic quotient $\bar{X}:= \mu^{-1}(0) / G$ is a symplectic manifold. If $L$ is a $G$-Lagrangian of $(X, \omega, \mu)$, then $\bar{L}:= L/G$ is an embedded submanifold of $\bar{X}$. Toric manifolds (viewed as symplectic quotients of Euclidean spaces) and its smooth toric orbits are typical examples.

In this paper, without additional remarks, all almost complex structures will be $G$-invariant and $\omega$-compatible. We choose an almost complex structure $J$ that coincides outside a compact set with the $J_X$ from Hypothesis \ref{hyp4}.

\subsubsection*{The domain and the fields}

Let $\Sigma$ be a Riemann surface, which is not necessarily compact and may contain a nonempty smooth boundary $\partial \Sigma$. Let $\nu\in \Omega^2(\Sigma)$ be an area form, which induces a Hermitian metric on $\Sigma$. Let $P \to \Sigma$ be a smooth principal $G$-bundle. Denote by ${\mc A}(P)$ the space of smooth $G$-connections on $P$. Denote by $Y:= P\times_G X$ the associated fibre bundle over $\Sigma$. Denote by ${\mc S}(Y)$ the space of smooth sections of $Y$. The adjoint bundle is ${\rm ad}P := P \times_{\rm Ad} {\mf g}$.

\subsubsection*{The equation}

Choose an ${\rm Ad}$-invariant metric on ${\mf g}$. It induces an isomorphism ${\mf g} \simeq {\mf g}^*$ as well as an identification ${\rm ad}P \simeq ({\rm ad}P)^*$.

The {\bf symplectic vortex equation} is the following equation on pairs ${\bf v} = (A, u) \in {\mc A}(P)\times {\mc S} (Y)$, which depends on the choice of $J$ and $\nu$:
\beq\label{eqn21}
\left\{ \begin{array}{ccc}
\ov\partial_A u & = & 0;\\
F_A + \mu(u) \nu & = & 0.
\end{array}\right.
\eeq
We explain the terms appeared in the equation. The connection $A$ induces a horizontal distribution on $Y$. Then for any $u \in {\mc S}(Y)$, the covariant derivative $d_A u$ is the composition of $du: T\Sigma \to TY$ with the projection
\beqn
TY \simeq T^V Y \oplus \pi_Y^* T\Sigma \to T^V Y.
\eeqn
On the other hand, since $J$ is $G$-invariant, $J$ induces a complex structure on $T^V Y$. The term $\ov\partial_A u$ is defined as the $(0, 1)$-part of the linear map $d_A u: T\Sigma \to T^V Y$ with respect to the complex structure on $\Sigma$ and the complex structure on $T^V Y$; $\ov\partial_A u$ is a section of $\Lambda^{0,1} T^* \Sigma \otimes T^V Y$. On the other hand, the curvature form $F_A$ is in the space $\Omega^2 ( \Sigma, {\rm ad}P )$; the moment map $\mu$ induces a section of $\pi_Y^* ( {\rm ad}P)^*$, therefore the composition $\mu(u)$ is in $\Omega^0 ( \Sigma, ({\rm ad} P)^* )$. Via the identification ${\mf g}^* \simeq {\mf g}$, $F_A$ and $\mu(u) \nu$ lie in the same vector space.

With respect to a local trivialization of $P$ over a subset $U \subset \Sigma$, the connection $A$ is identified with a ${\mf g}$-valued 1-form. Choose a holomorphic coordinate $z = s + {\bm i} t$ on $U$, then $A = d+ \phi ds + \psi dt$ for $\phi, \psi: U \to {\mf g}$. Via the trivialization the section $u$ is identified with a map, denoted by the same symbol, $u: U \to X$. We also write the area form as $\nu = \sigma(s, t) ds dt$. Then the symplectic vortex equation takes the following local form
\beq\label{eqn22}
\left\{ \begin{array}{ccc} 
\displaystyle \frac{\partial u }{\partial s} + {\mc X}_\phi(u) + J \Big( \frac{\partial u}{\partial t} + {\mc X}_\psi(u) \Big)  & = & 0;\\[0.2cm]
\displaystyle \frac{\partial \psi}{\partial s} - \frac{\partial \phi }{\partial t} + \big[ \phi, \psi \big] + \sigma \mu(u) & = & 0. \end{array} \right.
\eeq
When a trivialization and a holomorphic coordinate are manifest, we will use the notations ${\bf v} = (A, u)$ and ${\bf v} = (u, \phi, \psi)$ interchangeably.

\subsubsection*{The energy and energy density}

Given $\Sigma$ and an area form $\nu\in \Omega^2(\Sigma)$, the {\bf energy} of a pair ${\bf v} = (A, u)\in {\mc A}(P) \times {\mc S}(Y)$ is defined as
\beqn
E({\bf v}):= \frac{1}{2} \Big(  \big\| d_A u \big\|_{L^2(\Sigma)}^2 + \big\| F_A \big\|_{L^2(\Sigma)}^2 + \big\| \mu(u) \big\|_{L^2(\Sigma)}^2 \Big).
\eeqn
Here the $L^2$-norms are taken with respect to the metric on $\Sigma$ determined by $\nu$ and the complex structure of $\Sigma$, and the metric on ${\mf g}$ we used to identify ${\mf g}$ with ${\mf g}^*$. $E({\bf v})$ depends on $\nu$ but in this paper, the choice of $\nu$ will be clear from the context. Then if ${\bf v}$ is a solution to the symplectic vortex equation with respect to $\nu$, then the energy reduces to 
\beqn
E({\bf v}) = \big\| d_A u \big\|_{L^2(\Sigma)}^2 + \big\| \mu(u) \big\|_{L^2(\Sigma)}^2.
\eeqn

If $z = s + {\bm i} t$ is a local holomorphic coordinate and we write the area form as $\nu = \sigma(s, t) ds dt$, then the {\bf energy density} of ${\bf v}$ can be expressed as   
\beqn
{\bf e}:= \left( | \partial_s u + {\mc X}_\phi(u) |^2 +  \sigma | \mu(u) |^2 \right) ds dt .
\eeqn
If $\nu_0$ is another area form and $*: \Omega^2(\Sigma) \to \Omega^0(\Sigma)$ is the associated Hodge star operator, then the {\bf energy density function} of ${\bf v}$ with respect to $\nu_0$ is 
\beqn
e({\bf v}):= * {\bf e} \in \Omega^0(\Sigma).
\eeqn

\begin{defn}\label{defn6}
${\bf v} = (A, u)\in {\mc A}(P) \times {\mc S}(Y)$ is called {\bf bounded} if $E({\bf v}) < \infty$ and there exists a $G$-invariant compact subset $K \subset X$ such that $u(\Sigma) \subset P\times_G K$. 

A collection $({\bf v}^{(\alpha)})_{\alpha \in I} = (A^{(\alpha)}, u^{(\alpha)})_{\alpha \in I}$ (for possibly different area forms $\nu_\alpha$, on possibly different domains $P_\alpha \to \Sigma_\alpha$) is said to be {\bf uniformly bounded} if their energies are uniformly bounded and there exists a $G$-invariant compact subset $K\subset X$ such that $u^{(\alpha)}(\Sigma_\alpha) \subset P_\alpha \times_G K$.
\end{defn}

\subsubsection*{Boundary conditions}	

If $L$ is a $G$-Lagrangian of $X$, then we can impose the boundary condition
\beqn
u(\partial \Sigma) \subset P\times_G L.
\eeqn
To save notations, we just write it as $u(\partial\Sigma) \subset L$. More generally, if $\partial \Sigma$ has several components, we can impose the boundary condition for different $G$-Lagrangians on different boundary components. 

\subsubsection*{Gauge transformations}

A (smooth) gauge transformation on $P$ is a smooth map $g: P \to G$ satisfying that for all $h \in G$ and $p \in P$, $g(ph) = h^{-1} g(p) h$. All gauge transformations form an infinite dimensional Lie group, denoted by ${\mc G}(P)$. It acts on the space ${\mc A}(P) \times {\mc S} (Y)$ (on the right), which is denoted by $g^* (A, u) = (g^* A, g^* u)$. 
	
The equation \eqref{eqn21} is gauge invariant in the following sense. For $g\in {\mc G}(P)$, 
\begin{align*}
&\ \ov\partial_{g^* A} g^* u = ( g^{-1} )_* (\ov\partial_A u),& \ F_{g^* A} + \mu( g^* u) \nu = {\rm Ad}_g^{-1} ( F_A + \mu(u) \nu ).
\end{align*}
The notions of energy and energy density are also gauge invariant.

\subsubsection*{Regularity}

We also need to discuss objects which are not smooth. Let $p>1$ be a real number and $k\geq 0$ be an integer. Let $\Sigma$ be a Riemann surface and $P \to \Sigma$ be a smooth $G$-bundle. A connection on $P$ is said to be of class $W^{k, p}_{loc}$ if with respect any smooth trivializations of $P$ over a subset $U \subset \Sigma$, its connection form is of class $W^{k, p}_{loc}$. Let ${\mc A}^{k, p}_{loc}(P)$ be the space of all connections on $P$ of class $W^{k, p}_{loc}$. Similarly one can define ${\mc S}^{k, p}_{loc}(Y)$ and ${\mc G}^{k, p}_{loc}(P)$. For $k\geq 1$ and $p>2$, the equation \eqref{eqn21} and the boundary condition $u(\partial\Sigma) \subset P\times_G L$ make sense. One can also consider weak solutions but we do not need them in this paper.

It was proved as \cite[Theorem 3.1]{Cieliebak_Gaio_Mundet_Salamon_2002} that over a compact Riemann surface $\Sigma$ (without boundary), a solution to the symplectic vortex equation $(A, u)$ of regularity $W_{loc}^{1, p}$ for some $p>2$ can be gauge transformed via a gauge transformation of regularity $W_{loc}^{2, p}$ to a smooth vortex. For the reader's convenience, we provide a proof for its generalization to the bordered case.

From now on we assume that $\Sigma$ is {\bf of finite type}, i.e., there exists a sequence of precompact open subsets $U_l \subset \Sigma$ such that 1) each $\ov{U_l}$ is a smooth deformation retract of $\Sigma$; 2) each $\ov{U_l}$ has a smooth boundary; 3) for each $l$, $\ov{U_l} \subset U_{l+1}$ and there is a smooth retraction $\rho_l: \ov{U_{l+1}}\subset \ov{U_l}$. Notice that for $\Sigma = {\mb H}$, $U_l:= {\mb H} \cap B_{r_l}$ for growing radii $r_l$ does not satisfy this requirement since $\ov{U_l}$ has corners. But ${\mb H}$ is of finite type because one can modify $U_l$ near the corners so its closure has smooth boundary. Moreover, ${\mb C}$ or ${\mb H}$ with finitely many punctures are of finite type.

\begin{prop}\label{prop7}
Let $p>2$ and $\Sigma$ be a Riemann surface which is not necessarily compact and may have a nonempty smooth boundary. Let $P \to \Sigma$ be a smooth $G$-bundle and $(A, u) \in {\mc A}^{1, p}_{loc}(P) \times {\mc S}^{1, p}_{loc}(Y)$ be a solution to the symplectic vortex equation satisfying the boundary condition $u(\partial \Sigma)  \subset L$. Then there exists a gauge transformation $g\in {\mc G}^{2, p}_{loc}(P)$ such that $g^* (A, u)$ is smooth.
\end{prop}

\begin{proof}
Choose an exhausting sequence of open subsets ${U_l} \subset \Sigma$ as in the definition of finite type. If one can gauge transform $(A, u)$ to smooth ones over each $\ov{U_l}$, then one can patch the gauge transformations together to obtain a smooth vortex over $\Sigma$. So we can assume that $\Sigma$ is compact with smooth boundary. Then we can apply the local slice theorem (see \cite[Theorem F]{Wehrheim_Uhlenbeck}) to find a smooth connecton $A_0\in {\mc A}(P)$ and $g\in {\mc G}^{2,p}(P)$ such that 
\begin{align*}
&\ d_A^* ( (g^{-1})^* A_0 - A) = 0, &\ * ( (g^{-1})^* A_0 - A)|_{\partial \Sigma} = 0.
\end{align*}
This is equivalent to that 
\begin{align*}
&\ d_{A_0}^* ( g^* A - A_0) = 0, &\ * ( g^* A - A_0) |_{\partial \Sigma} = 0.
\end{align*}
Then $g^* (A, u)$ is a solution to an elliptic boundary value problem on $\Sigma$ with smooth coefficients, which implies that $g^* (A, u)$ is smooth.
\qed \end{proof}

\subsubsection*{Removal of boundary singularity}

For a solution $(A, u)$ to the symplectic vortex equation over a punctured surface, if the connections $A$ extends continuously over the punctures, then a standard result shows that one can gauge transform $(A, u)$ to a solution which extends smoothly over the punctures. A similar result holds in the open case.

\begin{lemma}\label{lemma6}
Let $L_-, L_+$ be two $G$-Lagrangians of $X$ such that $\bar{L}_-, \bar{L}_+$ intersect cleanly in $\bar{X}$. Let $(\Sigma, \partial \Sigma) = ({\mb D}^* \cap {\mb H}, {\mb D}^* \cap {\mb R})$ and $\partial_\pm \Sigma = {\mb D}^* \cap {\mb R}_\pm$. Suppose ${\bf v} = (u, \phi, \psi)$ is a bounded smooth solution to the symplectic vortex equation on $\Sigma$ with respect to $\nu \in \Omega^2({\mb D})$, then there exists a smooth gauge transformation $g: \Sigma \to G$ such that $g^* u$ extends continuously to ${\mb D}$. Moreover, if $L_- = L_+$, then one can choose $g$ such that $g^* {\bf v}$ extends smoothly.
\end{lemma}

The first part of this lemma is proved in Appendix \ref{appendixa5}. The second part can be proved by the same method as proving Proposition \ref{prop7}, using a weaker version of the local slice theorem (\cite[Theorem F']{Wehrheim_Uhlenbeck}).

\subsubsection*{The moduli space and topology}

From now on all $G$-bundles $P \to \Sigma$ are canonically trivialized, hence a connection $A$ on $P$ is identified with a 1-form $a \in \Omega^1(\Sigma)$, a section $u$ of $Y$ is identified with a map $u: \Sigma \to X$ and a gauge transformation $g$ on $P$ is identified with $g: \Sigma \to G$. If $\nu \in \Omega^2(\Sigma)$ is an area form and $L \subset X$ is a $G$-Lagrangian, denote by $\wt{\mc M}(\Sigma, \nu; X, L)$ be the set of all smooth solutions ${\bf v}$ to the symplectic vortex equation over $\Sigma$ with respect to the area form $\nu$ and satisfying the boundary conditoin $u(\partial \Sigma) \subset L$. When $\partial \Sigma = \emptyset	$, we abbreviate it by $\wt{\mc M}(\Sigma, \nu; X)$. When $\nu$ is ``standard'' we will also omit $\nu$. We also consider the case $\nu = 0$. 

Suppose $\Sigma_k \subset \Sigma$ is an exhausting sequence of open subsets and $f_k \in C^\infty(\Sigma_k)$. We say that $f_k$ {\bf converges u.c.s.} on $\Sigma$ to $f \in C^\infty(\Sigma)$ (the terminology is from \cite{McDuff_Salamon_2004}) if for any compact subset $K \subset \Sigma$, $f_k|_K$ converges to $f|_K$ in $C^\infty(K)$. This notion induces the definition of that a sequence ${\bf v}_k \in \Omega^1(\Sigma_k) \times C^\infty(\Sigma_k, \partial \Sigma_k; X, L)$ converges u.c.s. to ${\bf v}\in \Omega^1(\Sigma) \times C^\infty(\Sigma, \partial \Sigma; X, L)$, and in particular, a topology on $\wt{\mc M}(\Sigma, \nu; X, L)$.  

 The following lemma shows that up to gauge transformation, weak convergence in $W^{1, p}_{loc}$ implies convergence u.c.s..

\begin{lemma}\label{lemma7}
Let $\Sigma$ be a Riemann surface possibly having a nonempty smooth boundary and $\nu\in \Omega^2(\Sigma)$. Let $\Sigma_k \subset \Sigma$ be an exhausting sequence of open subsets and $\nu_k \in \Omega^2(\Sigma_k)$ converges u.c.s. to $\nu$. Given ${\bf v}_k = (A_k, u_k) \in \wt{\mc M}(\Sigma_k, \nu_k; X, L)$, ${\bf v} = (A, u) \in \wt{\mc M}(\Sigma, \nu; X, L)$. Suppose $p>2$ and for all compact subsets $K \subset \Sigma$, ${\bf v}_k|_K$ converges weakly in $W^{1, p}(K)$ to ${\bf v}|_K$. Then there exists a sequence of smooth gauge tranformations $g_k: \Sigma_k \to G$ such that $(g_k)^* {\bf v}_k$ converges u.c.s. to ${\bf v}$.
\end{lemma}

\begin{proof}
Let $U_l \subset \Sigma$ be the sequence of subsets in the definition of finite type. For each $k\geq 1$, let $l_k$ be the greatest number such that $\ov{U_{l_k}} \subset \Sigma_k$. Then it suffices to prove for the case that $\Sigma_k = U_{l_k}$ because gauge transformations on $\ov{U_{l_k}}$ can be extended to $\Sigma$ and then restricted to $\ov{\Sigma_k}$. In particular we may assume that $\ov{\Sigma_k}$ has smooth boundary. Without loss of generality we can also assume that $\Sigma_k \subsetneq \Sigma_{k+1}$ and there are smooth retractions $\rho_k: \ov{\Sigma_{k+1}} \to \ov{\Sigma_k}$. Moreover, by shrinking $\Sigma_k$ a little, we may assume that ${\bf v}_k$ is defined over $\ov{\Sigma_k}$.

For each $l$ and $k \geq l$, by the local slice theorem (\cite[Theorem F]{Wehrheim_Uhlenbeck}), there exists $g_{k; l} \in W^{2, p}(\ov{\Sigma_l}, G)$  such that $(g_{k;l})^* (A_k|_{\ov{\Sigma_l}})$ is in Coulomb gauge relative to $A|_{\ov{\Sigma_l}}$. By elliptic regularity, we know that $g_{k;l}$ is smooth and $(g_{k;l})^* {\bf v}_k|_{\ov{\Sigma_l}}$ converges in $C^\infty(\ov{\Sigma_l})$ to ${\bf v}|_{\ov{\Sigma_l}}$. Since $A_k|_{\ov{\Sigma_l}}$ converges to $A|_{\ov{\Sigma_l}}$ in $L^p$, we may choose $g_{k;l}$ in such a way that $g_{k;l}$ converges in $W^{1, p} (\ov{\Sigma_l})$ to the identity.

For each $l$ and $k \geq l+1$, denote $h_{k; l}:= (g_{k; l+1}|_{\ov{\Sigma_l}})^{-1} g_{k;l}: \ov{\Sigma_l} \to G$, which is smooth and converges to the identity in $W^{1, p}(\ov{\Sigma_l})$. By the Coulomb gauge condition, we know that 
\beqn
d_A^* ( g_{k; l+1}^* A_k - A) = d_A^* (g_{k; l}^* A_k - A) = 0.
\eeqn
Then by the interior estimate, $h_{k; l}$ converges to the identity in $C^\infty(\ov{\Sigma_{l-1}})$.

On the other hand, for $k\geq l\geq 2$, $\rho_{k;l}:= \rho_{k-1} \circ \cdots \circ \rho_{l-1}: \ov{\Sigma_k} \to \ov{\Sigma_{l-1}}$ is a smooth retraction. Define $h_{k;l}^* = h_{k;l}|_{\ov{\Sigma_{l-1}}} \circ \rho_{k;l}$, which converges to the identity in $C^\infty(\ov{\Sigma_{l'}})$ for any $l'\geq l-1$. We define a sequence $g_k: \ov{\Sigma_k} \to G$ by 
\beqn
g_k = g_{k, k} \prod_{t=2}^{k-1} h_{k; t}^*.
\eeqn
For any $l< k$, one has
\beqn
g_k|_{\ov{\Sigma_l}} = g_{k, k}|_{\ov{\Sigma_l}} \prod_{t=l+1 }^{k-1} \Big((g_{k; t+1})^{-1} g_{k;t}\Big)|_{\ov{\Sigma_l}} \prod_{t=2}^{l} h_{k;t}^*|_{\ov{\Sigma_l}} = \left( g_{k;l+1}\prod_{t=2}^{l} h_{k;t}^*\right)|_{\ov{\Sigma_l}}.
\eeqn
Since on each $\ov{\Sigma_l}$, $(g_{k;l+1})^* {\bf v}_k|_{\ov{\Sigma_l}}$ converges in $C^\infty(\ov{\Sigma_l})$ to ${\bf v}$ and every $h_{k; t}^*$ (for $t =2, \ldots, l$) converges to the identity in $C^\infty(\ov{\Sigma_{l}})$ as $k \to \infty$, it follows that $(g_k)^* {\bf v}_k|_{\ov{\Sigma_l}}$ converges in $C^\infty(\ov{\Sigma_{l}})$ to ${\bf v}$. Hence the lemma is proved. \qed\end{proof}

\subsection{Symplectic vortex equation over the unit disk}\label{subsection22}

Let ${\mb D} \subset {\mb C}$ be the unit disk and let $\nu_0$ be the standard area form on ${\mb D}$. Let $\lambda\geq 1$ be a real number. In this paper we are interested in the symplectic vortex equation over ${\mb D}$ with respect to $\lambda^2 \nu_0$, and the compactness as $\lambda \to \infty$. Firstly we show the uniform boundedness for different $\lambda$.

\begin{lemma}\label{lemma8}
Let $\lambda \geq 1$ and $c_0 \geq \sup\{ c_X,\ f_X(x)+1 \ |\ x \in L \}$, where $c_X$ and $f_X$ are as in Hypothesis \ref{hyp4}. Then for ${\bf v} = (A, u)\in  \wt{\mc M}({\mb D}, \lambda^2 \nu_0; X, L)$, one has	
\beq\label{eqn23}
u ( {\mb D} ) \subset X_{c_0}:= \{ x\in X\ |\  f_X(x) \leq c_0 \}.
\eeq
\end{lemma}

\begin{proof}
Suppose \eqref{eqn23} is not true. Then since $L \subset \mu^{-1}(0)$, there exists $z\in {\rm Int} {\mb D}$ such that $f_X(u(z)) = \sup f_X(u) > c_0$. Then one can show that $f_X(u)$ is subharmonic near $z$ as \cite[Lemma 2.7]{Cieliebak_Gaio_Mundet_Salamon_2002}. The maximum principle implies that $f_X(u) \equiv \sup f_X(u)$, which contradicts the fact that $\mu(u)|_{\partial {\mb D}} \equiv 0$.
\qed \end{proof}

We can also consider boundary marked points. Let $\ud k \geq 0$ be an integer and use ${\bf w}$ to denote a set of $\ud k$ marked points $\{ w_1, \ldots, w_{\ud k}\} \subset \partial {\mb D}$, always listed in the counterclockwise order. Denote by $\wt{\mc M}_{\ud k}^\lambda(L)$ the space of $({\bf v}, {\bf w})$ where ${\bf v}\in \wt{\mc M}(X, L; {\mb D}, \lambda^2 \nu_0)$ and ${\bf w}$ is a set of $\ud k$ boundary marked points.

\subsection{Affine vortices}

Affine vortices are bounded solutions to \eqref{eqn21} over $\Sigma = {\mb C}$ or $\Sigma = {\mb H}$ with respect to the standard area form $\nu_0 = ds dt$, where $z = s + {\bm i} t $ is the standard coordinate. The equation is equivalent to \eqref{eqn22} for $\sigma = 1$, which reads
\beq\label{eqn24}
\left\{ \begin{array}{ccc}\displaystyle  \frac{\partial u }{\partial s} + {\mc X}_\phi (u) + J \Big( \frac{\partial u }{\partial t} + {\mc X}_\psi (u) \Big) & = & 0,\\[0.2cm]
       \displaystyle                     \frac{\partial \psi}{\partial s} - \frac{\partial \phi}{\partial t} + \big[ \phi, \psi \big] + \mu(u) & = & 0.
\end{array} \right.
\eeq
A bounded solution ${\bf v} = (u, \phi, \psi)$ to \eqref{eqn24} over ${\mb C}$ is called a ${\mb C}$-vortex. A bounded solution ${\bf v} = (u, \phi, \psi)$ to \eqref{eqn24} subject to $u(\partial {\mb H}) \subset L$ is called an ${\mb H}$-vortex. In both cases, the energy density function of ${\bf v}$ means its energy density function with respect to the standard area form $\nu_0$. 

We first recall basic properties of ${\mb C}$-vortices proved in \cite{Gaio_Salamon_2005} and Ziltener's papers (\cite{Ziltener_Decay}, \cite{Ziltener_book}). These results will be extended to ${\mb H}$-vortices in this paper.

The first important property is the behavior near infinity. 	We have
\begin{prop}\label{prop9}\cite[Corollary 4]{Ziltener_Decay} There exists $\updelta>0$\footnote{Indeed, the constant $\updelta$ can be taken to be any number smaller than 2. However, in this paper, we do not need the optimal result whose proof is more delicate.} such that for every ${\mb C}$-vortex ${\bf v}$ with energy density function $e(z)$, one has
\beqn
\limsup_{z \to \infty} \Big( |z|^{2 + \updelta} e(z) \Big) < +\infty.
\eeqn
\end{prop}

In suitable gauge, ${\mb C}$-vortices have nice asymptotic behavior.
	
\begin{prop}\label{prop10}\cite[Proposition 11.1]{Gaio_Salamon_2005} If ${\bf v} = (A, u)$ is a ${\mb C}$-vortex and $A= d + \xi(z) d\theta$ for $|z|$ large where $\theta$ is the angular coordinate of ${\mb C}$. Then there exists a $W^{1, 2}$-map $x: S^1 \to \mu^{-1}(0)$ and an $L^2$-map $\eta: S^1 \to {\mf g}$ such that $x'(\theta) + {\mc X}_{\eta(\theta)}(x(\theta)) = 0$ and 
\beqn
\lim_{ r \to \infty} \sup_{0 \leq \theta \leq 2\pi} d ( u( re^{i\theta}), x(\theta) ) = 0,\ \lim_{r\to \infty} \int_0^{2\pi} | \xi  ( r e^{i \theta}) - \eta(\theta) |^2 d \theta = 0.
\eeqn
\end{prop}
We see that $x(\theta)$ stays within the same $G$-orbit in $\mu^{-1}(0)$. Therefore it follows that every ${\mb C}$-vortex ${\bf v}$ has a well-defined evaluation at infinity, denoted by 
\beq\label{eqn25}
\ev_\infty ({\bf v}) = [x(\theta)] \in \bar{X}.
\eeq

Lastly we recall the energy quantization result about ${\mb C}$-vortices.

\begin{prop}\label{prop11}\cite[Lemma D.1]{Ziltener_thesis}\cite{Ziltener_book}
There exists a constant $\epsilon_X >0$ such that for any ${\mb C}$-vortex ${\bf v}$ with nonzero energy, we have $E({\bf v}) \geq \epsilon_X$.
\end{prop}

Now we turn to ${\mb H}$-vortices. Firstly, we need a removal of singularity at the infinity of ${\mb H}$. In order to be able to generalize to the case of multiple Lagrangians, consider \eqref{eqn24} over $\arc:= {\mb H} \smallsetminus B_R$. Let $L_-, L_+$ be two $G$-Lagrangians whose quotients in $\bar{X}$ intersecting cleanly. We impose the boundary condition
\beq\label{eqn26}
u ( \arc \cap {\mb R}_\pm) \subset L_\pm.
\eeq
In Appendix \ref{appendixa} we prove the following result. 
\begin{thm}\label{thm12}\hfill
\begin{enumerate}
\item If ${\bf v} = (u, \phi, \psi)$ is a bounded solution to \eqref{eqn24} on $\arc$ with boundary condition \eqref{eqn26}, then there exists a $G$-orbit $O \subset L_-\cap L_+$ such that
\beqn
\lim_{ z \to \infty} d(u(z), O ) = 0.
\eeqn

\item There exists a constant ${\bm \updelta} = {\bm \updelta}(L_-, L_+)>0$ such that, for any bounded solution $(u, \phi, \psi)$ to \eqref{eqn24} with boundary condition \eqref{eqn26}, we have
\beq\label{eqn27}
\limsup_{z \to +\infty} \Big( |z|^{2+{\bm \updelta}} e(z) \Big) < +\infty.
\eeq
\end{enumerate}
\end{thm}

Therefore one can define the evaluation of an ${\mb H}$-vortex at $\infty$ as a point in $\bar{L}$. Moreover, using the same argument as proving Lemma \ref{lemma8}, we have
\begin{cor}\label{cor13}
For every ${\mb H}$-vortex ${\bf v} = (u, \phi, \psi)$, $u( {\mb H}) \subset X_{c_0}$, where $X_{c_0}$ is defined in \eqref{eqn23}.
\end{cor}

Another important ingredient is the energy quantization for ${\mb H}$-vortices. We prove the following theorem in Appendix \ref{appendixa4}.
\begin{thm}\label{thm14}
There exists a constant $\epsilon_{X, L}>0$ such that for any ${\mb H}$-vortex ${\bf v}$ with nonzero energy, we have $E({\bf v}) \geq \epsilon_{X, L}$.
\end{thm}

\subsection{Compactness}

We recall the basic compactness theorem of vortex equation over a bordered Riemann surface with Lagrangian boundary condition. Let $\Sigma$ be a Riemann surface of finite type possibly having a smooth boundary $\partial \Sigma$; let $\Sigma_i$ be an exhausting sequence of open subsets of $\Sigma$, and $\nu_i$ be a sequence of area forms on $\Sigma_i$. Suppose that $\nu_i$ converges u.c.s. to $\nu\in \Omega^2(\Sigma)$ where $\nu$ is either a smooth area form on $\Sigma$ or $\nu = 0$. Firstly we have the following ``compactness modulo bubbling'' result.

\begin{thm}\label{thm15}
Let ${\bf v}^{(i)} = (A^{(i)}, u^{(i)})  \in \wt{\mc M}(\Sigma_i, \nu_i; X, L)$ be a uniformly bounded sequence. Then there exists a subsequence (still indexed by $i$), a finite subset $Z \subset \Sigma$, ${\bf v} \in \wt{\mc M}(\Sigma, \nu; X, L)$ and a sequence of smooth gauge transformations $g^{(i)}: \Sigma_i \to G$, such that $( g^{(i)} )^* {\bf v}^{(i)}$ converges u.c.s. to ${\bf v}$ on $\Sigma \smallsetminus Z$.

Moreover, for any $z\in Z$, we have
\beqn
\lim_{r \to 0} \lim_{i \to \infty} E\big( {\bf v}^{(i)}; B_r(z) \cap \Sigma_i \big) \geq \epsilon(X, L),
\eeqn
where $\epsilon(X, L)$ is less than the minimum of the energy of any nontrivial holomorphic sphere in $X$ or any nontrivial holomorphic disk in $(X, L)$.
\end{thm}

It follows essentially from the energy quantization property of holomorphic disks or sphers in $X$. To construct the limiting smooth vortex ${\bf v}$ and prove the convergence u.c.s., one needs the regularity results Proposition \ref{prop7}, Lemma \ref{lemma7} and the removal of singularity result Lemma \ref{lemma6}. The detail is left to the reader.

Further, we can find a bubble tree at each $z \in Z$ as part of the geometric limit of the sequence ${\bf v}^{(i)}$. Since the bubbling is only a local phenomenon, we assume that $\Sigma$ is an open subset of ${\mb H}$ of finite type. We also assume $Z$ consists of a single element $z$. The description of bubble trees are discussed in Appendix \ref{appendixb}. Then we have
\begin{thm}(cf. \cite{Ott_compactness}, \cite{Guangbo_compactness})\label{thm16}
Assume $\Sigma \subset {\mb H}$ and $\Sigma_i$, $\nu_i$, $\nu$, ${\bf v}^{(i)}$ are the same as in Theorem \ref{thm15}. Choose $z \in Z$. Then there exists a subsequence (still indexed by $i$) and: 

1) if $z \in \partial {\mb H}$, a stable holomorphic disk modelled on a based branch ${\mc T}$
\beqn
{\mc D}:= \big( (u_\alpha)_{v_\alpha \in V({\mc T})}, (z_{\alpha \beta})_{e_{\alpha \beta}\in E({\mc T})} \big);
\eeqn

2) if $z \in {\rm Int} {\mb H}$, a stable holomorphic sphere modelled on a branch ${\mc T}$
\beqn
{\mc S}:= \big( (u_\alpha)_{v_\alpha \in V({\mc T})}, (z_{\alpha \beta})_{e_{\alpha \beta}\in E({\mc T})} \big).
\eeqn
(For the precise meaning see Appendix \ref{appendixb2}.) In either case, they satisfy the following conditions.
\begin{enumerate}
\item For each $v_\alpha \in V({\mc T})$, there exist a sequence of smooth gauge transformations $g_\alpha^{(i)}: \Sigma_\alpha \to G$ and a sequence of M\"obius transformations $\varphi_\alpha^{(i)}: \Sigma_\alpha \to {\mb H}$ such that $( g_\alpha^{(i)} )^* ( \varphi_\alpha^{(i)} )^* {\bf v}^{(i)}$ converges u.c.s. to $(0, u_\alpha)$ on $\Sigma_\alpha \smallsetminus Z_\alpha$;

\item For each $\alpha$, $\varphi_\alpha^{(i)}$ converges u.c.s. on $\Sigma_\alpha$ to the constant $z$; for any $e_{\alpha \beta}\in E({\mc T})$, $( \varphi_\beta^{(i)} )^{-1} \circ \varphi_\alpha^{(i)}$ converges u.c.s. on $\Sigma_\alpha$ to the constant $z_{\alpha \beta} \in \Sigma_\beta$. 

\item There is no energy loss, i.e.,
\beqn
\lim_{r\to \infty} \lim_{i \to \infty} E( {\bf v}^{(i)}; B_r(z) \cap {\mb H} ) = \sum_{v_\alpha \in V({\mc T}) } E(u_\alpha).
\eeqn
\end{enumerate}
\end{thm}

The proof of this theorem has been given by Ott \cite{Ott_compactness} in the case that $z \in {\rm Int} {\mb H}$ and by the second named author \cite{Guangbo_compactness} in the case $G = S^1$. The detail of a general proof is left to the reader because the proof has no essential difference from the proof of the standard Gromov compactness for holomorphic disks or spheres, and because our main concern is the compactness with respect to the adiabatic limit.

\begin{rem}
Besides \cite{Ott_compactness} and \cite{Guangbo_compactness}, there are other works treating compactness of vortex equations in various settings, such as \cite{Mundet_2003}, \cite{Cieliebak_Gaio_Mundet_Salamon_2002}, \cite{Mundet_Tian_2009}, \cite{Venugopalan_quasi}.
\end{rem}

\section{Compactness Modulo Bubbling}\label{section3}

In this section we start to consider adiabatic limit of vortices. Given an area form $\nu \in \Omega^2(\Sigma)$ and $\lambda > 0$, the energy of ${\bf v}  =(A, u) \in \wt{\mc M}(\Sigma, \lambda^2 \nu; X, L)$ is
\beqn
E ({\bf v}) = \frac{1}{2} \Big( \big\| d_A u \big\|_{L^2(\Sigma)}^2 + \lambda^{-2} \big\| F_A \big\|_{L^2(\Sigma)}^2 + \lambda^2 \big\| \mu(u) \big\|_{L^2(\Sigma)}^2 \Big).
\eeqn
Here the $L^2$ norms are defined with respect to the fixed area form $\nu$. If we impose a uniform energy bound, then for large $\lambda$, $\mu(u)$ will be small in the $L^2$-sense. If $\mu(u)$ is uniformly small, then $u$ will be close to a holomorphic curve in $\bar{X}$ (cf. Definition \ref{defn17}). If $\mu(u)$ is not uniformly small, then it corresponds to energy concentration and bubbling. According to different rates of the energy concentration compared to the speed of $\lambda \to +\infty$, different types of objects bubble off. The classification of bubbles is given in Subsection \ref{subsection31}. Among six types of bubbles, the type S2 (see Definition \ref{defn21}) is novel where ${\mb H}$-vortices bubble off. We treat this type of bubbles in more details.

This section corresponds to the ``hard rescaling'' in the bubbling analysis (see similar terminology used in \cite{McDuff_Salamon_2004} and \cite{Frauenfelder_disk}), where we can use the energy quantization property of various bubbles to construct a finite subset of the domain where energy concentrates. 

\subsection{Convergence in the adiabatic limit}\label{subsection31}

Let $\Sigma$ be a Riemann surface, possibly having a nonempty smooth boundary $\partial \Sigma$. Let $\nu$ be an area form on $\Sigma$. Let $\lambda_i$ be a sequence of positive numbers diverging to infinity. Let $(\Sigma_i, \partial\Sigma_i)$ be a sequence of exhausting open subsets of $(\Sigma, \partial \Sigma)$. Let $\nu_i$ be a smooth area form on $\Sigma_i$, such that for any compact subset $K \subset \Sigma$, the sequence $\nu_i|_K$ converges uniformly (with all derivatives) to $\nu|_K$. We fix these data in this subsection.

Take a sequence $\lambda_i \to +\infty$ and ${\bf v}^{(i)} \in \wt{\mc M}( \Sigma_i, \lambda_i^2 \nu_i; X, L)$. Let ${\bf e}_i$ be the energy densities of ${\bf v}^{(i)}$. In this section, the energy density function of ${\bf v}^{(i)}$ means its energy density with respect to $\nu$.

To describe the convergence towards a holomorphic curve in $\bar{X}$, we introduce the following notation. There exists a $G$-invariant neighborhood $U_X$ of $\mu^{-1}(0)$ on which the $G$-action is free. Moreover, there exists a smooth projection $\pi_{\bar{X}}: U_X \to \bar{X}$ whose restriction to $\mu^{-1}(0)$ coincides with the projection $\mu^{-1}(0) \to \bar{X}$. The following notion is independent of the choice of such $\pi_{\bar{X}}$. 

\begin{defn}\label{defn17} Suppose $\bar{u}: \Sigma \to \bar{X}$ is a $\bar{J}$-holomorphic map with $\bar{u}(\partial \Sigma) \subset \bar{L}$. We say that {\bf the sequence ${\bf v}^{(i)}	$ converge to $\bar{u}$} if
\begin{enumerate}
\item $\mu(u^{(i)}):\Sigma_i \to {\mf g}^*$ converges to zero uniformly on any compact subset of $\Sigma$;

\item For any compact subset $K\subset \Sigma$, for large $i$, the sequence of continuous maps $\pi_{\bar{X}}\circ u^{(i)}: K \to \bar{X}$ converges uniformly to the map $\bar{u}|_K$.
\end{enumerate}
\end{defn}

There is a different point of view on the above notion of convergence (see \cite[Section 2]{Gaio_Salamon_2005}). Consider the following equation for a pair ${\bf v}= (A, u)$
\beq\label{eqn31}
\ov\partial_A u = 0,\ \mu(u) = 0. 
\eeq
The group of gauge transformations on $\Sigma$ acts on the space of solutions to \eqref{eqn31}, and the set of gauge equivalence classes of solutions to \eqref{eqn31} is in one-to-one correspondence to the set of holomorphic maps $\bar{u}: \Sigma \to \bar{X}$. Then we can rewrite Definition \ref{defn17} as a convergence of sequence of ${\bf v}^{(i)}$ to a solution to \eqref{eqn31} modulo gauge transformations. We won't take this perspective but we remark that this viewpoint implies the following important fact.
\begin{prop}\label{prop18}
Suppose ${\bf v}^{(i)}$ converges to a $\bar{J}$-holomorphic map $\bar{u}: \Sigma \to \bar{X}$ with $\bar{u}(\partial \Sigma) \subset \bar{L}$. Let $e(\bar{u})$ be the energy density of $\bar{u}$ with respect to $\nu$. Then
\beqn
\lim_{i \to \infty} e_i = e(\bar{u})
\eeqn
uniformly on compact subsets of $\Sigma$.
\end{prop}

Now we have the adiabatic limit convergence theorem provided that the energy density doesn't blow up, which is a simple extension of the known results (see \cite{Gaio_Salamon_2005}) to the case where $\Sigma$ may have boundary.

\begin{thm}\label{thm19}
Suppose that the sequence ${\bf v}^{(i)}$ is uniformly bounded and for any compact subset $K \subset \Sigma$, we have
\beqn
\limsup_{i \to \infty} \sup_{z \in K} e_i(z) < \infty.
\eeqn
Then there exists a subsequence (still indexed by $i$) of ${\bf v}^{(i)}$ which converges to a $\bar{J}$-holomorphic map $\bar{u}: (\Sigma, \partial \Sigma) \to (\bar{X}, \bar{L})$. 
\end{thm}

The following theorem is proved at the very end of this section.

\begin{thm}\label{thm20}
Let $\Sigma$, $\nu$, $\partial \Sigma$, $\Sigma_i$, $\nu_i$, $\lambda_i$ be as above. There exist a constant $\epsilon>0$, depending on $X$, $L$, $J$ such that for any uniformly bounded sequence ${\bf v}^{(i)}\in \wt{\mc M}(\Sigma_i, \lambda_i^2 \nu_i; X, L)$, there exist a subsequence (still indexed by $i$), and a finite subset $Z\subset \Sigma$ satisfying the following conditions. 
\begin{enumerate}
\item For each $z \in Z$, we have
\beqn
\lim_{r \to 0} \limsup_{i \to \infty} E ( {\bf v}^{(i)}; B_r(z) \cap \Sigma ) \geq \epsilon
\eeqn
and for any compact subset $K \subset \Sigma \smallsetminus Z$, we have
\beqn
\limsup_{i \to \infty} \sup_{z \in K} e_i (z) < \infty.
\eeqn

\item The sequence ${\bf v}^{(i)}|_{\Sigma_i \smallsetminus Z}$ converges (modulo gauge) to a $\bar{J}$-holomorphic map $\bar{u}: ( \Sigma \smallsetminus Z, \partial (\Sigma \smallsetminus Z) ) \to (\bar{X}, \bar{L})$ in the sense of Definition \ref{defn17}.
\end{enumerate}
\end{thm}

\subsection{Bubbling zoology in adiabatic limit}\label{subsection32}

In the presence of Theorem \ref{thm19}, to prove Theorem \ref{thm20} we have to consider the case when $e_i$ blows up at some point, i.e., when bubbling happens. There are different bubbling types depending on the relative rate of energy concentration compared to the speed of the blow-up of the area form. In the closed case the classification of bubbles has been discussed in \cite[Section 12]{Gaio_Salamon_2005}, and \cite{Woodward_toric} contains a discussion about the open case. Here we summarize their classification of bubbles and provide proofs for the open situation. 

We still use our notations in Subsection \ref{subsection31}. Suppose $z_i \in \Sigma_i$ converges to $	z^* \in \Sigma$. Let $\Sigma^* \subset \Sigma$ be a precompact open subset containing $z^*$ and
\beq\label{eqn33}
c_i^2 := e_i(z_i) \geq \frac{1}{2} \sup_{\Sigma^*} e_i \to \infty.
\eeq
We denote $\gamma_i:= d(z_i, \partial \Sigma_i)$, where $d$ is the distance function induced from the metric determined by the area form $\nu$. (Indeed the remaining doesn't depend on the choice of any area form.)

\begin{defn}\label{defn21}
We say that the energy (of the sequence) blows up at $z^*$ in certain type (F1, F2, S1, S2, R1, R2) according to the following table.
\begin{center}
\vspace{0.1cm}
\begin{tabular}{|c|c|c|} 
\hline
Type & $\displaystyle \lim_{i \to \infty} ( c_i /\lambda_i )$ &  $ \displaystyle \lim_{i \to \infty} \gamma_i c_i $\\[0.15cm]
\hline
F1 & $=\infty$ & $= \infty$\\
\hline
F2 & $=\infty$ & $< \infty$\\
\hline
S1 & $\in (0, \infty)$ & $= \infty$\\
\hline
S2 & $\in (0, \infty)$ & $ < \infty$\\
\hline
R1 & $= 0$ & $ = \infty$\\  
\hline
R2 & $=0$ & $< \infty$\\
\hline
\end{tabular}
\vspace{0.1cm}
\end{center}
\end{defn}

Take a holomorphic coordinate $z = s + {\bm i} t$ centered at $z^*$ so that the area forms $\nu_i$ can be written as $\nu_i = \sigma_i(z) ds dt$, and we assume $\nu = \sigma ds dt$ with $\sigma(0) = 1$. The coordinate of $z_i$ is still denoted by $z_i$ as a point in ${\mb C}$ or ${\mb H}$. 

We have the following lemmata regarding different types of blowing up.
\begin{lemma}\label{lemma22}
If the energy blows up at $z^*$ in type F1, define $\varphi_i (w) = z_i + c_i^{-1} w$ for $w \in B ( \sqrt{c_i} ) \subset {\mb C}$. Then there exists a subsequence (still indexed by $i$) and a sequence of smooth gauge transformations $g_i: {\mb C} \to G$ such that the sequence $(g_i)^* (\varphi_i)^* {\bf v}^{(i)}$ converges u.c.s. on ${\mb C}$ to $(0, u)$, where $u: {\mb C} \to X$ is a $J$-holomorphic map with finite positive energy.
\end{lemma}
\begin{proof}
Standard result in analysis of symplectic vortex equation (see \cite[Section 12]{Gaio_Salamon_2005}).
\qed \end{proof}

\begin{lemma}\label{lemma23}
If the energy blows up in type F2 at $z^* \in \partial \Sigma$, define $\varphi_i(w) = {\rm Re} z_i + c_i^{-1} w$ for $w \in B ( {\sqrt{c_i}} ) \cap {\mb H}$. Then there exists a subsequence (still indexed by $i$), and a sequence of smooth gauge transformations $g_i: {\mb H} \to G$ such that the sequence $(g_i)^* (\varphi_i)^* {\bf v}^{(i)}$ converges u.c.s. on ${\mb H}$ to $(0,u)$, where $u: ({\mb H}, {\mb R}) \to (X, L)$ is a $J$-holomorphic map with finite positive energy.
\end{lemma}

\begin{proof}
This is a simple extension of Lemma \ref{lemma22} to the boundary case.
\qed \end{proof}

\begin{lemma}\label{lemma24} If the energy blows up at $z^*$ in type R1, define $\varphi_i(w) = z_i + c_i^{-1}w$ for $w \in  B ( \sqrt{c_i} ) \subset {\mb C}$. Then there exists a subsequence (still indexed by $i$) such that $(\varphi_i)^* {\bf v}^{(i)}$ converges to a $\bar{J}$-holomorphic map $\bar{u}: {\mb C} \to \bar{X}$ with finite positive energy in the sense of Definition \ref{defn17}.
\end{lemma}

\begin{proof} See \cite[Section 12]{Gaio_Salamon_2005}.
\qed \end{proof}

\begin{lemma}\label{lemma25} If the energy blows up at $z^* \in \partial \Sigma$ in type R2, define $\varphi_i(w) = {\rm Re} z_i + c_i^{-1}w$ for $w \in B ( {\sqrt{c_i}} ) \cap {\mb H}$. Then there exists a subsequence (still indexed by $i$) such that $(\varphi_i)^* {\bf v}^{(i)}$ converges to a $\bar{J}$-holomorphic map $\bar{u}:  ( {\mb H}, {\mb R} )  \to ( \bar{X}, \bar{L} )$ with finite positive energy in the sense of Definition \ref{defn17}.
\end{lemma}

\begin{proof}
This is an extension of the above lemma. After the rescaling we can use Theorem \ref{thm19} to prove it. Details are left to the reader.
\qed \end{proof}

\begin{lemma}\label{lemma26}
If the energy blows up at $z^*$ in type S1, define $\varphi_i(w) = z_i + \lambda_i^{-1}w$ for $w \in B ( {\sqrt{\lambda_i}} )$ (note that $\varphi_i^* \lambda_i^2 \nu_i$ converges u.c.s. to the standard area form on ${\mb C}$). Then there exists a subsequence (still indexed by $i$), and a sequence of smooth gauge transformations $g_i: {\mb C} \to G$ such that $(g_i)^* (\varphi_i)^* {\bf v}^{(i)}$ converges u.c.s. 	to some ${\bf v} \in \wt{\mc M}({\mb C}; X)$ with positive energy. 
\end{lemma}

\begin{proof}
See Step 5 of the proof of \cite[Proposition 12.3]{Gaio_Salamon_2005}.
\qed \end{proof}

\subsubsection*{Bubbling of ${\mb H}$-vortices}

Now we focus on the case of type S2 blow up which hasn't been considered in the literature before. We prove the following lemma in detail. 
\begin{lemma}\label{lemma27}
If the energy blows up at $z^*\in \partial \Sigma$ in type S2, define $\varphi_i(w) = {\rm Re} z_i + \lambda_i^{-1}w$ for $w \in B ( {\sqrt{\lambda_i}} )\cap {\mb H}$ (note that $\varphi_i^* \lambda_i^2 \nu_i$ converges u.c.s. to the standard area form on ${\mb H}$). Then there exists a subsequence (still indexed by $i$), and a sequence of smooth gauge transformations $g_i: {\mb H} \to G$ such that $(g_i)^* (\varphi_i)^* {\bf v}^{(i)}$ converges u.c.s. to some ${\bf v} \in \wt{\mc M}({\mb H}; X, L)$ with positive energy.
\end{lemma}

\begin{proof}
Consider $( B^{(i)}, v^{(i)} ):= (\varphi_i)^* {\bf v}^{(i)}$ restricted to the region $\Omega_i:= B ( {\sqrt{\lambda_i}} ) \cap {\mb H}$. Then $\Omega_i$ exhausts ${\mb H}$. Moreover, by \eqref{eqn33}, with respect to the flat metric,
\beqn
\sup_{i, \Omega_i} \big| F_{B^{(i)}} \big| = \lambda_i^{-2} \sup_{i, \varphi_i(\Omega_i)} \big| F_{A^{(i)}} \big| < +\infty.
\eeqn
Then we use the Uhlenbeck compactness theorem for the region ${\mb H}$ and the sequence $B^{(i)}$ (see \cite[Theorem A']{Wehrheim_Uhlenbeck}). Notice that $B^{(i)}$ is not defined on the whole ${\mb H}$; we extend $B^{(i)}$ arbitrarily to ${\mb H}$ then the sequence satisfies the Hypothesis of \cite[Theorem A']{Wehrheim_Uhlenbeck}. It implies that, for a chosen $p>2$,  there exist a subsequence (still indexed by $i$) and a sequence of gauge transformations $g_i\in {\mc G}^{2, p}_{loc}({\mb H})$, such that $(g_i)^* B^{(i)}$ converges weakly in $W^{1, p}(K)$, for any compact region $K \subset {\mb H}$. Then $\hat{\bf v}^{(i)}:= (g_i)^* (\varphi_i)^* {\bf v}_i$ solves the symplectic vortex equation with respect to $(\varphi_i)^* \nu_i$ and for any compact subset $K$, the energy densities of $\hat{\bf v}^{(i)}$ is uniformly bounded over $K$ for all $i$.

We can view $\hat{v}^{(i)}$ as a sequence of maps from $\Omega_i$ to $\Omega_i \times X$ which are the identity on the first coordinate. $\hat{v}^{(i)}$ is holomorphic with respect to an almost complex structure $J^{(i)}$ on $\Omega_i \times X$ which is twisted from $J$ by $\hat{B}^{(i)}$ and has its boundary lying in ${\mb R} \times L$ (a totally real submanifold). Since $\hat{B}^{(i)}$ converges in $C^0$, $J^{(i)}$ converges uniformly to a continuous almost complex structure on $\Omega_i \times X$. Therefore, by the compactness result for continuous almost complex structures (see \cite[Theorem 1.2]{IS_reflection}), there exists a subsequence (still indexed by $i$) such that $\hat{v}^{(i)}$ converges in $W^{1, p}_{loc}$ to a limit $\hat{v}$. Then the weak limit $\hat{\bf v} = (\hat{B}, \hat{v})$ which is of regularity $W^{1, p}_{loc}$, is a solution to the affine vortex equation on ${\mb H}$ with boundary condition $u(\partial {\mb H}) \subset L$. By Proposition \ref{prop7}, there exists a gauge transformation $g$ of regularity $W^{2, p}_{loc}$ which can transform $\hat{\bf v}$ to a smooth solution. Using $g$ to modify $g_i$, we may just assume that $g = {\rm Id}$. 

Lastly, one replace each $g_i$ by smooth ones such that it still holds that $\hat{\bf v}^{(i)}$ converges weakly in $W^{1, p}(K)$ to $\hat{\bf v}$ for any compact subset $K\subset {\mb H}$. Then using Lemma \ref{lemma7}, one can modify each $g_i$ once more such that $\hat{\bf v}^{(i)}$ converges u.c.s. to $\hat{\bf v}$ on ${\mb H}$.\qed \end{proof}

\begin{proof}[Proof of Theorem \ref{thm20}]
By the energy quantization property of holomorphic spheres, holomorphic disks and affine vortices, it is easy to construct by induction a subsequence (still indexed by $i$), and a finite subset $Z \subset \Sigma$ which satisfy the first item of this theorem. Then apply Theorem \ref{thm19} to the case of ${\bf v}^{(i)}|_{\Sigma_i  \smallsetminus Z}$, we obtain a further subsequence which satisfies the second item of this theorem.
\qed \end{proof}

\section{Stable Scaled Disks and Compactness}\label{section4}

In this section we describe the objects we use to compactify the space of disk vortices with growing area forms introduced in Section \ref{section2}, and state the main theorem. We need some notions and notations about certain type of trees and their precise definitions are given in Appendix \ref{appendixb}.

\subsection{Scaled holomorphic disks}

Let $({\mc T}, \ud{\mc T}, {\mf s})$ with labelling $\iota: \{1, \ldots, \ud k\} \to V_0(\ud{\mc T}{}) \cup V_1(\ud{\mc T})$  be a based colored rooted tree, not necessarily stable (see Appendix \ref{appendixb} for its definition). To each $v_\alpha \in V({\mc T})$, we associate a (bordered) Riemann surface $\Sigma_\alpha$ as follows. If $v_\alpha = v_\infty$, then $\Sigma_{\infty} := {\mb D}$; if $v_\alpha \in V(\ud {\mc T}) \smallsetminus \{v_\infty \}$, then $\Sigma_\alpha := {\mb H}$; if $\alpha\in V({\mc T}) \smallsetminus V(\ud {\mc T})$, then $\Sigma_\alpha := {\mb C}$. For each $v_\alpha \neq v_\infty$, the ``$\infty$'' of $\Sigma_\alpha$ makes usual sense.

\begin{defn}\label{defn29}
A {\bf stable scaled holomorphic disk} in $(X, L)$ with combinatorial type $({\mc T}, \ud{\mc T}, {\mf s})$ is the following object
\beqn
{\mc C}:= \Big( ({\mc C}_\alpha)_{v_\alpha \in V({\mc T})},\ (z_{\alpha\beta})_{e_{\alpha \beta} \in E({\mc T})},\ (w_j)_{1\leq j \leq \ud k}\Big)
\eeqn
where the notations mean the following.
\begin{itemize}
\item[O1] For each $v_\alpha \in V_0({\mc T})$, ${\mc C}_\alpha$ is a $J$-holomorphic sphere or $J$-holomorphic disk $u_\alpha : (\Sigma_\alpha, \partial \Sigma_\alpha) \to (X, L)$.

\item[O2] For each $v_\alpha \in V_\infty({\mc T})$, ${\mc C}_\alpha$ is a $\bar{J}$-holomorphic sphere or $\bar{J}$-holomorphic disk $\bar{u}_\alpha: (\Sigma_\alpha, \partial \Sigma_\alpha) \to (\bar{X}, \bar{L})$.	

\item[O3] For each $v_\alpha \in V_1({\mc T})\smallsetminus V_1(\ud{\mc T})$, ${\mc C}_\alpha$ is a ${\mb C}$-vortex ${\bf v}_\alpha = (A_\alpha, u_\alpha)$.

\item[O4] For each $v_{\ud\alpha} \in V_1(\ud{\mc T})$, ${\mc C}_\alpha$ is an ${\mb H}$-vortex ${\bf v}_{\ud \alpha} = (A_{\ud \alpha}, u_{\ud \alpha})$.

\item[O5] For each $e_{\alpha \alpha'} \in E({\mc T}) \smallsetminus E(\ud{\mc T})$, $z_{\alpha \alpha'} \in {\rm Int} \Sigma_{\alpha'}$; for each $e_{\ud\alpha \ud\alpha'} \in E(\ud{\mc T})$, $z_{\ud\alpha \ud\alpha'} \in \partial \Sigma_{\ud\alpha'}$; for each $j$, $w_j \in \partial \Sigma_{\iota(j)}$. 

\end{itemize}
We require that the collection ${\mc C}$ satisfies the following conditions.
\begin{itemize}
\item[C1] For each $v_\beta \in V({\mc T})$, the points $z_{\alpha \beta}$ for all $e_{\alpha\beta} \in E({\mc T})$ and $w_j$ for all $\iota(j) = \beta$ are distinct.  We denote the set of them by $Z_\beta$.

\item[C2] For each $v_\alpha \in V({\mc T})$, the energy $E({\mc C}_\alpha)$ makes usual sense. We require that if $v_\alpha \in V({\mc T})$ is unstable, then $E({\mc C}_\alpha) >0$.

\item[C3] For $v_\alpha \in V({\mc T}) \smallsetminus \{ v_\infty\}$, the evaluation $\ev_\infty({\mc C}_\alpha)$ makes sense. We require that, if $v_{\alpha'} \in V_0({\mc T}) \cup V_1({\mc T})$, then $\ev_\infty({\mc C}_\alpha) = u_{\alpha'}(z_{\alpha \alpha'}) \in X$; if $v_{\alpha'} \in V_\infty({\mc T})$, then $\ev_\infty( {\mc C}_\alpha) = \bar{u}_{\alpha'}(z_{\alpha\alpha'}) \in \bar{X}$. 
\end{itemize}
The total energy of ${\mc C}$ is defined to be 
\beqn
E({\mc C}) = \sum_{v_\alpha \in V({\mc T})} E({\mc C}_\alpha).
\eeqn
\end{defn}

\subsection{Isomorphisms of stable scaled disks}

\begin{defn}\label{defn30}
Suppose for $s = 1, 2$,
\beqn
{\mc C}^s = \Big(  ( {\mc C}_\alpha^s)_{v_\alpha \in V({\mc T}^s)}, ( z^s_{\alpha\beta} )_{e_{\alpha\beta}\in E({\mc T}^s)},\ (w_j^s)_{1\leq j \leq \ud k}\Big)
\eeqn
are two stable scaled holomorphic disks with combinatorial types $({\mc T}^s, \ud{\mc T}^s, {\mf s}^s)$. An isomorphism from ${\mc C}^1$ to ${\mc C}^2$ with combinatorial given by a tree isomorphism $\rho: ({\mc T}^1, \ud{\mc T}^1, {\mf s}^1) \to ({\mc T}^2, \ud{\mc T}^2, {\mf s}^2)$ consists of the following objects
\beqn
\big( (\varphi_\alpha)_{v_\alpha \in V({\mc T}^1)}, (g_\alpha)_{\alpha \in V_1({\mc T}^1)}   \big).
\eeqn
Here 
\begin{enumerate}
\item $\varphi_\infty$ is the identity map on ${\mb D}$; for each $v_{\ud\alpha}^1 \in V_0(\ud{\mc T}^1) \cup V_\infty(\ud{\mc T}^1)$, $\varphi_{\ud\alpha}$ is a M\"obius transformation on $\Sigma_\alpha \simeq {\mb H}$; for each $v_{\ud\alpha}^1 \in V_1(\ud{\mc T}^1)$, $\varphi_{\ud\alpha}$ is a translation of ${\mb H}$; for each $v_\alpha^1 \in \big( V_0({\mc T}^1) \cup V_\infty({\mc T}^1) \big) \smallsetminus V(\ud{\mc T}^1)$, $\varphi_\alpha$ is a M\"obius transformation on $\Sigma_\alpha \simeq {\mb C}$; for each $v_\alpha^1 \in V_1({\mc T}^1) \smallsetminus V(\ud{\mc T}^1)$, $\varphi_\alpha$ is a translation of ${\mb C}$.

\item For each $v_\alpha^1 \in V_1({\mc T}^1)$, $g_\alpha: \Sigma_\alpha \to G$ is a smooth gauge transformation.
\end{enumerate}
They are subject to the following restrictions:
\begin{enumerate}
\item For each $v_{\alpha}^1 \in V_0({\mc T}^1)$, we have $u^1_\alpha = g_{\wt\alpha'}( z_{\wt\alpha \wt\alpha'} )^{-1} u^2_{\rho(\alpha)} \circ \varphi_\alpha$; (We explain the notations: for $v_\alpha^1 \in V_0({\mc T}^1)$, there exists a unique path $\alpha \succ \cdots \succ \wt\alpha$ contained in $V_0({\mc T}^1)$ such that $\wt\alpha' \in V_1({\mc T}^1)$.) 

\item For each $v_\alpha^1 \in V_1({\mc T}^1)$, we have ${\bf v}_\alpha^1 = ( g_\alpha )^* ( \varphi_\alpha )^*  {\bf v}_{\rho(\alpha)}^2$;

\item For each $v_\alpha^1 \in V_\infty({\mc T}^1)$, we have $\bar{u}^1_\alpha = \bar{u}_{\rho(\alpha)}^2\circ \varphi_\alpha$;

\item For each $e_{\alpha \beta}^1 \in E({\mc T}^1)$, we have $z^1_{\alpha\beta} = \varphi_{\beta}^{-1} ( z^2_{\rho(\alpha)\rho(\beta)})$.

\item For each $e_j \in E_\infty({\mc T}^1)$, we have $w_j^1 = \varphi_{\iota(j)}^{-1} ( w_j^2)$.
\end{enumerate}
\end{defn}
One can check that a stable scaled holomorphic disk ${\mc C}$ has finitely many automorphisms (modulo gauge transformations). 

\subsection{Definition of Convergence}

Now we define the notion of convergence in the adiabatic limit, for a sequence $({\bf v}^{(i)}, {\bf w}^{(i)}) \in \wt{\mc M}{}_{\ud k}^{\lambda_i}(L)$ with $\lambda_i \to +\infty$ towards a stable scaled holomorphic disk. This concept naturally extends to the case that if we replace the sequence of disk vortices by a sequence of ``stable solutions''. 

\begin{defn}\label{defn31}
Given a sequence $\lambda_i \to +\infty$ and a sequence $( {\bf v}^{(i)}, {\bf w}^{(i)} ) \in \wt{\mc M}_{\ud k}^{\lambda_i} (L)$. Suppose $({\mc T}, \ud{\mc T}, {\mf s})$ is a based colored rooted tree and 
\beqn
{\mc C} = \Big( ({\mc C}_\alpha)_{\alpha \in V({\mc T})}, (z_{\alpha \beta})_{e_{\alpha \beta} \in E({\mc T})},\ (w_j)_{1\leq j \leq \ud k} \Big) 
\eeqn 
is a stable scaled holomorphic disk modelled on $({\mc T}, \ud{\mc T}, {\mf s})$. $({\bf v}^{(i)}, {\bf w}^{(i)})$ is said to {\bf converge (modulo gauge) to ${\mc C}$}, if there are 1) for each $v_\alpha \in V_0({\mc T}) \cup V_1({\mc T})$, a sequence of gauge transformations $g_\alpha^{(i)}: \Sigma_\alpha \to G$; 2) for each $v_\alpha \in V({\mc T})$, a sequence of holomorphic maps $\varphi^{(i)}_\alpha: ( \Sigma_\alpha, \partial \Sigma_\alpha ) \to ( {\mb D}, \partial {\mb D} )$ with $\varphi^{(i)}_{\infty} = {\rm Id}_{\mb D}$, subject to the following conditions.
\begin{itemize}
\item[L1] For every $v_\alpha \in V_0({\mc T}) \cup V_1({\mc T})$, the sequence $( g^{(i)}_\alpha )^* ( \varphi^{(i)}_\alpha )^* {\bf v}^{(i)}$ converges u.c.s. to ${\bf v}_\alpha$ on $\Sigma_\alpha \smallsetminus Z_\alpha$.

\item[L2] For every $v_\alpha \in V_\infty({\mc T})$, the sequence $( \varphi^{(i)}_\alpha )^* {\bf v}^{(i)}$ converges to $\bar{u}_\alpha: \Sigma_\alpha \to \bar{X}$ on $\Sigma_\alpha \smallsetminus Z_\alpha$ in the sense of Definition \ref{defn17}.

\item[L3] For every $e_{\alpha \beta} \in E({\mc T})$, $( \varphi^{(i)}_\beta )^{-1} \circ \varphi^{(i)}_\alpha$ converges u.c.s. to the constant $z_{\alpha \beta}\in \Sigma_\beta$ on $\Sigma_\alpha$;

\item[L4] For every $j$, $(\varphi_{\iota(j)}^{(i)})^{-1}(w_j^{(i)})$ converges to $w_j$.

\item[L5] The sequence of energies $E({\bf v}^{(i)})$ converges to $E({\mc C})$.
\end{itemize}
\end{defn}

We can prove that the limit is unique up to isomorphisms of stable scaled holomorphic disks defined in Definition \ref{defn30}. The proof of the following theorem is based on an inductive construction of an isomorphism between two limiting objects. A similar construction in the case of holomorphic disks was carried out in \cite[Theorem 4.10]{Frauenfelder_disk} and we leave the details to the reader.
\begin{prop}\label{prop32}
If $({\bf v}^{(i)}, {\bf w}^{(i)}) \in \wt{\mc M}_{\ud k}^{\lambda_i}(L)$ converges (modulo gauge) to both ${\mc C}^1$ and ${\mc C}^2$, then ${\mc C}^1$ and ${\mc C}^2$ are isomorphic in the sense of Definition \ref{defn30}. 
\end{prop}

We could state our main theorem now.
\begin{thm}\label{thm33}
If $\lambda_i \to +\infty$ and $( {\bf v}^{(i)}, {\bf w}^{(i)} )\in \wt{\mc M}_{\ud k}^{\lambda_i}(L)$ is a uniformly bounded sequence. Then a subsequence of $({\bf v}^{(i)}, {\bf w}^{(i)})$ converges (modulo gauge) to a stable scaled holomorphic disk ${\mc C}$ in the sense of Definition \ref{defn31}.
\end{thm}

\section{Soft Rescaling}\label{section5}

In this subsection we prove two soft rescaling results. This is an analogue of a similar technique in proving the Gromov compactness for holomorphic spheres (see \cite[Theorem 4.7.1]{McDuff_Salamon_2004}) and holomorphic disks (see \cite[Theorem 3.5]{Frauenfelder_disk}). 

\subsection{Boundary soft rescaling in adiabatic limit}

First, choose $\hbar > 0$ which is less than the energy of 1) any nonconstant holomorphic sphere in $\bar{X}$; 2) any nonconstant holomorphic sphere in $X$;
 3) any nonconstant holomorphic disk in $(X, L)$; 4) any nonconstant holomorphic disk in $(\bar{X}, \bar{L})$; 5) any nontrivial ${\mb C}$-vortex in $X$; 6) any nontrivial ${\mb H}$-vortex in $(X, L)$.

\begin{lemma}\label{lemma34}
Let $\Omega \subset {\mb H}$ be an open subset of finite type and $w_0 \in \Omega$. Suppose a sequence of volume forms $\nu_i \in \Omega^2(\Omega)$ converges u.c.s. to $\nu \in \Omega^2(\Omega)$. Suppose $\lambda_i \to \infty$. Take a uniformly bounded sequence ${\bf v}^{(i)}\in \wt{\mc M}(\Omega, \lambda_i^2 \nu_i; X, L)$. Assume 
\beq\label{eqn51}
\lim_{r \to \infty} \lim_{i \to \infty} E ( {\bf v}^{(i)}; B_r(w_0) \cap {\mb H} ) = m_0 >0.
\eeq
Then there exist a subsequence (still indexed by $i$), a positive number $r_0$, a sequence of points $z_i \to w_0$, and a sequence of positive numbers $\delta_i \to 0$ (uniquely determined by $z_i$), such that 
\beqn
w \in B_{r_0} \smallsetminus \{ w_0\}\Longrightarrow \lim_{r\to 0} \lim_{i \to \infty} E ( {\bf v}^{(i)}; B_r(w) \big) = 0,\ e_i(z_i) = \sup_{B_{r_0}(w_0) \cap {\mb H}} e_i 
\eeqn
and
\beqn
E ( {\bf v}^{(i)}; B_{\delta_i}(z_i) \cap {\mb H} ) = m_0 - \frac{\hbar}{2}.
\eeqn
Here all the disks are taken with respect to the standard metric of ${\mb H}$. 
\end{lemma}

\begin{proof}
This is the same as the case of $J$-holomorphic curves (see \cite{Frauenfelder_disk}) and we omit the details.
\qed \end{proof}

\begin{defn}\label{defn35}
Suppose we are in the situation of Lemma \ref{lemma34}.
\begin{enumerate}
\item We say that the subsequence constructed {\bf has energy concentration of type ${\mc R}$ at $w_0$} with respect to the sequence $\lambda_i$, if: in the case $w_0 \in {\rm Int} \Omega $, $\displaystyle \lim_{i \to \infty} \lambda_i \delta_i = \infty$; in the case $w_0 \in \partial \Omega$, $\displaystyle \lim_{i \to \infty} \lambda_i ( \delta_i + {\rm Im} z_i )= \infty$.

\item We say that the subsequence constructed {\bf has energy concentration of type ${\mc S}$ at $w_0$} with respect to the sequence $\lambda_i$, if: in the case $w_0 \in {\rm Int} \Omega$, $\displaystyle \lim_{i \to \infty} \lambda_i \delta_i < \infty$; in the case $w_0 \in \partial \Omega$, $\displaystyle \lim_{i \to \infty} \lambda_i ( \delta_i + {\rm Im} z_i ) < \infty$.
\end{enumerate}
\end{defn}

\subsection{When the first bubble is in the quotient}

We will see that if a subsequence has energy concentration of type ${\mc R}$ (resp. type ${\mc S}$) at $w_0$, then for some further subsequence, to construct a bubble tree which is part of the stable scaled holomorphic disk, the first bubble we have to attach is a holomorphic curve in the quotient (resp. an affine vortex). The precise meaning of this is given by the following four propositions. 

\begin{prop}(cf. \cite[Proposition 44]{Ziltener_book})\label{prop36}
In the situation of Lemma \ref{lemma34}, if the sequence has an energy concentration of type ${\mc R}$ at $w_0 \in {\rm Int} \Omega$, then there exist a subsequence (still indexed by $i$), a sequence of M\"obius transformations $\varphi_i: {\mb C} \to \Omega \subset {\mb H}$, a finite subset $Z \subset {\mb C}$, satisfying the following conditions.
\begin{enumerate}
\item $\varphi_i$ converges u.c.s. to the constant $w_0$. 

\item $\hat{\bf v}^{(i)}:=(\varphi_i)^*{\bf v}^{(i)}$ converges modulo bubbling on ${\mb C} \smallsetminus Z$ to a holomorphic map $\bar{v}: {\mb C} \to \bar{X}$ in the sense of Definition \ref{defn17}.

\item If $E(\bar{v}) = 0$ then $\#Z\geq 2$.

\item For each $z_j \in Z$, the following limit
\beqn
m_j:= m(z_j):= \lim_{r \to 0} \lim_{i \to \infty} E ( \hat{\bf v}^{(i)}; B_r(z_j) )
\eeqn
exists and is positive. Moreover
\beqn
m_0 = E(\bar{v}) + \sum_{z_j\in Z} m_j. 
\eeqn

\item We have $\bar{u}(w_0) = \bar{v}(\infty) \in \bar{X}$. 
\end{enumerate}
\end{prop}
\begin{proof}
See the proof of \cite[Proposition 44]{Ziltener_book}, though the statement of the proposition there differs a bit from this one.
\qed \end{proof}

\begin{prop}\label{prop37} In the situation of Lemma \ref{lemma34}, if the sequence has an energy concentration of type ${\mc R}$ at $w_0 \in \partial \Omega$, then there exist a subsequence (still indexed by $i$), a sequence of M\"obius transformations $\varphi_i: {\mb H} \to \Omega \subset {\mb H}$, a finite subset $Z \subset {\mb H}$, satisfying the following conditions.

\begin{enumerate}
\item $\varphi_i$ converges u.c.s. to the constant $w_0$;

\item $\hat{\bf v}^{(i)}:= (\varphi_i)^* {\bf v}^{(i)}$ converges on ${\mb H} \smallsetminus Z$ to a holomorphic map $\bar{v}: ({\mb H}, {\mb R}) \to (\bar{X}, \bar{L})$ in the sense of Definition \ref{defn17};

\item If $E(\bar{v})=0$ and $\#Z<2$ then $Z$ consists of only one point in ${\rm Int}{\mb H}$;

\item For each $z_j \in Z$, the following limit
\beq\label{eqn52}
m_j:= m(z_j) = \lim_{r\to \infty} \lim_{i \to \infty} E (\hat{\bf v}^{(i)}; B_r(z_j)\cap {\mb H} )
\eeq
exists and is positive. Moreover
\beqn
m_0 = E(\bar{v}) + \sum_{z_j \in Z} m_j.
\eeqn

\item We have $\bar{u}(w_0) = \bar{v}(\infty)$.
\end{enumerate}

\end{prop}

\begin{proof}
Without loss of generality, assume $w_0 = 0 \in {\mb H}$. We assume $z_i = x_i + {\bm i} y_i$. By taking a subsequence if necessary, we have the following two possibilities, listed in the below table; and we define a sequence $\tau_i$ in respective cases.
\begin{center}
\beq\label{eqn53}
\begin{tabular}{|c|c|c|} 
\hline
Case & $\displaystyle \lim_{i \to \infty} ( y_i /\delta_i )$ &  $ \tau_i$\\
\hline
${\mc R1}$ & $< \infty$ & $ \delta_i$\\
\hline
${\mc R2}$ & $=\infty$ & $ y_i$\\
\hline
\end{tabular}
\eeq
\end{center}
Then denote $\Omega_i:= B \left( \frac{1}{\sqrt\tau_i}\right) \cap {\mb H}$ and define $\varphi_i: \Omega_i \to {\mb H}$ to be the map $\varphi_i(w) = x_i + \tau_i w$. We will prove that the properties (1)--(5) of this proposition hold for this sequence of M\"obius transformations and some subsequence. The argument is very much in parallel with the proof of \cite[Theorem 3.5]{Frauenfelder_disk}.

We denote 
\beqn
\nu_i' = \tau_i^{-2} \varphi_{i}^* \nu_i
\eeqn
and it is easy to see that $\nu_i'$ converges uniformly on compact subsets to the standard area form on ${\mb H}$. Consider the uniformly bounded sequence $\hat{\bf v}^{(i)}:= \varphi_i^* {\bf v}^{(i)} \in \wt{\mc M}(\Omega_i, \lambda_i^2\tau_i^2 \nu_i'; X, L)$. Since $\Omega_i$ exhausts ${\mb H}$ and $\lambda_i \tau_i \to \infty$, by Theorem \ref{thm19}, there exists a subsequence (still indexed by $i$), a finite subset $Z \subset {\mb H}$ and a holomorphic disk $\bar{v}: ({\mb H}, {\mb R}) \to (\bar{X}, \bar{L})$ such that $\hat{\bf v}^{(i)}$ converges to $\bar{v}$ in the sense of Definition \ref{defn17}. By taking a subsequence, we have that, for each $z_j \in Z$,
\beqn
m_j = \lim_{\epsilon \to 0} \lim_{i \to \infty} E ( \hat{\bf v}^{(i)}; B_{\epsilon}(z_j)\cap {\mb H})
\eeqn
exists and is no less than $\hbar$. Following \cite{Frauenfelder_disk}, we denote 
\beqn
\big( \kappa_i, \rho_i, \rho \big) = \left\{ \begin{array}{cc}
 \displaystyle \big( \frac{{\bm i} y_i}{\delta_i}, 1, 1 \big),\ & {\rm in\ case\ }({\mc R1}),\\[0.3cm]
 \displaystyle \big( {\bm i}, \frac{\delta_i}{y_i} , 0 \big),\ & {\rm in\ case\ }({\mc R2}).
\end{array} \right.
\eeqn
Choose a subsequence so that $\kappa_i$ converges to some $\kappa\in {\mb H}$; moreover, $\rho_i \to \rho$.

Now we can prove the following facts, following the line of \cite[P. 232-236]{Frauenfelder_disk}.

{\bf Step 1.} We prove that for any $z_j \in Z$, 
\beq\label{eqn54}
| z_j - \kappa | \leq \rho.
\eeq
Indeed, suppose for some $z_j \in Z$ and some $R> \rho$, 
\beq\label{eqn55}
R > | z_j - \kappa | > \rho.
\eeq
Then by \eqref{eqn51}, there exists $\epsilon>0$ such that
\beqn
\lim_{i \to \infty} E ( {\bf v}^{(i)}; B_\epsilon (w_0)\cap {\mb H} ) \leq m_0 + \frac{\hbar }{8}.
\eeqn
Therefore, for sufficiently large $i$, 
\beqn
E ( {\bf v}^{(i)}; B_\epsilon  \cap {\mb H} ) \leq m_0 + \frac{\hbar }{4}.
\eeqn
Therefore, since $\tau_i \to 0$ and $z_i \to 0$, for large $i$ we have
\beqn
E ( {\bf v}^{(i)}; B_{2 R \tau_i}(z_i) \cap {\mb H} ) \leq m_0 + \frac{\hbar }{4}.
\eeqn
Since $\varphi_i$ maps $B_R(\kappa)\cap {\mb H} \subset B_{2R}(\kappa^i)\cap {\mb H}$ into $B_{2R \tau_i}(z_i)\cap {\mb H}$, we have
\beqn
E ( \hat{\bf v}^{(i)}; B_R(\kappa)\cap {\mb H} ) \leq E ( {\bf v}^{(i)}; B_{2R \tau_i}(z_i) ) \leq m_0 + \frac{\hbar}{4}.
\eeqn
Moreover, because $\varphi_i$ maps $B_{\rho_i}(\kappa_i) \cap {\mb H}$ onto $B_{\delta_i}(z_i)$, we have
\beqn
E ( \hat{\bf v}^{(i)}; B_{\rho_i}(\kappa_i)\cap {\mb H} ) = E ( {\bf v}^{(i)}; B_{\delta_i} \cap {\mb H} ) = m_0 - \frac{\hbar}{2}.
\eeqn
\eqref{eqn55} implies that for $r$ sufficiently small and $i$ sufficiently large, we have $B_r(z_j ) \subset B_R(\kappa) \smallsetminus B_{\rho_i}(\kappa_i)$. It implies that
\beqn
\lim_{r\to 0} \lim_{i \to \infty} E (\hat{\bf v}^{(i)}; B_r(z_j) \cap {\mb H} ) \leq \lim_{i \to \infty} \Big( E ( \hat{\bf v}^{(i)}; B_R(\kappa) ) - E ( {\bf v}^{(i)}; B_{\rho_i}(\kappa_i) ) \Big) \leq \frac{3\hbar }{4}.
\eeqn
This contradicts the condition on points in $Z$. Therefore \eqref{eqn54} holds.

{\bf Step 2.} We prove that 
\beq\label{eqn56}
\lim_{R \to \infty} \lim_{i \to \infty} E ( {\bf v}^{(i)}; B_{R \tau_i}(z_i)\cap {\mb H} ) = m_0.
\eeq
Indeed, by the definition of $m_0$, $\tau_i$ and the fact that $\tau_i \to 0$, for every $R \geq 1$,
\beqn
m_0 - \frac{\hbar }{2} \leq \lim_{i \to \infty} E ( {\bf v}^{(i)}; B_{R \tau_i} (z_i) ) \leq m_0.
\eeqn
If \eqref{eqn56} is false, then there exists $\rho \in (0, \frac{\hbar}{2}]$ such that,
\beq\label{eqn57}
\lim_{R \to +\infty} \lim_{i \to \infty} E ({\bf v}^{(i)}; B_{R \tau_i}(z_i) ) = m_0 - \rho.
\eeq
Then we choose $R_0>1$ sufficiently large such that 
\beq\label{eqn58}
m_0 - \rho \geq \lim_{i \to +\infty} E ( {\bf v}^{(i)}; B_{R_0 \tau_i}(z_i) ) \geq m_0 - \frac{3}{2} \rho. 
\eeq
Therefore, for any $a> b >1$, 
\beq\label{eqn59}
\lim_{i \to \infty} E ( {\bf v}^{(i)}; B_{a R_0 \tau_i} (z_i) \smallsetminus B_{ b R_0 \tau_i}(z_i) ) \leq \frac{\rho}{2}.
\eeq
We will show that it leads to a contradiction. Indeed, for all $l \in {\mb N}$, there exists $\epsilon_l \in (0, \frac{1}{ l})$ and $i_l\in {\mb N}$ such that
\beqn
i \geq i_l \Longrightarrow \Big| E ( {\bf v}^{(i)}; B_{\epsilon_l}(z_i) \cap {\mb H} ) - m_0 \Big| \leq \frac{1}{l}.
\eeqn
We can assume that $\epsilon_{l+1} < \epsilon_l$ and $i_{l+1} > i_l$ for all $l$. Define $\epsilon_i= \epsilon_l$ for all $ i_l \leq i < i_{l+1}$, we have
\beqn
\lim_{i \to \infty} E ( {\bf v}^{(i)}; B_{\epsilon_i}(z_i) \cap {\mb H} ) = m_0.
\eeqn
Since $\lim_{i \to \infty} \epsilon_i = 0$, for every $R>1$, we have
\beq\label{eqn510}
\lim_{i \to \infty} E ( {\bf v}^{(i)}; B_{R \epsilon_i}(z_i) \cap {\mb H} ) = m_0.
\eeq
Then by \eqref{eqn57}, we have
\beq\label{eqn511}
\lim_{i \to \infty} \frac{\tau_i}{\epsilon_i} = 0.
\eeq
By Definition \ref{defn35}, it implies that $\displaystyle \lim_{i \to \infty} \lambda_i \epsilon_i = +\infty$. 

Denote $\Omega_i' = B\left( \frac{1}{ \sqrt \epsilon_i} \right) \cap {\mb H}$ and define $\varphi_i': \Omega_i' \to {\mb H}$ by $\varphi_i'(z) = x_i + \epsilon_i z$. Denote $\nu_i'' = \epsilon_i^{-2} (\varphi_i')^* \nu_i\in \Omega^2(\Omega_i')$, which converges u.c.s. to the standard area form on ${\mb H}$. Then $(\varphi_i')^* {\bf v}^{(i)}\in \wt{\mc M}( \Omega_i', \lambda_i^2 \epsilon_i^2 \nu_i''; X, L)$ is uniformly bounded. Then by Theorem \ref{thm20}, a subsequence of $(\varphi_i')^* {\bf v}^{(i)}$ converges (modulo bubbling) to a holomorphic disk in $(\bar{X}, \bar{L})$. Then for any $R>1$ and any $\delta>0$, we have
\beqn
\begin{split}
&\ \lim_{i \to \infty} E ( (\varphi_i')^* {\bf v}^{(i)}; ( B_R ( {\bm i} \epsilon_i^{-1} y_i) \smallsetminus  B_\delta ( {\bm i} \epsilon_i^{-1} y_i) ) \cap {\mb H} ) \\
= &\ \lim_{i \to \infty} E ( {\bf v}^{(i)}; ( B_{R \epsilon_i} (z_i) \smallsetminus B_{\delta \epsilon_i} (z_i) ) \cap {\mb H} )\\
\leq &\  \lim_{i \to \infty} E ( {\bf v}^{(i)}; ( B_{R\epsilon_i}(z_i) \smallsetminus  B_{\tau_i} (z_i) ) \cap {\mb H} ) \\
\leq &\ m_0 - (m_0 - \hbar/2) = \hbar/2.
\end{split}
\eeqn
Here the first inequality follows \eqref{eqn511}, and the second inequality follows from \eqref{eqn510} and the definition of $\tau_i$. Therefore we see that the limit of $(\varphi_i')^* {\bf v}^{(i)}$ is a constant disk and the bubbling only happens at the origin of the ${\mb H}$. Therefore, for any $T>0$, we have
\beq\label{eqn512}
\lim_{i \to \infty} E ( {\bf v}^{(i)}; ( B_{\epsilon_i}(z_i) \smallsetminus B_{e^{-T} \epsilon_i}(z_i) ) \cap {\mb H} ) = 0. 
\eeq

On the other hand, since $\displaystyle \lim_{i \to \infty} ( y_i/ \epsilon_i ) = 0$, for $i$ large enough we have
\beq\label{eqn513}
B_{\epsilon_i}(z_i) \cap {\mb H} \subset B_{2 \epsilon_i}(x_i) \cap {\mb H} \subset B_{3 \epsilon_i}(z_i) \cap {\mb H}.
\eeq
On the other hand, since $R_0 \tau_i \geq y_i$, we have
\beq\label{eqn514}
B_{R_0 \tau_i }(y_i)\cap {\mb H} \subset B_{2 R_0 \tau_i}(x_i) \cap {\mb H} \subset B_{3 R_0 \tau_i}(y_i) \cap {\mb H}.
\eeq
Therefore, by \eqref{eqn58}, \eqref{eqn513} and \eqref{eqn514}, we have
\beq\label{eqn515}
\begin{split}
\rho \leq &\ \lim_{i \to \infty} E ( {\bf v}^{(i)}; ( B_{\epsilon_i}(z_i) \smallsetminus B_{3 R_0 \tau_i} (z_i) ) \cap {\mb H} ) \\
     \leq &\ \lim_{i \to \infty} E ( {\bf v}^{(i)}; ( B_{2 \epsilon_i}(x_i) \smallsetminus B_{2 R_0 \tau_i} (x_i) ) \cap {\mb H})\\
	   \leq &\ \lim_{i \to \infty} E ( {\bf v}^{(i)}; ( B_{3 \epsilon_i}(z_i) \smallsetminus B_{R_0 \tau_i} (z_i) ) \cap {\mb H} ) \\
		\leq &\ \frac{3}{2} \rho.  
\end{split}
\eeq

Similarly, we can show that 
\beq\label{eqn516}
\lim_{i \to \infty} E ( {\bf v}^{(i)}; ( B_{2\epsilon_i}(x_i) \smallsetminus B_{e^{-T} 2 \epsilon_i}(x_i) ) \cap {\mb H} ) = 0.
\eeq

On the other hand, since $\frac{3}{2} \rho \leq \frac{3}{4} \hbar < \hbar$, we can apply the annulus lemma (Proposition \ref{prop57}) to the neck region $( B_{ 2 \epsilon_i}(x_i) \smallsetminus B_{2 R_0 \tau_i}(x_i) ) \cap {\mb H} $, which is biholomorphic to a long strip (notice that $\lambda_i \tau_i \to +\infty$, so the area form on the strip is bounded from below). Then for $T$ large enough,  the energy on the region $( B_{e^{-T} 2 \epsilon_i}(x_i) \smallsetminus B_{e^T 2 R_0 \tau_i} (x_i) ) \cap {\mb H}$ is smaller than $\rho/2$. This implies that
\beq\label{eqn517}
\lim_{i \to \infty} E ( {\bf v}^{(i)}; ( B_{e^{-T} \epsilon_i}(z_i) \smallsetminus B_{e^T 3 R_0 \tau_i}(z_i) ) \cap {\mb H} ) < \rho/2.
\eeq
Combining \eqref{eqn515}--\eqref{eqn517}, it implies that
\beq\label{eqn518}
\lim_{i \to \infty} E ( {\bf v}^{(i)}; ( B_{e^T 3R_0 \tau_i} (z_i) \smallsetminus B_{3 R_0 \tau_i}(z_i) )\cap {\mb H} ) > \rho/2.
\eeq
This contradicts \eqref{eqn59}. Therefore \eqref{eqn56} holds.

{\bf Step 3.} We prove that in case ${\mc R1}$, if the limit $\bar{v}$ is constant, then $\#Z \geq 2$. Indeed, if $\bar{v}$ is constant, then since $d (Z, \kappa) \leq 1$, it follows that for $1< r< R$, 
\beqn
\lim_{i \to \infty} E ( \hat{\bf v}^{(i)}; (B_R(\kappa) \smallsetminus B_r(\kappa)) \cap {\mb H}) = E ( \bar{v}; (B_R(\kappa) \smallsetminus B_r(\kappa)) \cap {\mb H} ) = 0.
\eeqn
Therefore the limit of $E ( \hat{\bf v}^{(i)}; B_r(\kappa) \cap {\mb H} )$ is independent of $r>1$. Then for every $r>1$, choose $\rho_1 \in (1, r)$ and $\rho_2 > r$, we have $B_{\rho_1}(\kappa) \subset B_r(\kappa_i) \subset B_{\rho_2}(\kappa)$ for $i$ large enough. Hence
\beqn
\lim_{i \to \infty} E ( {\bf v}^{(i)}; B_{r \tau_i}(z_i) \cap {\mb H} ) = \lim_{i \to \infty} E ( \hat{\bf v}^{(i)}; B_r(\kappa_i)\cap {\mb H})
\eeqn
is also independent of $r>1$. Then by \eqref{eqn56}, for every $r>1$, we have
\beqn
\lim_{i \to \infty} E ( \hat{\bf v}^{(i)}; B_r(\kappa) \cap {\mb H} ) = m_0.
\eeqn
However, for $r = 1$ the above limit is $m_0 - \hbar/2$. Therefore there must be energy concentration of $\hat{\bf v}^{(i)}$ which happens on $\partial B_1(\kappa) \cap {\mb H}$, so $Z \neq \emptyset$. However, since $\kappa_i$ is the point which has the greatest energy	 density, $Z \neq \emptyset$ implies that $\kappa \in Z$. Therefore $\# Z \geq 2$. 

{\bf Step 4.} We prove that 
\beq\label{eqn519}
m_0 = E(\bar{v}) + \sum_{z^j \in Z} m_j.
\eeq
We take $r>1$. Then for $i$ large, by \eqref{eqn54}, $Z \subset B_r(\kappa_i)$. Then by \eqref{eqn56}, we have
\beqn
\begin{split}
m_0 = &\ \lim_{R \to \infty} \lim_{i\to \infty} E ( {\bf v}^{(i)}; B_{R \tau_i}(z_i) \cap {\mb H}) \\[0.25cm]
    = &\ \lim_{R \to \infty} \lim_{i \to \infty} E ( \hat{\bf v}^{(i)}; B_R (\kappa_i) \cap {\mb H} ) \\[0.25cm]
	  = &\ \lim_{R \to \infty} \lim_{i \to \infty} E ( \hat{\bf v}^{(i)}; (B_R (\kappa_i) \smallsetminus B_r(\kappa_i) ) \cap {\mb H} ) + \lim_{i \to \infty} E ( \hat{\bf v}^{(i)}; B_r(\kappa_i)\cap {\mb H} ) \\[0.25cm]
		= &\ 	E ( \bar{v}; {\mb H} \smallsetminus B_r(\kappa) ) + \lim_{\epsilon\to 0} \lim_{i \to \infty} E (\hat{\bf v}^{(i)}; ( B_r(\kappa) \smallsetminus \cup_{z_j} B_\epsilon(z_j) ) \cap {\mb H} ) + \sum_{z_j\in Z} m_j\\
		= &\ E ( \bar{v} ) + \sum_{z_j \in Z} m_j.
\end{split}
\eeqn

{\bf Step 5.} We prove that in case ${\mc R2}$, $\bar{v}$ is necessarily constant and $Z = \{{\bm i}\}$. Indeed, in this case, \eqref{eqn54} implies that $Z = \{{\bm i}\}$. Since $\kappa_i = {\bm i}$, we have
\beqn
E ( \hat{\bf v}^{(i)}; B_{\frac{\delta_i}{y_i}} ({\bm i}) \cap {\mb H} ) = m_0 - \hbar/2.
\eeqn
Since $\displaystyle \lim_{i \to \infty} \left( \delta_i / y_i \right) = 0$, we see that
\beqn
m_1 = \lim_{\epsilon \to 0} \lim_{i \to \infty} E ( \hat{\bf v}^{(i)}; B_\epsilon({\bm i}) \cap {\mb H} ) \geq \lim_{i \to \infty} E ( \hat{\bf v}^{(i)}; B_{\frac{\delta_i}{y_i }}({\bm i}) \cap {\mb H} )  = m_0 - \hbar/2.
\eeqn
Therefore, by \eqref{eqn519}, we have
\beqn
E ( \bar{v} ) = m_0 - m_1 \leq \hbar/2.
\eeqn
By the energy quantization property of holomorphic disks, we know that $\bar{v}$ is necessarily a constant holomorphic disk. 

{\bf Step 6.} The fact that $\bar{v}(\infty) = \bar{u}(0)$ can be proved using the annulus lemma (Proposition \ref{prop57}), noticing that no energy escapes at the neck. The proof of the proposition is complete.
\qed \end{proof}

\subsection{When the first bubble is an affine vortex}

We have the following two propositions treating type ${\mc S}$ energy concentrations. The first one (the interior case) was proved in \cite{Ziltener_book} and the second one is the analogue of the first one in the boundary case.

\begin{prop}(cf. \cite[Proposition 44]{Ziltener_book})\label{prop38}
In the situation of Lemma \ref{lemma34}, if the sequence has an energy concentration of type ${\mc S}$ at $w_0 \in {\rm Int} \Omega$, then there exist a subsequence (still indexed by $i$), a sequence of M\"obius transformations $\varphi_i: {\mb C} \to \Omega \subset {\mb H} \subset {\mb C}$, a sequence of smooth gauge transformations $g_i: {\mb C} \to G$, a finite subset $Z \subset {\mb C}$, satisfying the following conditions.

\begin{enumerate}
\item $\varphi_i$ converges u.c.s. on ${\mb C}$ to the constant $w_0$.

\item $\hat{\bf v}^{(i)}:= (g_i)^* (\varphi_i)^* \hat{\bf v}^{(i)}$ converges u.c.s. on ${\mb C}\smallsetminus Z$ to a ${\mb C}$-vortex $\hat{\bf v}$

\item If $(B, v)$ is trivial then $\#Z \geq 1$.

\item For each $z_j \in Z$, we have
\beqn
m_j:= m(z_j) = \lim_{r\to \infty} \lim_{i \to \infty} E ( \hat{\bf v}^{(i)}; B_r(z_j)\cap {\mb H} )
\eeqn
exists and is positive; and
\beqn
m_0 = E( B, v) + \sum_{z_j \in Z} m_j.
\eeqn

\item We have $\bar{u}(w_0) = \ev_\infty({\bf v})\in \bar{X}$.

\end{enumerate}
\end{prop}

\begin{prop}\label{prop39}
In the situation of Lemma \ref{lemma34}, if the sequence has an energy concentration of type ${\mc S}$ at $w_0 \in \partial \Omega$, then there exist a subsequence (still indexed by $i$), a sequence of M\"obius transformations $\varphi_i: {\mb H} \to \Omega \subset {\mb H}$, a sequence of smooth gauge transformations $g_i: {\mb H} \to G$, a finite subset $Z \subset {\mb H}$, satisfying the following conditions.

\begin{enumerate}
\item $\varphi_i$ converges u.c.s. on ${\mb H}$ to the constant $w_0$.

\item $\hat{\bf v}^{(i)}:= (g_i)^* (\varphi_i)^* {\bf v}^{(i)}$ converges u.c.s. on ${\mb H} \smallsetminus Z$ to an ${\mb H}$-vortex ${\bf v}$.

\item If ${\bf v}$ is trivial then $\# \geq 1$;

\item For each $z_j \in Z$, we have
\beqn
m_j:= m(z_j) = \lim_{r\to \infty} \lim_{i \to \infty} E ( \hat{\bf v}^{(i)}; B_r(z_j)\cap {\mb H} )
\eeqn
exists and is positive; and
\beq\label{eqn520}
m_0 = E( {\bf v}) + \sum_{z_j \in Z} m_j.
\eeq

\item We have $\bar{u}(w_0) = \ev_\infty( {\bf v})\in \bar{L}$.
\end{enumerate}
\end{prop}

\begin{proof}
We write $z_i = x_i  + {\bm i} y_i$. We define 
\beqn
\kappa_i = {\bm i} y_i \lambda_i,\ \rho_i = \lambda_i \delta_i.
\eeqn
Choose a subsequence such that $\kappa:= \displaystyle \lim_{i \to \infty} \kappa_i$ and $\rho:= \displaystyle \lim_{i \to \infty} \rho_i$ exist. Define
\beqn
\varphi_i(w) = x_i + \lambda_i w,\ w \in B( {\sqrt{\lambda_i}} ) \cap {\mb H}.
\eeqn
Denote $\nu_i' = \lambda_i^{-2} \varphi_i^* \nu_i$. Then $\hat{\bf v}^{(i)} := (\varphi_i)^* {\bf v}^{(i)}\in \wt{\mc M}(B (\sqrt\lambda_i)\cap {\mb H}, \nu_i'; X, L)$. $\nu_i'$ converges u.c.s. on ${\mb H}$ to the standard area form. Then by Theorem \ref{thm15}, there exist a subsequence (still indexed by $i$), a finite subset $Z \subset {\mb H}$, and an ${\mb H}$-vortex ${\bf v}$, such that for any compact subset $K\subset {\mb H}\smallsetminus Z$, $\hat{\bf v}^{(i)}$ converges u.c.s. on ${\mb H}\smallsetminus Z$ to ${\bf v}$ modulo gauge transformation. Moreover, for any $z_j \in Z$, we have
\beqn
m_j:= \lim_{r \to \infty} E ( \hat{\bf v}^{(i)}; B_r(z^j) \cap {\mb H} ) \geq \hbar.
\eeqn

{\bf Step 1.} By the same argument as deriving \eqref{eqn54} while replacing $\tau_i$ by $\lambda_i^{-1}$, we can obtain
\beqn
z_j \in Z \Longrightarrow |z_j - \kappa| \leq \rho.
\eeqn 

{\bf Step 2.} We prove that
\beq\label{eqn521}
\lim_{R \to +\infty} \lim_{i \to \infty} E ( {\bf v}^{(i)}; B_{R \lambda_i^{-1}}(z_i) \cap {\mb H} ) = m_0. 
\eeq
Indeed, if the limit in \eqref{eqn521} is $m_0 -\rho$ for some $\rho>0$, then as we did in {\bf Step 2} in the proof of Proposition \ref{prop37}, we can find a subsequence (still indexed by $i$) and a sequence $\epsilon_i>0$ such that
\beqn
\lim_{i \to +\infty} \lambda_i \epsilon_i = +\infty,\ \lim_{i \to \infty} E ( {\bf v}^{(i)}; B_{\epsilon_i}(z_i) \cap {\mb H} ) = m_0.
\eeqn
We define $\varphi_i': B \left( \frac{1}{ \sqrt\epsilon_i} \right) \cap {\mb H} \to {\mb H}$ by $\varphi_i' (z) = x_i + \epsilon_i z$. As we did in {\bf Step 2} in the proof of Proposition \ref{prop37}, we can show that $(\varphi_i')^* {\bf v}^{(i)}$ converges on ${\mb H}^*$ to a trivial holomorphic disk. Therefore we derive \eqref{eqn514}--\eqref{eqn516} in the same way with $\tau_i$ replaced by $\lambda_i^{-1}$. The annulus lemma (Proposition \ref{prop57}) can still be applied to the neck region $( B_{2\epsilon_i}(x_i) \smallsetminus B_{2R_0 \tau_i}(x_i) ) \cap {\mb H}$ because the pull-backed area form on the strip is uniformly bounded from below. Then we can derive \eqref{eqn517} and \eqref{eqn518} and a similar contradiction. Therefore \eqref{eqn521} holds.

{\bf Step 3.} We prove that if $E({\bf v}) = 0$, then $Z \neq \emptyset$. Indeed, if $Z = \emptyset$, then the convergence $\hat{\bf v}^{(i)}$ is uniformly on any compact subset of ${\mb H}$. Therefore, 
\beqn
\begin{split}
0 = E({\bf v}) = &\ \lim_{R \to +\infty} \lim_{i \to \infty} E ( \hat{\bf v}^{(i)}; B_R(0) \cap {\mb H} ) \\
= &\ \lim_{R \to +\infty} \lim_{i \to \infty} E ( {\bf v}^{(i)}; B_{R \lambda_i^{-1}}(x_i) ) \\
= &\ m_0. 
\end{split}
\eeqn
This contradicts the fact $m_0>0$.

{\bf Step 4.} We prove \eqref{eqn520} as in {\bf Step 4} of the proof of Proposition \ref{prop37}.

{\bf Step 5.} The fact that $\bar{u}(w_0) = \ev_\infty( {\bf v})$ can be proved by using the annulus lemma (Proposition \ref{prop57}) by noticing that no energy escapes at the neck.
\qed \end{proof}

\section{Proof of the Compactness Theorem}\label{section6}

Now we are ready to prove the main compactness theorem (Theorem \ref{thm33}). The construction contains the following steps. First we will construct the limit curve, whose components are called the principal components. Via the sequence of M\"obius transformations associated to each principal component, the pull-backed sequence of vortices converges modulo bubbling in a sense depending on the scale of this component. Secondly, besides those bubbling points in the new curve, there might be energy which escapes from boundary nodal points and marked points. We construct a new object which exhausts the escaped energy by adding ``connecting bubbles''. Thirdly, the remaining energy loss happens at the bubbling points. We can find the complete bubble tree to be attached to the bubbling points on the principal components, where we need to use the soft rescaling results obtained in Section \ref{section5}. This will give us the limiting object and complete the proof.

\subsection{The limiting curve}

First, ignore the sequence $\lambda_i$. Taking a subsequence (still indexed by $i$), we may assume that $\lim_{i \to \infty} w_j^{(i)} = w_j \in \partial {\mb D}$ for $j = 1, \ldots, \ud k$. Then adding three more marked points $w_{-1}, w_{-2}, w_{-3}$ which avoid $w_j^{(i)}$ for all $j$ and all large $i$, we may assume that the sequence of disks with marked points 
\beqn
(w_1^{(i)}, \ldots, w_{\ud k}^{(i)}, w_{-1}, w_{-2}, w_{-3})
\eeqn
converges in $\ov{\mc M}_{0, \ud k+3}$ to a stable marked disk with $\ud k+3$ marked points, with combinatorial type given by a ribbon tree. Removing the exterior vertices corresponding to (the limit of) $w_{-1}$, $w_{-2}$, $w_{-3}$, we obtained a nodal disk (independent of the additional marked points $w_{-1}, w_{-2}, w_{-3}$) whose components are all stable except possibly for the one where $w_{-1}, w_{-2}, w_{-3}$ originally lied on. Such a nodal disk has its combinatorial type given by a rooted ribbon tree $(\ud{\mc T}{}_0, v_\infty)$, where $v_\infty$ represents the disk from which we removed $w_{-1}, w_{-2}, w_{-3}$. 

Now each vertex $v_{\ud\alpha} \in V(\ud{\mc T}{}_0)$ corresponds to a subtree of $\ud{\mc T}{}_0$, which contains all vertices $\ud\beta$ with $\ud \beta \geq \ud\alpha$, denoted by $\ud{\mc T}{}_0(\ud\alpha)$, which has root $v_{\ud \alpha}$. Moreover, $\alpha$ corresponds to a subset $I_{\ud\alpha} \subset \{1, \ldots, \ud k\}$. If $v_{\ud\alpha} \neq v_\infty$, then $I_{\ud\alpha}$ contains at least two elements. We denote
\beqn
d_{\ud \alpha}^{(i)} = {\rm diam} \big\{ w_j^{(i)}\ |\ j \in I_{\ud\alpha} \big\}
\eeqn
where the diameter is taken with respect to the metric on ${\mb D}$. We also define $d_\infty^{(i)} = 1$ for all $i$. Then it is easy to see that 
\begin{enumerate}
\item $v_{\ud\alpha} \neq v_\infty \Longrightarrow \displaystyle \lim_{i \to \infty} d_{\ud\alpha}^{(i)} = 0$, and 

\item ${\mc T}_0(\ud\alpha) \subset \ud{\mc T}{}_0(\beta), \ud\alpha \neq \ud\beta \displaystyle \Longrightarrow \lim_{i \to \infty} ( d_{\ud\alpha}^{(i)} /  d_{\ud\beta}^{(i)} ) = 0$.
\end{enumerate}

Now we define a coloring ${\mf s}_0: V(\ud{\mc T}{}_0) \to \{0, 1, \infty\}$ by using the sequence $\lambda_i$. Taking a subsequence if necessary, we define
\beqn
{\mf s}_0 (v_{\ud\alpha}) = \left\{ \begin{array}{cl} 0,\ & {\rm if}\ \displaystyle \lim_{i \to \infty} \lambda_i d_{\ud\alpha}^{(i)} = 0;\\
                                                    1,\ & {\rm if}\ \displaystyle \lim_{i \to \infty} \lambda_i d_{\ud\alpha}^{(i)} \in (0, \infty);\\
																										\infty,\ & {\rm if}\ \displaystyle \lim_{i \to \infty} \lambda_i d_{\ud\alpha}^{(i)} = \infty.
\end{array} \right.
\eeqn

Moreover, if there are two adjacent vertices $\ud\alpha$, $\ud\beta$ with ${\mf s}_0 (\ud\alpha) = 0$, ${\mf s}_0 (\ud\beta) = \infty$, then we modify the tree by inserting a new vertex $\ud\gamma$ between $\ud\alpha$ and $\ud\beta$, and define ${\mf s}_0 (\ud\gamma) = 1$. Then we obtain a colored rooted ribbon tree, denoted by $(\ud{\mc T}{}_0, {\mf s}_0 )$. To each such new vertex, we associate the sequence $d_{\ud\gamma}^{(i)}:= \lambda_i^{-1}$.

On the other hand, for each $\ud\alpha \in V({\mc T}_0) \smallsetminus \{v_\infty \}$, we consider the sequence of M\"obius transformations $\varphi_{\ud\alpha}^{(i)}: {\mb H} \to {\mb D}$ given by
\beq\label{eqn61}
\varphi_{\ud\alpha}^{(i)} (w) = w_{j_{\ud\alpha}}^{(i)} \frac{ {\bm  i} - d_{\ud\alpha}^{(i)} w }{{\bm i } + d_{\ud\alpha}^{(i)} w }.
\eeq
It maps $0 \in {\mb H}$ to $w_{j_{\ud\alpha}}^{(i)} \in \partial {\mb D}$ where $j_{\ud\alpha} \in I_{\ud\alpha}$ is the ``first'' element of $I_{\ud\alpha}$. We also define $\varphi_{v_0}^{(i)} = {\rm Id}_{\mb D}$ for all $i$. Then for all $j \in I_{\ud\alpha}$, we denote
\beqn
Z_{\ud\alpha}:= \left\{ \lim_{i \to \infty} \big( \varphi_{\ud\alpha}^{(i)} \big)^{-1} \big( w_j^{(i)} \big) \ |\ j \in I_{\ud\alpha}   \right\}.
\eeqn
By our construction, $Z_{\ud\alpha}$ is a finite subset of ${\mb R} \subset {\mb H}$ with at least two elements, except for the case where $\ud\alpha$ is the newly added vertex in $V_1(\ud{\mc T}{}_0)$, in which case $Z_{\ud\alpha} = \{0\}$. Moreover, by direct calculation, for every $v_{\ud\alpha} \in V(\ud{\mc T}{}_0)\smallsetminus \{v_\infty\}$ and for each compact subset $K \subset {\mb H}$,
\beqn
\lim_{i \to \infty} \big( d_{\ud\alpha}^{(i)} \big)^{-1}  \big\| d \varphi_{\ud\alpha}^{(i)} \big\|_{L^\infty(K)} \in (0, \infty).
\eeqn
We assign $\Sigma_{\ud\alpha} = {\mb D}$ if $\ud\alpha = \infty$ and otherwise $\Sigma_{\ud\alpha} = {\mb H}$. 

Now using the sequence of M\"obius transformations $\varphi_{\ud\alpha}^{(i)}$ to pull-back the sequence of vortices ${\bf v}^{(i)}$, we accomplish the first step of our construction. It is summarized as follows.

\begin{prop}\label{prop40}
There exist the following objects
\begin{enumerate}

\item a subsequence of the original sequence ${\bm s}_0$ (still indexed by $i$);

\item a based colored rooted tree ${\mc T}{}_0 = ({\mc T}_0, \ud{\mc T}{}_0, {\mf s}_0)$ with $\iota_0: \{1, \ldots, \ud k\} \to V_0 (\ud{\mc T}{}_0)\cup V_1(\ud{\mc T}{}_0)$ and for each $v_{\ud\alpha} \in V(\ud{\mc T}{}_0)$, a corresponding subset $I_{\ud\alpha} \subset \{1, \ldots, \ud k\}$;

\item a collection ${\mc C}_0$ of objects
\beqn
\Big( ({\mc C}_\alpha)_{v_\alpha \in V({\mc T}{}_0)},\ (z_{\alpha \alpha'})_{e_{\alpha \alpha'} \in E({\mc T}{}_0)},\ (w_j)_{1 \leq j \leq \ud k} \Big)
\eeqn
which are described as in Definition \ref{defn29} (O1)--(O5) for the case ${\mc T} = {\mc T}_0$;

\item for each $v_\alpha \in V({\mc T}{}_0)$ a finite subset $W_\alpha \subset \Sigma_\alpha \smallsetminus Z_\alpha$ and for each $w_\alpha \in W_\alpha$, a number $m(w_\alpha)>0$ (see Definition \ref{defn29} for the meaning of $Z_\alpha$);

\item for each $v_\alpha \in V({\mc T}{}_0)$ a sequence of M\"obius transformations $\varphi_\alpha^{(i)}: \Sigma_\alpha \to {\mb D}$ with $\varphi_\infty^{(i)} = {\rm Id}_{{\mb D}}$, and a sequence of numbers $d_\alpha^{(i)}>0$;

\item for each $v_\alpha \in V_0({\mc T}{}_0) \cup V_1({\mc T}{}_0)$ a sequence of smooth maps $g_\alpha^{(i)}: \Sigma_\alpha \to G$;

\end{enumerate}
They satisfy the following conditions.

\begin{enumerate}
\item[(i)] ${\mc C}_0$ and ${\mc T}{}_0$ satisfy the conditions of Definition \ref{defn29} except (C3).

\item[(ii)] For each $v_\alpha \in V({\mc T}_0)$, the subset $Z_\alpha$ and the subset $W_\alpha$ are disjoint. 

\item[(iii)] For $v_\alpha, v_\beta \in V({\mc T}_0)$, $\alpha> \beta \Longrightarrow \displaystyle \lim_{i \to \infty} ( d_\alpha^{(i)} /d_\beta^{(i)} ) = 0$; moreover, $d_{v_\infty}^{(i)} \equiv 1$ and 
\beqn
\lim_{i \to \infty} \lambda_i d_\alpha^{(i)} \in  \left\{ \begin{array}{cl} \{ +\infty\}, &\ v_\alpha \in V_\infty({\mc T}_0);\\
                                                                       (0, +\infty), &\ v_\alpha \in V_1({\mc T}_0);\\
																																			  \{0\}, &\ v_\alpha \in V_0({\mc T}_0).
\end{array} \right.
\eeqn

\item[(iv)] For each $v_\alpha \in V({\mc T}_0)$ and for any compact subset $K \subset \Sigma_\alpha$, we have
\beqn
\lim_{i \to \infty} \big( d_\alpha^{(i)} \big)^{-1} \big\| d \varphi_\alpha^{(i)} \big\|_{L^\infty(K)} \in (0, +\infty).
\eeqn
Moreover, if $v_{\ud\alpha} \in V(\ud{\mc T}{}_0)$, then 
\beqn
\varphi_{\ud\alpha}^{(i)} (z) = z_{j_{\ud \alpha}}^{(i)} \frac{ {\bm i} - d_{\ud\alpha}^{(i)} z}{{\bm i} + d_{\ud\alpha}^{(i)} z }.
\eeqn

\item[(v)] The collection ${\mc C}_0$, the sequences $\varphi_\alpha^{(i)}$ and $g_\alpha^{(i)}$ satisfy the conditions (L1)--(L4) of Definition \ref{defn31} with ${\mc T} = {\mc T}_0$ and each $Z_\alpha$ replaced by $Z_\alpha \cup W_\alpha$.

\item[(vi)] For each $w_\alpha \in W_\alpha$, we have
\beqn
m(w_\alpha):= \lim_{\epsilon \to 0} \limsup_{i \to \infty} E ( ( \varphi_\alpha^{(i)} )^* {\bf v}^{(i)}; B_\epsilon(w_\alpha) \cap {\mb H} ) > 0.
\eeqn
\end{enumerate}
\end{prop}

\begin{proof}
The tree $({\mc T}_0, \ud{\mc T}{}_0, {\mf s}_0)$, the sequences of M\"obius transformations $\varphi_\alpha^{(i)}$ and positive numbers $d_\alpha^{(i)}$ in (5) are constructed as we just described. Notice that at this moment ${\mc T}_0 = \ud{\mc T}{}_0$. They satisfy the conditions (iii) and (iv). For each $\alpha$, for the sequences $\varphi_\alpha^{(i)}$ defined by \eqref{eqn61}, we consider the sequence of vortices 
\beqn
( \varphi_\alpha^{(i)} )^* {\bf v}^{(i)}
\eeqn
defined on an exhausting sequence of compact subsets of $\Sigma_\alpha$. Then by applying Theorem \ref{thm15} (if $v_\alpha \in V_0({\mc T}_0) \cup V_1({\mc T}_0)$) or Theorem \ref{thm20} (if $\alpha \in V_\infty({\mc T}_0)$), we can extract a subsequence (still indexed by $i$) and objects claimed in (3), (4), (6). It is easy to check that this collection satisfies the conditions (i)--(vi). \qed 
\end{proof}

\subsection{Bubbling at nodes and markings}\label{subsection63}

For each edge $\ud{e} \in E(\ud{\mc T}{}_0)\cup E_\infty(\ud{\mc T}{}_0)$, define 
\beq\label{eqn62}
N_{\ud e}^{(i)}:= \left\{ \begin{array}{cl}\displaystyle  \left\{ z\in {\mb H}\ |\  \sqrt{ \frac{d_{\ud\alpha}^{(i)}}{ d_{\ud \alpha'}^{(i)}} } < |z| <  \sqrt{ \frac{ d_{\ud \alpha'}^{(i)} }{ d_{\ud \alpha}^{(i)} } }  \right\},\ & \ud{e} = e_{\ud\alpha \ud\alpha'} \in E(\ud{\mc T}{}_0);\\[0.6cm]
               \Big\{ z\in {\mb H}\ |\ 0 < |z| <  \big(  d_{\iota(j)}^{(i)} \big)^{-1}  \Big\},\ & \ud{e}  = e_j \in E_\infty(\ud{\mc T}{}_0),
\end{array} \right.
\eeq
\beq\label{eqn63}
d_{\ud{e}}^{(i)} = \left\{ \begin{array}{cc} \sqrt{ d_{\ud\alpha}^{(i)} d_{\ud\alpha'}^{(i)}},\ & \ud{e} = e_{\ud\alpha \ud\alpha'} \in E(\ud{\mc T}{}_0);\\
                             ( d_{\ud\beta}^{(i)})^2,\ & \ud{e} = e_j \in E_\infty(\ud{\mc T}{}_0),\ \iota(j) = \ud\beta, \end{array} \right.
\eeq
Define $z_{\ud{e}}^{(i)}$ to be $z_{j_{\ud\alpha}}^{(i)}$ in the former case and $w_j^{(i)}$ in the latter case. Via the exponential map $w \mapsto z = e^w$ we identify  $N_{\ud{e}}^{(i)}$ with a sequence of strips
\beq\label{eqn64}
S_{\ud{e}}^{(i)} =: \big( p^{(i)}_{\ud{e}}, q_{\ud{e}}^{(i)} \big) \times [0, \pi].
\eeq
Define a sequence of M\"obius transformations $\psi_{\ud{e}}^{(i)}: S_{\ud{e}}^{(i)} \to {\mb D}$ by 
\beq\label{eqn65}
\psi_{\ud{e}}^{(i)} = z_{\ud{e}}^{(i)} \frac{{\bm i} - d_{\ud{e}}^{(i)} e^w}{ {\bm i} + d_{\ud{e}}^{(i)} e^w}.
\eeq
Denote by ${\bf v}_{\ud e}^{(i)}$ the pull-back of ${\bf v}^{(i)}$ via $\psi_{\ud{e}}^{(i)}$. For $\tau>0$, denote
\beqn
S^{(i)}_{\ud{e},\tau}:= \big[ p_{\ud{e}}^{(i)} + \tau, q_{\ud{e}}^{(i)} - \tau \big] \times \big[ 0, \pi \big].
\eeqn

\begin{defn}\label{defn41}
The package of objects 
\beqn
{\mc I}: = \Big\{ {\bm s}_0, {\mc T}_0, {\mc C}_0, (W_\alpha^0)_{v_\alpha \in V({\mc T}_0)}, (m(w_\alpha^0))_{w_\alpha^0 \in W_\alpha^0}, (d_\alpha^{(i)}), (\varphi_\alpha^{(i)}), (g_\alpha^{(i)}) \Big\}
\eeqn
satisfying the conditions listed in Proposition \ref{prop40} is called an {\bf induction package}. We say that ${\mc I}$ is {\bf exhaustive} (resp. {\bf semi-exhaustive}) at nodes, if for any $\ud{e} \in E(\ud{\mc T}{}_0) \cup E_\infty(\ud{\mc T}{}_0)$, we have
\beqn
\lim_{\tau \to +\infty} \limsup_{i \to \infty} E \big( ( \psi_{\ud{e}}^{(i)} )^* {\bf v}^{(i)}; S_{\ud{e},\tau}^{(i)} \big) = 0
\eeqn
\beq\label{eqn66}
\Big( {\rm resp.}\ \lim_{\tau\to +\infty} \limsup_{i \to \infty} \big( \sup_{S_{\ud{e},\tau}^{(i)}} e^{(i)} \big) < \infty. \Big)
\eeq
Here $e^{(i)}$ is the energy density function of ${\bf v}_{\ud e}^{(i)}$ with respect to the flat metric on the strip.
\end{defn}

By the energy quantization property, it is easy to see that being exhaustive at nodes implies being semi-exhaustive at nodes. 

Based on the induction package ${\mc I}_0$ constructed by Proposition \ref{prop40}, we are going to construct induction packages which are semi-exhaustive and exhaustive at nodes. The induction process is described as follows. 

For each $\ud{e}\in E(\ud{\mc T}{}_0) \cup E_\infty(\ud{\mc T}{}_0)$, consider the limit in \eqref{eqn66}. If it is finite then we do nothing; if it is infinite, then we can find a subsequence (still indexed by $i$) and a sequence $w^{(i)}:= (s^{(i)}, t^{(i)}) \in S_{\ud{e}}^{(i)}$ such that 
\beqn
\lim_{i \to \infty} e^{(i)}(w^{(i)}) = \infty,\ \lim_{i \to \infty} \min \left\{ | s^{(i)} - p_{\ud{e}}^{(i)}|, |s^{(i)} - q_{\ud{e}}^{(i)} |\right\} = \infty.
\eeqn
Then choose a sequence $l_i\to +\infty$ such that $ s^{(i)} - l_i \geq p_{\ud{e}}^{(i)}$ and $s^{(i)} + l_i \leq q_{\ud{e}}^{(i)}$. Consider the sequence of half annuli
\beqn
N^{(i)}:= \big\{ z \in {\mb H}\ |\  e^{-l_i} < |z| < e^{l_i} \big\}.
\eeqn
Consider the sequence of M\"obius transformations $\varphi_{\ud{e}}^{(i)}: N^{(i)} \to {\mb D}$ defined by 
\beq\label{eqn67}
\varphi^{(i)}_{\ud{e}, N} (w) = z_{\ud{e}}^{(i)} \frac{{\bm i} - d_{\ud{e}}^{(i)} e^{s^{(i)}}  w}{{\bm i} + d_{\ud{e}}^{(i)} e^{s^{(i)}} w} .
\eeq

Now the sequence $N^{(i)}$ exhausts ${\mb H}^*:= {\mb H}\smallsetminus \{0\}$. We consider the sequence 
\beqn
{\bf v}_{\ud{e}, N}^{(i)}:= ( \varphi_{\ud{e}, N}^{(i)} )^* {\bf v}^{(i)}
\eeqn
on $N^{(i)}$. We have the following two possibilities. 

{\bf I.} If ${\mf s}_{\ud{e}} > 1$, then $\lambda_i d_{\ud{e}}^{(i)} e^{s^{(i)}} \to \infty$. Then the sequence of pull-back area forms $( \varphi_{\ud{e}, N}^{(i)})^* \lambda_i^2 \nu_0$ blows up uniformly on compact subsets of ${\mb H}^*$. Since the total energy of ${\bf v}^{(i)}$ is bounded, so is that of ${\bf v}_{\ud{e},N}^{(i)}$. Then by Theorem \ref{thm20}, we can find a subsequence (still denoted by $i$), a finite subset $W_{\ud{e}} \subset {\mb H}^*$, such that 
\begin{enumerate}
\item ${\bf v}_{\ud{e},N}^{(i)}$ converges on ${\mb H}^* \smallsetminus W_{\ud{e}}$ to a holomorphic strip $\bar{u}_{\ud{e}}$ in $(\bar{X}, \bar{L})$.

\item For each $w \in W_{\ud{e}}$, we have
\beq\label{eqn68}
m(w):= \lim_{\epsilon \to 0} \lim_{i \to \infty} E ( {\bf v}_{\ud{e},N}^{(i)}; B_\epsilon(w) \cap {\mb H}^* ) > 0.
\eeq
\end{enumerate}

{\bf II.} If ${\mf s}_{\ud{e}} < 1$, then $\lambda_i d_{\ud{e}}^{(i)} e^{s^{(i)}} \to 0$. Then the sequence of pull-back area forms $( \varphi^{(i)}_{\ud{e}} )^* \lambda_i^2 \nu_0$ converges to zero uniformly on any compact subset of ${\mb H}^*$. Then in a similar way as the above case, we can find a subsequence (still denoted by $i$), and a finite subset $W_{\ud{e}} \subset {\mb H}^*$ such that 
\begin{enumerate}
\item There exists a sequence of gauge transformations $g_{\ud{e}}^{(i)}: {\mb H}^* \to G$ such that the sequence $( g_{\ud{e}}^{(i)} )^* {\bf v}_{\ud{e},N}^{(i)}$ converges modulo bubbling on ${\mb H}^* \smallsetminus W_{\ud{e}}$ to a holomorphic strip $u_{\ud{e}}$ in $(X, L)$.

\item For each $w\in W_{\ud{e}}$, \eqref{eqn68} holds. 
\end{enumerate}

In either case, above, we can construct a new induction package ${\mc I}$ by doing the following to the original ${\mc I}_0$. 

\begin{enumerate}
\item We choose the subsequence ${\bm s}_1$ of ${\bm s}_0$ as we did;

\item We do a type-$\ud{\mc{GD}}$ growth to $({\mc T}_0, \ud{\mc T}{}_0, {\mf s}_0)$ at the edge $\ud{e}$, obtaining a new tree $({\mc T}, \ud{\mc T}, {\mf s})$ (see Definition \ref{defn59}), where the new vertex is temporarily denoted by $v_{\ud\gamma}$. The map $\iota$ automatically extends to the new tree and we set $I_{\ud\gamma} = I_{\ud{e}'}$. 

\item To the collection ${\mc C}_0$, we add a new object corresponding to the new vertex $\ud\gamma$: in the case ({\bf I}) above, we add $\bar{u}_{\ud\gamma} = \bar{u}_{\ud{e}}$; in the case ({\bf II}) above, we add $u_{\ud\gamma} = u_{\ud{e}}$ and we associate to $\ud\gamma$ the sequence of gauge transformations $g{}_{\ud{e}}^{(i)}$.

\item If $\ud{e}  = e_{\ud\alpha \ud\alpha'} \in E({\mc T}_0)$, then we remove the node $z_{\ud\alpha \ud\alpha'}$ and for the two new edges $e_{\ud\alpha \ud\gamma}$ and $e_{\ud\gamma \ud\alpha'}$, define $z_{\ud\alpha \ud\gamma} = 0$, $z_{\ud\gamma \ud\alpha'} = z_{\ud\alpha \ud\alpha'}$. If $\ud{e} = e_j$ with $\iota(j) = \ud\beta$, then for the new edge $e_{\ud\gamma \ud{e}'}$, define $z_{\ud\gamma \ud{e}'} = w_j^0$ and define $w_j = 0$. 

\item To the new vertex $v_{\ud\gamma}$, we associate the subset $W_{\ud\gamma} \subset \Sigma_{\ud\gamma} \smallsetminus Z_{\ud\gamma} = {\mb H}^*$ which consists of the bubbling points; and to each $w \in W_{\ud\gamma}$, we associate the number $m(w)$ given by \eqref{eqn68}.

\item To the new vertex $v_{\ud\gamma}$, we associate the sequence of M\"obius transformations $\varphi_{\ud\gamma}^{(i)} = \varphi_{\ud{e}, N}^{(i)}$ defined by \eqref{eqn67} and the sequence of numbers $d_{\ud\gamma}^{(i)} =d_{\ud{e}}^{(i)} e^{s^{(i)}}$. 

\end{enumerate}

It is routine to check that the new package of objects gives an induction package. The above operation can be done inductively and the induction process stops at finite time because of energy quantization. Then we prove
\begin{prop}\label{prop44}
For any induction package 
\beqn
{\mc I}_0 = \left\{ {\bm s}_0, {\mc T}_0, {\mc C}_0, (W_\alpha^0)_{v_\alpha \in V({\mc T}_0)}, (m(w_\alpha^0))_{w_\alpha^0 \in W_\alpha^0}, (d_\alpha^{(i)}), (\varphi_\alpha^{(i)}), (g_\alpha^{(i)}) \right\}
\eeqn
which is not semi-exhaustive at nodes, we can construct an induction package
\beqn
{\mc I}_1 : = \left\{ {\bm s}_1, {\mc T}_1, {\mc C}_1, (W_\alpha^1)_{v_\alpha \in V({\mc T}_1)}, (m(w_\alpha^1))_{w_\alpha^1 \in W_\alpha^1}, (d_\alpha^{(i)}), (\varphi_\alpha^{(i)}),  (g_\alpha^{(i)}) \right\}
\eeqn
which is semi-exhaustive at nodes such that ${\bm s}_1$ is a subsequence of ${\bm s}_0$, ${\mc T}_1$ is obtained from ${\mc T}_0$ by finitely many type-$\ud{\mc{GD}}$ growths (see Definition \ref{defn59}), and objects in ${\mc I}_1$ labelled by the old vertices coincide with the corresponding objects in ${\mc I}_0$.
\end{prop}

\subsection{Constructing connecting bubbles}

Now suppose we have finished the induction described in the previous subsection. Then, for each boundary nodal point (or marked point) represented by the edge $\ud{e} \in E(\ud{\mc T}{}_1) \cup E_\infty({\mc T}_1)$, if we consider the sequence of strips $S_{\ud{e}}^{(i)}$ as in \eqref{eqn64} with M\"obius transformations $\varphi_{\ud{e}}^{(i)}$ defined in \eqref{eqn65}, then for the sequence
\beqn
{\bf v}_{\ud e}^{(i)}:= ( \varphi_{\ud{e}}^{(i)} )^* {\bf v}^{(i)},
\eeqn
the limit in \eqref{eqn66} is finite. However, the solutions on the sequence of strips may have Floer-Morse type degeneration. Denote
\beqn
\wt\updelta:= \wt\updelta (L, L) = \min\{ \updelta(L, L), \updelta'(L, L) \}. 
\eeqn
Here $\updelta$ is the constant given by Lemma \ref{lemma51}, and $\updelta'$ is the constant which appeares in Po\'zniak's isoperimetric inequality (Lemma \ref{lemma50}). It is easy to see that we have the following.
\begin{prop}\label{prop45}
There exist a subsequence (still denoted by $i$) and sequences of numbers
\beqn
s_1^{(i)}, \ldots, s_m^{(i)} \in  \big( p_{\ud{e}}^{(i)}, q_{\ud{e}}^{(i)} \big),\ s_1^{(i)}< s_2^{(i)}< \cdots < s_m^{(i)},\ (m \geq 0)
\eeqn
satisfying
\begin{enumerate} 
\item For each $l \in \{1, \ldots, m\}$ we have
\beqn
\lim_{i\to \infty} \big( s_l^{(i)} - p_{\ud{e}}^{(i)} \big)  = \lim_{i\to \infty} ( q_{\ud{e}}^{(i)} - r_l^{(i)} ) = +\infty;
\eeqn
and $l_1 < l_2 \Longrightarrow \displaystyle \lim_{i \to \infty} ( s_{l_2}^{(i)} -  s_{l_1}^{(i)} )= \infty$.

\item For each $l \in \{1, \ldots, m\}$ and $i$, there exists $t_l^{(i)} \in [0, \pi]$ such that
\beq\label{eqn69}
\lim_{i \to \infty} e_i(  s_l^{(i)}, t_l^{(i)} ) \geq \big( \frac{1}{2} \updelta \big)^2.
\eeq
Here $e_i: S_{\ud{e}}^{(i)} \to {\mb R}_+ \cup \{0\}$ is the energy density function of ${\bf v}_{\ud e}^{(i)}$ with respect to $ds dt$.

\item For any sequence $( s^{(i)}, t^{(i)}) \in S_{\ud{e}}^{(i)}$ satisfying
\beqn
\lim_{i \to \infty} \min \left\{ |s^{(i)} - p_{\ud{e}}^{(i)}|, |s^{(i)} - q_{\ud{e}}^{(i)}| \right\} = \lim_{i \to \infty} | s^{(i)} - s_l^{(i)}| = \infty,\ \forall l = 1, \ldots, m,
\eeqn
we have
\beq\label{eqn610}
\limsup_{i \to \infty} e_i( s^{(i)}, t^{(i)}) < \big( \frac{1}{2} \updelta \big)^2.
\eeq
\end{enumerate}
\end{prop}

Moreover, if we denote 
\beq\label{eqn611}
\varphi_{\ud{e},l}^{(i)}(w) = z_{\ud{e}}^{(i)} \frac{ {\bm i} - d_{\ud{e}}^{(i)} e^{s_l^{(i)}} w}{{\bm i} + d_{\ud{e}}^{(i)} e^{s_l^{(i)}} w}, 
\eeq
denote $s_0^{(i)} = p_{\ud{e}}^{(i)}$, $s_{m+1}^{(i)} = q_{\ud{e}}^{(i)}$ and for $l = 1, \ldots, m$, denote
\beq\label{eqn612}
\Sigma_{\ud{e},l}^{(i)}:= \left\{ z \in {\mb H}\ |\   e^{s_{l-1}^{(i)} - s_l^{(i)}}  \leq |z| \leq e^{s_{l+1}^{(i)} - s_l^{(i)}} \right\}.
\eeq
which exhausts ${\mb H}^*$. We consider ${\bf v}_{\ud{e},l}^{(i)}:= ( \varphi_{\ud{e},l}^{(i)} )^* {\bf v}^{(i)}$.

Now, if ${\mf s}(\ud{e}') = \infty$, then $( \varphi_{\ud{e},l}^{(i)} )^* \lambda_i^2 \nu_0$ blows up uniformly on ${\mb H}^*$. Then there is a subsequence (still indexed by $i$) such that for each $l = 1, \ldots, m$, the sequence ${\bf v}_{\ud e, l}^{(i)}$ converges (modulo gauge) to a holomorphic strip $\bar{u}_{\ud{e},l}$ in $(\bar{X}, \bar{L})$. We may assume that $t_{\ud{e},l}^{(i)}$ converges to $t_{\ud{e},l} \in [0, \pi]$. Then \eqref{eqna6} and Proposition \ref{prop18} implies that the limit $\bar{u}_{\ud{e},l}$ is nonconstant and then the energy of $\bar{u}_{\ud{e},l}$ is no less than some positive constant which only depends on $(\bar{X}, \bar{L})$. 

In this case, we can do the following update to the induction package. 
\begin{enumerate}
\item Choose the subsequence ${\bm s}_2$ of ${\bm s}_1$ we just found.

\item We do type-$\ud{\mc{GD}}$ growth to the tree $({\mc T}_1, \ud{\mc T}{}_1, {\mf s}_1)$ at the edge $\ud{e}$ (see Definition \ref{defn59}), but instead of inserting just one intermediate vertex, we insert $m$ intermediate vertices, temporarily labelled by $\ud\gamma{}_1, \ldots, \ud\gamma{}_m$ and the edge $\ud{e}$ breaks into $\ud{e}{}_0, \ldots, \ud{e}{}_m$. Thus we obtained a new based colored rooted tree $({\mc T}_2, \ud{\mc T}{}_2, {\mf s}_2)$, where the new vertices are contained in $V_\infty({\mc T}_2)$. 

\item To the collection ${\mc C}_1$, we add $m$ new objects corresponding to the new vertices: for each $\ud\gamma{}_l$ ($l= 1, \ldots, m$), we add $\bar{u}_{\ud\gamma{}_l} = \bar{u}_{\ud{e}, l}$; we remove the node $z_{\ud{e}}$ from the collection and instead, for the new edges $\ud{e}{}_l$, define $z_{\ud{e}{}_m} = \cdots = z_{\ud{e}{}_1} = 0$ and $z_{\ud{e}{}_0} = z_{\ud{e}}$.

\item To each new vertex $\ud\gamma{}_l$, the subset $W_{\ud\gamma{}_l} = \emptyset$; we associate the sequence of M\"obius transformations $\varphi_{\ud\gamma{}_l}^{(i)} = \varphi_{\ud{e}, l}^{(i)}$ given by \eqref{eqn611}, and the sequence of numbers $d_{\ud\gamma{}_l}^{(i)}:= d_{\ud{e}}^{(i)} e^{s_l^{(i)}}$.
\end{enumerate}

On the other hand, if ${\mf s}(\ud{e}')  \leq 1$, then for the sequence ${\bf v}_{\ud{e},l}^{(i)}$, there exists a sequence of gauge transformations $g_{\ud{e},l}^{(i)}: {\mb H}^* \to G$ such that $( g_{\ud{e},l}^{(i)} )^* {\bf v}_{\ud{e},l}^{(i)}$ converges to $(0, u_{\ud{e},l})$ uniformly with all derivatives on compact subsets of ${\mb H}^*$. Here $u_{\ud{e},l}$ is a holomorphic strip in $(X, L)$ with positive energy. Then we can update the induction package accordingly. We leave the details to the reader.

We carry out the above process for every edge $\ud{e} \in E(\ud{\mc T}{}_1)$. Then we obtain a new induction package, denoted by
\beqn
{\mc I}_2: = \Big\{ {\bm s}_2, {\mc T}_2, {\mc C}_2, (W_\alpha^2)_{v_\alpha \in V({\mc T}_2)}, (m(w_\alpha^2))_{w_\alpha^2 \in W_\alpha^2}, (d_\alpha^{(i)}), (\varphi_\alpha^{(i)}), (g_\alpha^{(i)}) \Big\}.
\eeqn

\subsection{Proof that no energy is lost at nodes and markings}

We would like to show that the induction package we just constructed is exhaustive at nodes. It suffices to show that, for any edge $\ud{e} \in E(\ud{\mc T}{}_2) \cup E_\infty(\ud{\mc T}{}_2)$, for the sequence of necks $S_{\ud{e}}^{(i)}$, we have
\beq\label{eqn613}
\lim_{\tau\to +\infty} \limsup_{i \to \infty} E ( {\bf v}^{(i)}; S_{\ud{e},\tau}^{(i)} ) = 0.
\eeq
If ${\mf s}(\ud{e}') > 1$, then the sequence of area forms $\lambda_i^2 ( \varphi_{\ud{e}}^{(i)} )^* \nu_0$ diverges uniformly to infinity on any $S_{\ud{e},\tau}^{(i)}$. If we write $( \varphi_{\ud{e}}^{(i)} )^* {\bf v}^{(i)} = {\bf v}_{\ud{e}}^{(i)} = (B_{\ud e}^{(i)}, u_{\ud e}^{(i)})$ and transform $B_{\ud{e}}^{(i)}$ into temporal gauge so that $B_{\ud{e}}^{(i)} = \psi^{(i)} dt$, and if we write
\beq\label{eqn614}
\gamma_s^{(i)}(t) = v_{\ud{e}}^{(i)}(s, t),\ \psi_s^{(i)}(t) = \psi^{(i)}(s, t),\ s \in \big( p_{\ud{e}}^{(i)}, q_{\ud{e}}^{(i)}  \big),
\eeq
then \eqref{eqn610} implies that for $\tau$ large enough and $s \in \big( p_{\ud{e}}^{(i)} + \tau, q_{\ud{e}}^{(i)}- \tau \big)$, $( \gamma_s^{(i)}, \psi_s^{(i)} ) \in {\mc P}_{\updelta}$. Here ${\mc P}_{\updelta}$ denotes the set of paths $(x(t), \eta(t))$ satisfying \eqref{eqna8}. By the same argument as used in the proof of Proposition \ref{prop57}, \eqref{eqn613} holds. 

On the other hand, if ${\mf s}(\ud{e}') \leq 1$, we cannot use the argument of Proposition \ref{prop57} because the sequence of area forms $\lambda_i^2 (\psi_{\ud{e}}^{(i)})^* \nu_0$ are not uniformly bounded from below. We write
\beqn
\nu_{\ud{e}}^{(i)} = \sigma_{\ud{e}}^{(i)} dsdt = ( \varphi_{\ud{e}}^{(i)} )^* \lambda_i^2 \nu_0.
\eeqn
By \eqref{eqn62}, \eqref{eqn63} and by estimating the derivative of $\varphi_{\ud{e}}^{(i)}$ we see that  
\beqn
\sigma_{\ud{e}}^{(i)} \leq (16 \lambda_i d_{\ud{e}}^{(i)} )^2 e^{2s} \leq 16 ( \lambda_i d_{\ud{e}'}^{(i)} )^2 e^{ - 2 ( q_{\ud{e}}^{(i)} - s) }.
\eeqn
Since $\lambda_i d{}_{\ud{e}'}^{(i)}$ is bounded, for each $\epsilon > 0$, we choose $a^{(i)} (\epsilon) < q_{\ud{e}}^{(i)}$ such that for $s \leq a^{(i)} (\epsilon )$, 
\beq\label{eqn616}
\sigma^{(i)} \leq 16 ( \lambda_i d_{\ud{e}'}^{(i)} )^2 e^{-2( q_{\ud{e}}^{(i)} - a^{(i)} (\epsilon) )} e^{-2 (a^{(i)} (\epsilon) - s)}  < \epsilon e^{-2 (a^{(i)} (\epsilon)  -s)}.
\eeq

Denote $s_0: = q_{\ud{e}}^{(i)} - a^{(i)} (\epsilon)$, which can be made only depend on $\epsilon$ but not on $i$. Then we transform $(B_{\ud{e}}^{(i)}, u_{\ud{e}}^{(i)})$ into temporal gauge and with notations similar to \eqref{eqn614} with an extra requirement that $\psi^{(i)} ( q_{\ud{e}}^{(i)} - s_0, t) \equiv 0$. Then for $i$ large and $s \in [ p_{\ud{e}}^{(i)} + s_0, q_{\ud{e}}^{(i)}- s_0 ]$, we have
\beq\label{eqn617}
 \psi^{(i)} (s, t) = - \int_s^{q_{\ud{e}}^{(i)}- s_0} \frac{\partial \psi^{(i)}}{\partial \tau} (\tau, t) d\tau = \int_s^{q_{\ud{e}}^{(i)}- s_0 } \sigma^{(i)}(\tau, t) \mu(u^{(i)} (e^{\tau + {\bm i} t})) d\tau.
\eeq
Since $|\mu(u^{(i)})|$ is uniformly bounded by a constant $c_1$, by \eqref{eqn616} and \eqref{eqn617},
\beqn
\big| \psi^{(i)} (s, t) \big| \leq c_1 \int_s^{q_{\ud{e}}^{(i)} - s_0}  \epsilon e^{-2( s_0 - \tau)} d\tau \leq \frac{c_1 \epsilon }{2} \big( 1- e^{-2(q_{\ud{e}}^{(i)} - s_0 - s)} \big) \leq \frac{c_1 \epsilon}{2}.
\eeqn
Notice that $|{\mc X}_{\psi_s^{(i)}}|$ is controlled by $|\psi_s^{(i)}|$. Then there exists $c_2>0$ such that 
\beqn
\big|\dot\gamma_s^{(i)} \big| \leq \big| \dot\gamma_s^{(i)} + {\mc X}_{\psi_s^{(i)}} (\gamma_s^{(i)} ) \big| + \big| {\mc X}_{\psi_s^{(i)} }(\gamma_s^{(i)} ) \big| \leq \frac{1}{2} \updelta' +  c_2\epsilon.
\eeqn
So for $\epsilon$ small enough, $|\dot\gamma_s^{(i)}| < \updelta'$. Recall that $\updelta'$ is the constant associated with the cleanly intersecting pair $(L, L)$ in $X$ in Po\'zniak's isoperimetric inequality Lemma \ref{lemma50}. Then we can define the local action ${\bm a}(\gamma{}_s^{(i)})$ by \eqref{eqna6}, which satisfies \eqref{eqna7}. We also consider the equivariant local action given by
\beqn
{\bm A} \big( \wt\gamma_s^{(i)} \big):= {\bm A} \big( \gamma_s^{(i)} ,\psi_s^{(i)} \big) = {\bm a} \big( \gamma_s^{(i)} \big) + \int_0^\pi \big\langle \mu \big( \gamma_s^{(i)} (t) \big), \psi_s^{(i)} (t) \big\rangle dt.
\eeqn
Then for appropriate constants $c_3, \cdots$ (omitting the index $i$ to save space)
\beq\label{eqn618}
\begin{split}
&\ {\bm A} \big( \gamma_s^{(i)}, \psi_s^{(i)} \big) \\
\leq &\ {\bm a} \big( \gamma_s  \big)  + \int_0^\pi \big| \mu \big( \gamma_s (t) \big) \big| \cdot \big| \psi_s (t) \big| dt \\
 \leq &\ {\bm c}' \int_0^\pi \big| \dot\gamma_s (t) \big|^2 dt + c_3\epsilon \\
\leq &\  2{\bm c}' \int_0^\pi \Big( \big| \dot\gamma_s + {\mc X}_{\psi_s}(\gamma_s) \big|^2 + \big| {\mc X}_{\psi_s}(\gamma_s ) \big|^2  +  \sigma (s, t) \big| \mu \big( \gamma_s \big) \big|^2 \Big) dt + c_4 \epsilon \\
\leq &\ 2{\bm c}' \int_0^\pi \Big( \big| \dot\gamma_s + {\mc X}_{\psi_s}(\gamma_s) \big|^2 + \sigma (s, t) \big|\mu \big( \gamma_s (t) \big) \big|^2 \Big) dt  + c_5\epsilon.
\end{split}
\eeq
Notice that 
\beqn
\int_0^\pi \Big( \big| \dot\gamma_s^{(i)} + {\mc X}_{\psi_s^{(i)}}(\gamma_s^{(i)}) \big|^2 + \sigma^{(i)} (s, t) \big|\mu ( \gamma_s^{(i)} ) \big|^2 \Big) dt = \int_0^\pi e_i (s, t) dt.
\eeqn
For $s \geq s_0$, we define 
\beqn
F^{(i)} (s) = E \Big( {\bf v}_{\ud{e}}^{(i)}; \big[ p_{\ud{e}}^{(i)} + s, q_{\ud{e}}^{(i)} -s \big] \times \big[ 0, \pi \big] \Big).
\eeqn
Then \eqref{eqn618} implies 
\beqn
\Big| {\bm A} \Big( \wt\gamma_{p_{\ud{e}}^{(i)} + s}^{(i)} \Big) - {\bm A} \Big( \wt\gamma_{q_{\ud{e}}^{(i)} - s}^{(i)} \Big) \Big|  \leq  - 2{\bm c}' \frac{d }{ds} F^{(i)} (s)  + c_5 \epsilon.
\eeqn
On the other hand, we know that the left-hand-side above is equal to $F^{(i)}(s)$. Therefore
\beqn
 \frac{d}{ds} \left( \exp \left( \frac{s}{2 {\bm c}' } \right) F^{(i)} (s) \right) \leq  c_5 \epsilon \frac{d}{ds} \left( \exp \left( \frac{s }{ 2 {\bm c}'}  \right) \right).
\eeqn
and therefore
\beqn
\begin{split}
F^{(i)} (s) \leq &\ \exp \Big( - \frac{s }{2 {\bm c}' } \Big) \Big( \exp \Big( \frac{ q_{\ud{e}}^{(i)} - s_0 }{2{\bm c}' } \Big) F^{(i)}(q_{\ud{e}}^{(i)} - s_0 )  +  c_5 \epsilon \exp \Big( \frac{s}{2 {\bm c}' }  \Big) \Big) \\
\leq &\ \exp \Big( - \frac{1}{2 {\bm c}' } ( s - (q_{\ud{e}}^{(i)} - s_0))\Big)   + c_5 \epsilon. 
\end{split}
\eeqn
Let $\epsilon$ be arbitrarily small, we see that the limit \eqref{eqn613} is zero. 

Therefore, the induction package constructed in the previous step is exhaustive at nodes. In summary, we have proved the following. 
\begin{prop}\label{prop46}
For any induction package
\beqn
{\mc I}_0: = \left\{ {\bm s}_0, {\mc T}_0, {\mc C}_0, (W_\alpha^0)_{v_\alpha \in V({\mc T}_0)}, (m(w_\alpha^0))_{w_\alpha \in W_\alpha^0}, (d_\alpha^{(i)}), (\varphi_\alpha^{(i)}), (g_\alpha^{(i)}) \right\}
\eeqn
which is not exhaustive at nodes, we can construct an induction package
\beqn
{\mc I}: = \left\{ {\bm s}, {\mc T}, {\mc C}, (W_\alpha)_{v_\alpha \in V({\mc T})}, (m(w_\alpha))_{w_\alpha \in W_\alpha}, (d_\alpha^{(i)}), (\varphi_\alpha^{(i)}), (g_\alpha^{(i)}) \right\}
\eeqn
which is exhaustive at nodes such that, ${\bm s}$ is a subsequence of ${\bm s}_0$, the tree ${\mc T}$ is obtained from ${\mc T}_0$ by doing type-$\ud{\mc{GD}}$ growths for finitely many times, and objects of ${\mc I}$ labelled by the old vertices coincide with the corresponding objects in ${\mc I}_0$.
\end{prop}

\subsection{Apply soft rescaling}

So far we have constructed an induction package ${\mc I}$. Notice that so far in ${\mc I}$ the tree ${\mc T}$ is actually a colored rooted ribbon tree, i.e., ${\mc T} = \ud{\mc T}$. By property (vi) of the definition of induction package (see Proposition \ref{prop40}) and the fact that ${\mc I}$ is exhaustive at nodes, one has
\beqn
\lim_{i \to \infty} E ( {\bf v}^{(i)}) = \sum_{v_\alpha \in V({\mc T})} E({\mc C}_\alpha) + \sum_{v_\alpha \in V({\mc T})} \sum_{w_\alpha \in W_\alpha}  m(w_\alpha).
\eeqn
To continue to find the limit, we apply the soft rescaling results (Proposition \ref{prop36}--\ref{prop39}) to the bubbling points $w_\alpha$ and Theorem \ref{thm16}.

\subsubsection{Growing the roots and stems}

First, for any $v_\alpha \in V_\infty({\mc T})$ and $w_\alpha \in W_\alpha$, by the previous construction, we can choose a subsequence (still indexed by $i$), such that the subsequence ${\bf v}_\alpha^{(i)}:= ( \varphi_\alpha^{(i)} )^* {\bf v}^{(i)}$ has energy concentration of either type ${\mc R}$ or type ${\mc S}$ at $w_\alpha$, with respect to the sequence $\lambda_i d_\alpha^{(i)}$. 

{\bf I.} If $w_\alpha \in {\rm Int} \Sigma_\alpha$ and the energy concentration has type ${\mc R}$ with respect to $\lambda_i d_\alpha^{(i)}$, then we can find a subsequence (still indexed by $i$), a sequence of M\"obius transformations $\varphi_{w_\alpha}^{(i)}: {\mb C} \to \Sigma_\alpha$ and a finite subset $W \subset {\mb C}$ satisfying the conditions listed in Proposition \ref{prop36}.

\begin{enumerate}

\item We replace the subsequence by the new subsequence we just found (still indexed by $i$).

\item By a type-$\mc{GR}$ growth at the vertex $v_\alpha$, we obtain a new tree ${\mc T}_1$ with a new vertex called $w_\alpha$ with new edge $e_{w_\alpha \alpha}$. 

\item To the collection ${\mc C}$ we add the following objects.
\begin{itemize}
\item For the new vertex $w_\alpha$, add $\bar{u}_{w_\alpha}: {\mb C} \to \bar{X}$ which is a holomorphic sphere and is the limit of $( \varphi_{w_\alpha}^{(i)} )^* {\bf v}_\alpha^{(i)}$;

\item For the new edge $e_{w_\alpha \alpha}$, we give $z_{w_\alpha \alpha} = w_\alpha \in {\rm Int} \Sigma_\alpha$.

\item For the new vertex $w_\alpha$, we add the set $W_{w_\alpha} = W$ and to each $w_j \in W_{w_\alpha}$, we associate the mass $m(w_j)$ given by \eqref{eqn52}.

\item For the new vertex $w_\alpha$, we choose the associated sequence of M\"obius transformations by $\varphi_{w_\alpha}^{(i)} = \varphi_\alpha^{(i)} \circ \varphi_{w_\alpha}^{(i)}$, and the sequence $d_{w_\alpha}^{(i)} = \tau_i d_\alpha^{(i)}$ (here $\tau_i$ is the sequence defined in \eqref{eqn53}). 

\end{itemize}
\end{enumerate}
It is routine to check that it is indeed an induction package. 

There are three remaining cases listed as follows. 

{\bf II.} $w_\alpha \in {\rm Int} \Sigma_\alpha$ with type ${\mc S}$ energy concentration with respect to $\lambda_i d_\alpha^{(i)}$

{\bf III.} $w_\alpha \in \partial \Sigma_\alpha$ with type ${\mc R}$ energy concentration with respect to $\lambda_i d_\alpha^{(i)}$.

{\bf IV.}  $w_\alpha \in \partial \Sigma_\alpha$ with type ${\mc S}$ energy concentration with respect to $\lambda_i d_\alpha^{(i)}$.

It is also routine to write down how to grow the tree and update the induction package as we did in the first case. We omit the details. Then we can repeat the process for finitely many times until for all $\alpha \in V_\infty({\mc T})$, $W_\alpha = \emptyset$. 

\subsubsection{Growing the flowers}

Now for any $v_\alpha \in V_0({\mc T}) \cup V_1({\mc T})$, we have the sequences $g_\alpha^{(i)}$ and $\varphi_\alpha^{(i)}$. By construction the sequence of area forms $\nu_\alpha^{(i)}:= \lambda_i^2 ( \varphi_\alpha^{(i)} )^* \nu_0$ converges to a constant multiple of the standard area form on $\Sigma_\alpha$. Without loss of generality, we assume that the limit $\nu_\alpha$ is either one or zero times of the standard area form, depending on whether $v_\alpha \in V_1({\mc T})$ or $v_\alpha \in V_0({\mc T})$. We also know that the sequence ${\bf v}_\alpha^{(i)}: =( g_\alpha^{(i)} )^* ( \varphi_\alpha^{(i)} )^* {\bf v}^{(i)}$ converges modulo bubbling on $\Sigma_\alpha \smallsetminus W_\alpha$ to a vortex ${\bf v}_\alpha \in \wt{\mc M}(\Sigma_\alpha; X, L)$. We have the following two situations.

{\bf I.} If $w_\alpha \in {\rm Int} \Sigma_\alpha$ (resp. $w_\alpha \in \partial \Sigma_\alpha$), then by Theorem \ref{thm16}, we can construct a stable holomorphic sphere (resp. disk) $\big( (u_\beta)_{v_\beta \in V({\mc T}(w_\alpha))}, (z_{\beta \beta'})_{e_{\beta\beta'\in E({\mc T}(w_\alpha))}} \big)$ in $X$ modelled on a branch (resp. based branch) ${\mc T}(w_\alpha)$, and, for each $v_\beta \in V({\mc T}(w_\alpha))$, a sequence of M\"obius transformations $\hat\varphi_\beta^{(i)}: \Sigma_\beta \to \Sigma_\alpha$ and a sequence of gauge transformations $\hat{g}_\beta^{(i)}: \Sigma_\beta \to G$ such that 
\begin{itemize}
\item For any $v_\beta \in V({\mc T}(w_\alpha))$, the sequence $( \hat{g}_\beta^{(i)} )^* ( \hat\varphi_\beta^{(i)} )^* {\bf v}_\alpha^{(i)}$ converges to $(0, u_\beta)$ uniformly on any compact subset of $\Sigma_\beta \smallsetminus Z_\beta$.

\item For each $e_{\beta \beta'} \in E({\mc T}(w_\alpha))$, $( \hat\varphi_{\beta'}^{(i)} )^{-1} \hat\varphi_\beta^{(i)}$ converges uniformly on any compact subset to the constant $z_{\beta \beta'} \in \Sigma_{\beta'}$.

\item We have $\displaystyle \lim_{\epsilon \to 0} \lim_{i \to \infty} E ( {\bf v}_\alpha^{(i)}; B_\epsilon ( w_\alpha) ) = \sum_{v_\beta \in V({\mc T}(w_\alpha))} E(u_\beta)$. 
\end{itemize}
In this case, we do the following operations to the induction package ${\mc I}$. Moreover, we ignore the sequences of numbers $(d_\alpha^{(i)})$ and all conditions required for them in the definition of induction package. 
\begin{enumerate}
\item We choose the subsequence we just found.

\item We grow the tree ${\mc T}$ a type-$\mc{GF}$ (resp. type-$\ud{\mc{GF}}$) growth at the vertex $v_\alpha$ by attaching the branch (resp. based branch) ${\mc T}(w_\alpha)$.

\item We modify the collection ${\mc C}$ as follows.
\begin{itemize}

\item Include the components $u_\beta$ corresponding to all $v_\beta \in V({\mc T}(w_\alpha))$ and include the nodes $z_{\beta \beta'}$ for all $e_{\beta\beta'} \in E({\mc T}(w_\alpha))$.

\item Include the sequence of M\"obius transformations $\varphi_\beta^{(i)}:= \varphi_\alpha^{(i)} \circ \hat\varphi_\beta^{(i)}$ and the sequence of gauge transformations $g_\beta^{(i)}:= \hat{g}_\beta^{(i)} g_\alpha^{(i)}(w_\alpha) $ for all $v_\beta \in V({\mc T}(w_\alpha))$.

\item Remove the number $m(w_\alpha)$ from the collection of masses. 

\end{itemize}
\end{enumerate}

We can repeat the above process for every $v_\alpha \in V_0({\mc T}) \cup V_1({\mc T})$ and every $w_\alpha \in W_\alpha$. It stops after finitely many times and then $W_\alpha = \emptyset$ for all $\alpha \in V({\mc T})$. This finishes the construction of the stable scaled holomorphic disk as a subsequential limit of the sequence ${\bf v}^{(i)}$ with respect to $\lambda_i$, and thus finishes the proof of Theorem \ref{thm33}.

\appendix

\section{Analysis of Vortices}\label{appendixa}

In this appendix we establish several necessary estimates related to the problem and provide proofs of Theorem \ref{thm12} and Theorem \ref{thm14}. The techniques used here are standard, and all results which don't involve boundary conditions are essentially covered in the previous literature such as \cite{Cieliebak_Gaio_Mundet_Salamon_2002} and \cite{Gaio_Salamon_2005}. A new technical result here is the proof of the existence of admissible almost complex structures with respect to a $G$-Lagrangian (Lemma \ref{lemma49}).

\subsection{Preliminaries}

\subsubsection*{A neighborhood of $\mu^{-1}(0)$}

We assumed that $0$ is a regular value of $\mu$. Therefore we can fix two numbers $\updelta>0$, ${\bm m} = {\bm m}(\updelta)>0$ satisfying that for any $x \in X$ with $|\mu(u)|\leq \updelta$, the map $\xi \mapsto {\mc X}_\xi(x) \in T_x M$ is injective and
\beq
\omega( {\mc X}_\xi, J {\mc X}_\xi) \geq {\bm m} |\xi|^2.
\eeq
We fix these two numbers throughout this appendix.

\subsubsection*{$G$-invariant metrics}

Let $h$ be a $G$-invariant Riemannian metric on $X$ with respect to which $J$ is isometric. We denote by $\nabla$ the Levi-Civita connection of $h$, and $R$ the Riemannian curvature. 

Let $\Xi \subset {\mb H}$ be an open subset with coordinates $(s, t)$. For any smooth map $(u, \phi, \psi): \Xi\to X \times {\mf g} \times {\mf g}$, we denote
\beqn
v_s:= \partial_s u + {\mc X}_\phi(u), \ v_t: =\partial_t u + {\mc X}_\psi(u),\ \kappa:= \partial_s \psi - \partial_t \phi + [\phi, \psi].
\eeqn
Moreover, we have a natural covariant derivative on $u^* TX \oplus {\mf g}$ associated with $(u, \phi, \psi)$, defined as follows. For any $V \in \Gamma(u^* TX)$, we define
\beq\label{eqna2}
\nabla_{A, s} V = \nabla_s V + \nabla_V {\mc X}_\phi,\ \nabla_{A, t} V = \nabla_t V + \nabla_V {\mc X}_\psi;
\eeq
for $\xi: \Xi \to {\mf g}$, we define
\beqn
\nabla_{A, s} \xi = \partial_s \xi + \left[ \phi, \xi \right],\ \nabla_{A, t} \xi = \partial_t \xi + \left[ \psi, \xi \right].
\eeqn
By the $G$-invariance of $h$ one can check that this covariant derivative preserves the metric.

We can extend the covariant derivative to tensor fields along $u$, by Leibniz rule. We denote $v_s = \partial_s u + {\mc X}_\phi$, $v_t = \partial_t u + {\mc X}_\psi$. 
 Then we have
\begin{lemma}\label{lemma47}
If $T$ is a $G$-invariant tensor field on $X$ and $\nabla$ is the Levi-Civita connection of a $G$-invariant metric, then $\nabla_{A, s} T = \nabla_{v_s} T$. 
In particular, if $J$ is a $G$-invariant almost complex structure, then $\nabla_{A, s}J = \nabla_{v_s} J$. 
\end{lemma}
It is straightforward to extend the above result to ${\mf g}$-valued tensor fields. We denote
\beqn
\rho(v_s, V) = \nabla_{A, s} \left( d\mu \cdot J V \right) - d\mu \cdot \left(J  \nabla_{A, s} V \right),\ \rho(v_t, V) = \nabla_{A, t} \left( d\mu \cdot JV \right) - d\mu \cdot \left( J \nabla_{A, t} V \right).
\eeqn

To estimate the energy density, it is convenient to have a special type of metrics on $X$.
\begin{defn}\label{defn48}
Let $L$ be a $G$-Lagrangian of $(X, \omega, \mu)$. Let $J$ be a $G$-invariant almost complex structure. A $(J, L, \mu)$-admissible Riemannian metric is a $G$-invariant Riemannian metric $h$ on $X$ satisfying
\begin{enumerate}
\item $J$ is isometric;

\item $J(TL)$ and $TL$ are orthogonal with respect to $h$;

\item $L$ is totally geodesic with respect to $h$;

\item $T\mu^{-1}(0)$ is orthogonal to $J{\mc X}_\xi$ for all $\xi \in {\mf g}$.
\end{enumerate}
\end{defn}
In the non-equivariant case Frauenfelder \cite{Frauenfelder_disk} proved the existence of a similar type of metric for a Lagrangian submanifold satisfying (1)--(3). Here we generalize this result.

\begin{lemma}\label{lemma49}
For any $G$-Lagrangian $L$, and $G$-invariant almost complex structure $J$, there exists a $(J, L, \mu)$-admissible Riemannian metric.
\end{lemma}

\begin{proof}
By \cite[Lemma A.3]{Frauenfelder_disk}, there exists a Riemannian metric $\bar{h}$ on the symplectic quotient $\bar{X}$ satisfying (1)--(3) for the Lagrangian $\bar{L}$. Now we construct a suitable lift of $\bar{h}$ to $\mu^{-1}(0)$. Let $h_0$ be the metric on $X$ induced by $\omega$ and $J$. Let $H\subset T\mu^{-1}(0)$ be the orthogonal complement (with respect to $h_0$) of the distribution $T{\mf g}$ generated by infinitesimal $G$-actions. Then it is easy to see that $H$ is $G$-invariant, and we have  an isomorphism $H \simeq \pi^* T\bar{X}$, where $\pi: \mu^{-1}(0) \to \bar{X}$ is the projection. Then we can pull-back $\bar{h}$ to a $G$-invariant metric on $H$, and choose a $G$-invariant metric on $T{\mf g}$ such that $\langle {\mc X}_\xi, {\mc X}_\xi\rangle = |\xi|^2$. Let $h'$ be the direct sum of the two components, which is a Riemannian metric on $\mu^{-1}(0)$. 

We claim that $L$ is totally geodesic in $\mu^{-1}(0)$ with respect to $h'$. Let $\nabla'$ be the Levi-Civita connection of $h'$. It suffices to check that $(\nabla')_X Y \in TL$ for any vector fields $X, Y$ tangent to $L$. Since this is a pointwise condition, we assume that $X, Y$ are both $G$-invariant. For any $Z$ orthogonal to $TL$ in $\mu^{-1}(0)$, we have (the inner products in the following are the ones for $h'$)
\beq\label{eqna3}
\begin{split}
\langle (\nabla')_X Y, Z \rangle = &\ X \langle Y, Z \rangle - \langle Y, (\nabla')_X Z \rangle \\
= &\ - \langle Y, [ X, Z] \rangle - \langle Y, (\nabla')_Z X \rangle\\
= &\ \langle Y, [Z, X] \rangle - Z \langle Y, X \rangle + \langle (\nabla')_Z Y, X \rangle\\
= &\  \langle Y, [Z, X] \rangle - Z \langle Y, X \rangle + \langle [Z, Y], X \rangle + \langle (\nabla')_Y Z, X \rangle \\
= &\ \langle Y, [Z, X] \rangle - Z \langle Y, X \rangle  + \langle [Z, Y], X \rangle - \langle Z, (\nabla')_Y X \rangle\\
= &\ \langle Y, [Z, X] \rangle - Z\langle Y, X \rangle + \langle [Z, Y], X \rangle - \langle Z, (\nabla')_X Y \rangle.
\end{split}
\eeq
We choose $Z$ to be pull-backed from $\bar{X}$ so that $Z$ is $G$-invariant and $\pi_* Z$ is a smooth vector field on $\bar{X}$ and orthogonal to $T\bar{L}$.

If $X, Y \in \Gamma(H \cap TL) \simeq \Gamma(\pi^* T\bar{L})$, then we choose $X, Y$ to be pull-backed from $\bar{X}$ and such that $\langle X, Y \rangle$ is a constant. Then $\pi_* X, \pi_* Y$ are smooth vector fields on $\bar{X}$. Then by the definition of $h'$, we see that
\beqn
\langle Y, [Z, X]\rangle = \langle \pi_* Y, [ \pi_* Y, \pi_* Z] \rangle_{\bar{h}},\ \langle [Z, Y], X\rangle = \langle [\pi_* Z, \pi_* Y], \pi_* X \rangle_{\bar{h}}.
\eeqn
Then by the same calculation as \eqref{eqna3}, we know that the sum of the above two terms is equal to $2\langle \bar\nabla_{\pi_* X} (\pi_* Y), \pi_* Z \rangle_{\bar{h}}$, which vanishes by the totally geodesic assumption on $\bar{h}$. 

On the other hand, if $X \in \Gamma( T{\mf g})$ and $Y\in \Gamma(H \cap TL)$, then $\langle Y, X\rangle \equiv 0$. We can choose $Z$ being $G$-invariant and $[ Y, Z]$ vanishing at a point where we want to evaluate \eqref{eqna3}. Then \eqref{eqna3} vanishes at that point. Finally, if $X, Y \in \Gamma(T{\mf g})$, then we take $X = {\mc X}_\xi$, $Y = {\mc X}_\eta$ for $\xi, \eta$ constants. Then $\langle Y, X \rangle$ is a constant and $[X, Z] = [Y, Z]= 0$. So \eqref{eqna3} vanishes. 

Now we would like to extend $h'$ to a metric on $X$ which satisfies (1)--(4).

We choose local coordinates $\theta_1, \ldots, \theta_k,x_1, \ldots, x_m$ of $L$ where the first $k = {\rm dim} G$ coordinates are coordinates of $G$-orbits. Extend them to coordinates 
\beqn
\theta_1, \ldots, \theta_k, x_1, \ldots, x_m, \tau_1, \ldots, \tau_k, y_1, \ldots, y_m
\eeqn
on $X$ such that the first $k$ coordinates are still coordinates of $G$-orbits, $L$ is parametrized by $(\theta, x, 0, 0)$ and satisfying
\beqn
J \frac{\partial}{\partial \theta_i} = \frac{\partial }{\partial \tau_i},\ i=1,\ldots, k,\ J \frac{\partial }{\partial x_j} = \frac{\partial }{\partial y_j},\ j=1, \ldots, m,
\eeqn
on $L$. We remark that $\frac{\partial }{\partial y_j}$ may not be tangent to $\mu^{-1}(0)$ but $\mu^{-1}(0)$ can be parametrized as $( \theta, x, \tau(\theta, x, y), y)$ where $\tau$ satisfies $\tau(\theta, x, 0) = 0$. We write $J$ as
\beqn
J = \left( \begin{array}{cc} A(\theta, x, \tau, y) & B(\theta, x, \tau, y) \\ C(\theta, x,\tau, y) & D(\theta, x, \tau, y)
\end{array} \right)
\eeqn
where $A, B, C, D$ are of size $(m+k) \times (m+k)$ and $A(\theta, x, 0, 0) = D(\theta, x, 0, 0) = 0$, $C(\theta, x, 0, 0) = - B(\theta, x, 0, 0)^T = I_{m+k}$. We consider a locally defined metric 
\beqn
\wt{h}(\theta, x, \tau, y) = \left( \begin{array}{cc} \wt{a}(\theta, x,\tau, y) & \wt{b}^T(\theta, x, \tau, y)  \\ \wt{b}(\theta, x,\tau, y) & \wt{c}(\theta, x,\tau, y)
\end{array}\right).
\eeqn
where the matrix decomposition is written with respect to the same coordinates. Then the value of $\wt{a}(\theta, x, \tau(\theta, x, y), y)$ and part of $\wt{b}(\theta, x, \tau(\theta, x, y), y)$ and $\wt{c}(\theta, x, \tau(\theta, x, y), y)$ have been determined by the choice of $h'$. Then we choose the undetermined part of $\wt{h}|_{\mu^{-1}(0)}$ such that $T\mu^{-1}(0) \bot J{\mc X}_\xi$ for any $\xi \in{\mf g}$ with respect to $\wt{h}|_{\mu^{-1}(0)}$ and such that $\wt{h}(J{\mc X}_\xi, J{\mc X}_\eta) = \wt{h}({\mc X}_\xi, {\mc X}_\eta)$. Moreover, we require that the 1-jet of $\wt{h}$ along $L$ in the $\tau$-direction satisfy
\beq\label{eqna4}
\frac{\partial}{\partial \tau_i} \wt{a}(\theta, x, 0, 0) + \frac{\partial }{\partial \tau_i} \left( A^T \wt{a} A + C^T \wt{b} A + A^T \wt{b}^T C + C^T \wt{c} C \right)(\theta, x, 0, 0) = 0,\ i = 1, \ldots, k.
\eeq
Since $A|_L\equiv 0$, the left-hand-side of this equation is $\partial_{\tau_i} \wt{a}(\theta, x, 0, 0)$ plus terms which don't contain derivatives of $\wt{a}$. Therefore \eqref{eqna4} has a solution subordinate to all the constrains we have put on the 0-jet of $\wt{h}$. This gives us a metric $\wt{h}$ defined in the coordinate patch.

Now we define $h(v, w) = \frac{1}{2} \big( \wt{h}(v, w) + \wt{h}(Jv, Jw) \big)$ and we claim that $h$ satisfies Definition \ref{defn48} inside the coordinate patch we are considering, except for the $G$-invariance. As in \cite{Frauenfelder_disk}, we can use such locally constructed metrics and a partition of unity to construct a global metric $h$, satisfying (1)--(4). 

Indeed, the first condition is automatic. For the fourth condition, for $Y \in T\mu^{-1}(0)$ and $Z = J{\mc X}_\xi$, decompose $Y = Y_1 + Y_2$ where $Y_1 \in H$ and $Y_2 = {\mc X}_\eta$ for some $\eta \in {\mf g}$. Then $h(Y, J{\mc X}_\xi) = \frac{1}{2} \wt{h} ( J Y_1 + J {\mc X}_\eta, - {\mc X}_\xi) = 0$ by the condition required for $\wt{h}$. Therefore, $J(TL)$ is orthogonal to $TL$. 

For the totally geodesic condition, we see that \eqref{eqna4} implies that for $X, Y$ tangent to $L$, $\nabla_X Y \in T\mu^{-1}(0)|_L$ where $\nabla$ is the Levi-Civita connection of $h$. Moreover, by the condition that $\wt{h}(J{\mc X}_\xi, J{\mc X}_\eta) = \wt{h}({\mc X}_\xi, {\mc X}_\eta)$ along $\mu^{-1}(0)$, we see that $h|_{\mu^{-1}(0)} = \wt{h}|_{\mu^{-1}(0)}$. Since $L$ is totally geodesic with respect to $\wt{h}|_{\mu^{-1}(0)} = h'$, this implies that $\nabla_X Y \in TL$ and $L$ is totally geodesic in the local coordinate patch. 

Now the metric $h$ constructed above may not be $G$-invariant. We integrate $h$ against the Haar measure of $G$, getting a $G$-invariant metric $\hat{h}$. The point-wise conditions (1), (2), (4) are clearly preserved. To see that $L$ is totally geodesic with respect to $\hat{h}$, it suffices to show that for any vector fields $X, Y$ tangent to $L$, $\hat\nabla_X Y$ is tangent to $L$, where $\hat \nabla$ is the Levi-Civita connection of $h$. Indeed, if we denote by $\nabla$ the Levi-Civita connection of $h$, then 
\beqn
\hat\nabla_X Y = \int_G \left( g_* \right)^{-1}\nabla_{g_* X} (g_* Y)  dg.
\eeqn
Since $L$ is $G$-invariant, $g_* X$, $g_* Y$ are both tangent to $L$. Therefore we see that $\hat\nabla_X Y$ is tangent to $L$. Therefore $\hat{h}$ is a $(J, L, \mu)$-admissible metric.
\qed \end{proof}

Now we fix a $(J, L, \mu)$-admissible metric $h$ on $X$. We use $\langle \cdot, \cdot \rangle$ to denote the inner product of $h$ in the remaining of this appendix.  Then by (4) of Definition \ref{defn48}, there exists ${\bm n}= {\bm n}(\updelta)>0$ such that for any $x \in X$ with $|\mu(x)| \leq \updelta$, and any $Y \in T_x X$, 
\beqn
\big| \langle J {\mc X}_{d\mu\cdot Y}, Y \rangle \big| \leq {\bm n} \Big( \big| d\mu \cdot Y \big|^2  + \big| \mu(x) \big| |Y|^2 \Big).
\eeqn
Since a rescaling of $h$ is still $(J, L, \mu)$-admissible, we may assume instead
\beq\label{eqna5}
\big| \langle J {\mc X}_{d\mu \cdot Y}, Y \rangle \big| \leq \frac{1}{2} \big| d\mu \cdot Y \big|^2  + {\bm n} \big| \mu(x) \big| \big| Y \big|^2.
\eeq	
Moreover, we may assume that $h$ coincides with a small constant multiple of $\omega(\cdot, J \cdot)$ whenever $|\mu(x)| \geq \updelta$, so that \eqref{eqna5} holds throughout $X$. 

\subsection{The isoperimetric inequality}

We first recall Po\'zniak's isoperimetric inequality (\cite[Lemma 3.4.5]{Pozniak}). Let $(Y, \omega)$ be a symplectic manifold and let $L_0, L_1\subset Y$ be two compact Lagrangian submanifolds which intersect cleanly in $Y$. 

\begin{lemma}\label{lemma50}\cite[Lemma 3.4.5]{Pozniak} There exist constants $\updelta' = \updelta'(L_0, L_1) >0$ and ${\bm c}' = {\bm c}'(L_0, L_1)>0$ satisfying the following condition. Let $x: [0, \pi] \to Y$ be a $C^1$-path with $x(0) \in L_0, x(1) \in L_1$ and $\int_0^\pi |x'(t)|^2 dt \leq ( \updelta')^2$. Then there exists a $C^1$-map $v: [0,1]\times [0, \pi] \to X$ with 
\beqn
v(s, 0) \in L_-, v(s, \pi) \in L_+, v(0, t) \in L_- \cap L_+, v(1, t) = x(t).
\eeqn
Moreover, if we define the symplectic action of the path $x$ as 
\beq\label{eqna6}
{\bm a}(x) = - \iint_{[0,1]\times[0, \pi]} v^* \omega,
\eeq
then  
\beq\label{eqna7}
|{\bm a}(x)|\leq {\bm c}' \int_0^\pi |x'(t)|^2 dt.
\eeq
\end{lemma}

Now we consider two $G$-Lagrangian submanifolds $L_0, L_1$ of the Hamiltonian $G$-manifold $(X, \omega, \mu)$. Suppose they intersect cleanly in $\bar{X}$. Consider the path spaces
\beqn
{\mc P}:= \left\{ (x, \eta) \in C^{\infty}([0,\pi], X \times {\mf g} ) \ |\  x(0)\in L_0,\ x(\pi) \in L_1 \right\},
\eeqn
\beqn
{\mc P}_0:= \left\{ (x,\eta) \in {\mc P}\ |\ x'(t) + {\mc X}_{\eta(t)} (x(t)) = 0 \right\}.
\eeqn
Then we define a ``local action functional'' analogous to that in \cite{Frauenfelder_thesis} and \cite{Gaio_Salamon_2005} when a path $(x, \eta) \in {\mc P}$ is sufficiently close to ${\mc P}_0$. More precisely, for $(x, \eta) \in {\mc P}$, we denote
\beqn
l(x, \eta) = \int_0^\pi \left| x'(t) + {\mc X}_{\eta(t)} (x(t)) \right| dt. 
\eeqn
Then we have
\begin{lemma}\label{lemma51}
There exist positive constants $\updelta = \updelta(L_0, L_1)$ and ${\bm c} = {\bm c}(L_0, L_1)$ such that for $(x, \eta) \in {\mc P}$ with
\beq\label{eqna8}
\sup_{t\in [0,\pi ]} \left| \mu(x(t)) \right| \leq \updelta,\ \sup_{t\in [0,\pi]} \left| x'(t) + {\mc X}_{\eta(t)} (x(t)) \right| \leq \updelta,
\eeq
there exists $(\wt{x}, \wt\eta) \in {\mc P}_0$ such that
\beqn
\sup_{t\in [0,\pi ]} \left| \eta(t) - \wt\eta(t) \right| \leq {\bm c} l(x, \eta),\ d(x(t), \wt{x}(t)) \leq {\bm c} \left( \left| \mu(x(t)) \right| + l(x, \eta) \right).
\eeqn
\end{lemma}
\begin{proof}
For small $\updelta>0$, there exist a unique path $(x_0, \eta_0): [0,\pi ]\to \mu^{-1}(0) \times {\mf g}$ such that
\beqn
x_0(0) \in L_0,\ x_0(\pi ) \in L_1,\ x(t) = \exp_{x_0(t)} (J {\mc X}_{\eta_0}),\ \left| \eta_0(t) \right|  \leq c_1 \left| \mu(x(t)) \right|.
\eeqn
Moreover, there exists $c_2>0$ such that
\beqn
\left| x_0'(t) + {\mc X}_{\eta(t)} (x_0(t)) \right| \leq c_2 \left| x'(t) + {\mc X}_{\eta(t)} (x(t)) \right|.
\eeqn
On the other hand, since $\bar{L}_0$ and $\bar{L}_1$ intersect cleanly, there exists $\bar{q} \in \bar{L}_0 \cap \bar{L}_1$ such that
\beqn
d( \bar{q}, \bar{x}_0(0)) \leq c_3 \int_0^\pi |\bar{x}_0'(t)| dt. 
\eeqn
Then we can choose a lift $q \in L_0 \cap L_1$ of $\bar{q}$, such that
\beqn
d \left( q, x_0(0) \right) \leq c_4 l(x_0, \eta) \leq c_2 c_4 l(x, \eta).
\eeqn
On the other hand, choose $g: [0,\pi ]\to G$ such that $g(0) = {\rm Id}$, $ g'(t) g(t)^{-1} = -\eta(t)$. Define
\beqn
(\wt{x}(t), \wt\eta(t)) = \left( g(t) q, \eta (t) \right) = g^{-1} (q, 0)\in {\mc P}_0.
\eeqn
Then we see there exists $c_5>0$ such that 
\beqn
\begin{split}
d(x(t), \wt{x}(t)) \leq &\ d(x(t), x_0(t)) + d(x_0(t), \wt{x}(t)) \\[0.2cm]
\leq &\  c_5 |\mu(x(t))| + d ( g(t)^{-1}x_0 (t), q) \\
\leq &\ c_5 |\mu(x(t))| + c_5 d(x_0(0), q) + c_5  \int_0^\pi \left| \frac{d}{dt} g(t)^{-1} x_0(t) \right| dt\\
\leq &\ c_6 \left( |\mu(x(t))| + l(x, \eta) \right).
\end{split}
\eeqn
\qed \end{proof}

For $\updelta>0$, denote by ${\mc P}_\updelta \subset {\mc P}$ the subset of pairs $(x, \eta)$ that satisfy \eqref{eqna8}. Then we define the local action functional for all $(x, \eta) \in {\mc P}_\updelta$ by
\beqn
{\mc A}_{loc} (x, \eta) = - \int_{[0,1]\times [0,\pi ]} u^* \omega + \int_0^\pi \langle \mu(x(t)), \eta(t) \rangle dt. 
\eeqn

\begin{lemma}\label{lemma52}
There exist positive constants $\updelta = \updelta(L_0, L_1)$ and ${\bm c} = {\bm c}(L_0, L_1)$ such that for $(x, \eta) \in {\mc P}_\updelta$, we have
\beqn
\left| {\mc A}_{loc}(x, \eta) \right| \leq {\bm c} \int_0^1 \left( \left| x'(t) + {\mc X}_{\eta(t)} (x(t)) \right|^2 + \left| \mu(x(t)) \right|^2 \right) dt.
\eeqn
\end{lemma}

\begin{proof}
The same as in Step 3 of the proof of \cite[Lemma 3.17]{Frauenfelder_thesis}.
\qed \end{proof}

\subsection{{\it a priori} estimates}

\begin{lemma}\label{lemma53}
Let $u: B_{R+r} \to {\mb R}_+\cup \{0\}$ be a $C^2$-function and $a>0$ such that
\beqn
\Delta u \geq - a(u +u^2). 
\eeqn
Then for any $x \in B_R$, 
\beqn
\int_{B_r(x)} u \leq \frac{\pi}{8a} \Longrightarrow u(x) \leq \max \{  \frac{\pi}{8} , \frac{4a r^2 }{\pi} \} \frac{1}{r^2}  \int_{B_r(x)} u. 
\eeqn
\end{lemma}

\begin{proof}
We first prove for the case that $r = 1$. Using the Heinz trick (cf. \cite[Page 82]{McDuff_Salamon_2004}), define the function $f: [0, 1]\to {\mb R}$ by
\beqn
f(\rho) = (1- \rho)^2 \sup_{B_\rho(x)} u.
\eeqn
Let $\rho^* \in [0,1)$ be some number at which $f$ attains its maximum. Choose $w^* \in B_{\rho^*}(x)$ such that $u(z^*) = \sup_{B_{\rho^*}(x)} u$ and denote $c^* = u(z^*)$. Denote $\delta = \frac{1 - \rho^* }{2}< 1$. Then for $w \in B_\delta(w^*)$,
\beqn
u(w) \leq \sup_{B_{\rho^* + \delta}(x)} u \leq \frac{  ( 1- \rho^*)^2 }{ ( 1- \rho^* - \rho)^2 } \sup_{B_{\rho^*}(x)} u = 4 c^*. 
\eeqn
Therefore on $B_\delta(w^*)$, we have 
\beqn
\Delta u \geq - a (u +u^2) \geq  - \left( 4ac^* + 16 a(c^*)^2 \right)=: -4 m c^* ,
\eeqn
which implies that the function $\wt{u}(w) = u(w) + m c^* |w- w^*|^2$ is subharmonic on $B_\delta (w^*)$. Therefore, for any $\rho \in (0, \delta]$, we have
\beqn
c^* = u(w^*) \leq \frac{1 }{ \pi \rho^2} \int_{B_\rho (w^*)} \left( u + mc^* |w- w^*|^2 \right) = \frac{1}{\pi \rho^2}\int_{B_\rho(w^*)} u + \frac{1}{2} m c^* \rho^2. 
\eeqn
Now if $ m \delta^2 \leq 1$, then we take $\rho = \delta$, which implies that 
\beqn
\frac{1}{2} c^* \leq c^* - \frac{1}{2} mc^* \delta^2  \leq \frac{1}{\pi \delta^2} \int_{B_\delta(w^*)} u. 
\eeqn
Then
\beqn
u(x) = f(0) \leq f(\rho^*) = 4 \delta^2 c^* \leq \frac{8}{\pi} \int_{B_\delta(w^*)} u \leq \frac{8}{\pi} \int_{B_1(x)} u. 
\eeqn
On the other hand, if $m \delta^2 > 1$, then take $\rho = \sqrt{ \frac{ 1 }{m }  } < \delta$, we see 
\beqn
\frac{c^*}{2} \leq \frac{ m}{\pi} \int_{B_\rho(w^*)} u \leq \frac{ m}{\pi} \int_{B_1(x)} u. 
\eeqn
Therefore
\beqn
\frac{\pi}{8a } =  \epsilon \geq \int_{B_1(x)} u \geq \frac{ \pi c^* }{2 m } =  \frac{ \pi c^* }{a + 4 a c^*}.
\eeqn
It implies that $c^* \leq \frac{1}{4}$. Therefore
\beqn
u(x) \leq u(w^*) = c^* \leq \frac{2 ( a + 4 ac^*) }{\pi} \int_{B_1(x)} u \leq \frac{4 a }{\pi} \int_{B_1(x)} u. 
\eeqn
In summary, we see that 
\beqn
\int_{B_1(x)} u \leq \frac{\pi }{8 a} \Longrightarrow u(x) \leq \max\{ \frac{8}{\pi} , \frac{4 a }{\pi} \} \int_{B_1(x)} u. 
\eeqn
For general $r>0$, the estimate follows by applying the above argument to $v(x) = u(rx)$ and $a$ replaced by $ar^2$.
\qed \end{proof}



\subsection{Mean-value estimate}

Let $h$ be a $(J, L, \mu)$-admissible metric on $X$ satisfying \eqref{eqna5}. Let $\nabla$ be the Levi-Civita connection of $h$ and we have the covariant derivatives defined by \eqref{eqna2}.

Now we consider the vortex equation on $\Xi$. An area form can be written as $\sigma ds dt$ for a smooth function $\sigma: \Xi \to (0, +\infty)$. We assume that there exist $c^{(l)} >0$, $l = 1, \ldots$ such that
\beq\label{eqna9}
\left| \nabla^l \sigma \right| \leq c^{(l)} \sigma.
\eeq

The vortex equation is written as
\beq\label{eqna10}
v_s + J v_t = 0,\ \kappa + \sigma \mu(u) = 0.
\eeq
Using a $G$-invariant metric $h$, we define the energy density for a solution $(u, \phi, \psi)$ by 
\beqn
e(s, t) = | v_s(s, t) |_h^2 + \sigma(s, t) | \mu(u(s, t)) |^2,
\eeqn
where second norm is the $G$-invariant metric on ${\mf g}$ we used to define the vortex equation. 

\begin{lemma}\label{lemma54}
Let $\Xi\subset {\mb C}$ be an open subset. For any compact subset $K\subset X$, there exists ${\bm c} = {\bm c}(K)>0$ depending on $K$ (and also on the constants $c^{(l)}$ of \eqref{eqna9}) such that for any solution $(u, \phi, \psi)$ of \eqref{eqna10} satisfying $u(\Xi) \subset K$, its energy density function $e$ satisfies 
\beq\label{eqna11}
\Delta e\geq -{\bm c} \left( e + e^2 \right). 
\eeq
Here $\Delta$ is the standard Laplacian in coordinates $(s, t)$. If $\sigma$ is constant, we have
\beq\label{eqna12}
\Delta e \geq - {\bm c} e^2.
\eeq
\end{lemma}

\begin{proof}

Since the covariant derivative respects the metric, we have
\beqn
\begin{split}
\frac{1}{2}\Delta \left| \mu(u) \right|^2 = &\  \left| \nabla_{A, s} \mu \right|^2 + \left| \nabla_{A, t} \mu \right|^2  +  \left\langle \left( (\nabla_{A, s})^2  + ( \nabla_{A, t})^2 \right) \mu, \mu \right\rangle\\
= &\ \left| d\mu \cdot v_s \right|^2 + \left| d \mu \cdot v_t \right|^2  + \left\langle \nabla_{A, s} ( d\mu \cdot v_s) + \nabla_{A, t} (d\mu \cdot v_t), \mu \right\rangle\\
= &\ \left| d\mu \cdot v_s \right|^2 + \left| d \mu \cdot v_t \right|^2  + \left\langle \nabla_{A, s} ( d\mu \cdot (- Jv_t) + \nabla_{A, t} (d\mu \cdot Jv_s), \mu \right\rangle\\
= &\ \left| d\mu \cdot v_s \right|^2 + \left| d \mu \cdot v_t \right|^2  + \left\langle \rho(v_t, v_s) - \rho(v_s, v_t) , \mu \right\rangle\\
&\ + \left\langle d\mu \cdot \left( J \nabla_{A, t} v_s - J \nabla_{A, s} v_t \right), \mu     \right\rangle\\
= &\ \left| d\mu \cdot v_s \right|^2 + \left| d \mu \cdot v_t \right|^2  + \left\langle \rho(v_t, v_s) - \rho(v_s, v_t) - d\mu \cdot (J {\mc X}_\kappa)  , \mu \right\rangle
\end{split}
\eeqn
Denote $\bar\rho(v_s, v_t) = \rho(v_s, v_t) - \rho(v_t, v_s)$. Then there exist $c_K>0$ and for any $\epsilon>0$, $c_{K, \epsilon}>0$, depending on the metric $h$ and the compact subset $K$ such that
\beq\label{eqna13}
\begin{split}
\frac{1}{2}\Delta  \left( \sigma \left| \mu(u) \right|^2\right) = &\ \sigma^2 \langle d\mu \cdot J {\mc X}_\mu, \mu \rangle\\
&\ + \sigma \left(  \left| d\mu \cdot v_s \right|^2 + \left| d\mu\cdot v_t \right|^2 + \langle \mu(u), - \bar\rho(v_s, v_t) \rangle \right)\\
& + \frac{ \Delta \sigma }{ 2} \left| \mu(u) \right|^2  + 2 \langle (\partial_s \sigma) d\mu\cdot v_s  + (\partial_t \sigma) d\mu \cdot v_t, \mu \rangle  \\
 \geq &\ \sigma^2 \omega( {\mc X}_\mu, J{\mc X}_\mu) + \sigma \left( |d\mu \cdot v_s |^2 + |d\mu \cdot v_t|^2 \right) \\
& + \frac{\Delta \sigma }{ 2} |\mu(u)|^2 -  c_K \left( |d\sigma| |v_s| |\mu| + \sigma |\mu| |v_s|^2\right)\\
\geq &\ \sigma^2 \left( \omega ( {\mc X}_\mu, J{\mc X}_\mu ) -  \left( \frac{ |\Delta \sigma| }{2\sigma^2 } + \epsilon \right) |\mu(u)|^2   \right)\\
& + \sigma \left( |d\mu \cdot v_s|^2 + |d\mu \cdot v_t|^2 \right) - c_{K, \epsilon} |v_s|^4 - c_{K, \epsilon} \left( \frac{ |d\sigma|^2 }{\sigma^2} \right) |v_s|^2.
\end{split}
\eeq

To estimate $\Delta |v_s|^2$, we have the following standard calculations.
\begin{align}
\begin{split}\label{eqna14}
\nabla_{A, s} v_t - \nabla_{A, t} v_s = &\ \nabla_s ( \partial_t u + {\mc X}_\psi ) + \nabla_{v_t} {\mc X}_\phi - \nabla_t (\partial_s u + {\mc X}_\phi) - \nabla_{v_s} {\mc X}_\psi \\
= &\ \nabla_{\partial_s u} {\mc X}_\psi + X_{\partial_s \psi} + \nabla_{v_t} {\mc X}_\phi - \nabla_{\partial_t u} {\mc X}_\phi - X_{\partial_t \phi} - \nabla_{v_s} {\mc X}_\psi \\
= &\ X_{\partial_s \psi} - X_{\partial_t \phi} + \nabla_{{\mc X}_\psi} {\mc X}_\phi - \nabla_{{\mc X}_\phi} {\mc X}_\psi \\
 = &\  - \sigma {\mc X}_\mu.
\end{split}\\
\begin{split}
\nabla_{A, t} {\mc X}_\mu  = &\  \nabla_t {\mc X}_\mu + \nabla_{{\mc X}_\mu} {\mc X}_\psi \\
            = &\  {\mc X}_{\partial_t \mu} + \nabla_{\partial_t u} {\mc X}_\mu + \nabla_{{\mc X}_\psi} {\mc X}_\mu + \left[ {\mc X}_\mu, {\mc X}_\phi \right]\\
                     = &\  {\mc X}_{d\mu\cdot \partial_t u} + X_{[\phi, \mu]} + \nabla_{v_t} {\mc X}_\mu\\
										= &\  {\mc X}_{d\mu \cdot v_t} + \nabla_{v_t} {\mc X}_\mu
\end{split}\\
\begin{split}
\nabla_{A, s} v_s + \nabla_{A, t} v_t = &\  \nabla_{A, t} (J v_s) - \nabla_{A, s} (J v_t) \\
= &\ J \left( \nabla_{A, t} v_s - \nabla_{A, s} v_t \right)  + \left( \nabla_{A, s} J \right) v_t - \left( \nabla_{A, t} J \right) v_s\\
= &\  \sigma J {\mc X}_\mu + \left( \nabla_{v_s} J \right) v_t - \left( \nabla_{v_t} J \right) v_s .
\end{split}
\end{align}
On the other hand, by the $G$-invariance of $\nabla J$ and Lemma \ref{lemma47}, there exist tensors $L_1, L_2, L_3$ such that
\begin{align*}
\nabla_{A, s} \left( \left( \nabla_{v_s} J \right) v_t \right) = &\ L_1 (v_s, v_s, v_t) + L_2 \left( \nabla_{A, s} v_s, v_t \right) + L_3\left( v_s, \nabla_{A, s} v_t \right);\\
\nabla_{A, s} \left( \left( \nabla_{v_t}J \right) v_s \right) = &\ L_1 (v_s, v_t, v_s) + L_2 \left( \nabla_{A, s} v_t, v_s \right) + L_3 \left( v_t, \nabla_{A, s} v_s \right).
\end{align*}
Therefore, we have
\beqn
\begin{split}
( \nabla_{A, s}^2 + \nabla_{A, t}^2 ) v_s = &\ \nabla_{A, s} \left( \nabla_{A, s} v_s + \nabla_{A, t} v_t \right) - \left[ \nabla_{A, s}, \nabla_{A, t} \right] v_t - \nabla_{A, t} \left( \nabla_{A, s} v_t - \nabla_{A, t} v_s \right)\\
= &\ \nabla_{A, s} \left( \sigma J {\mc X}_\mu + \left( \nabla_{v_s} J\right) v_t - \left( \nabla_{v_t} J \right) v_s \right) \\
&\ - R(v_s, v_t)v_t + \sigma \nabla_{v_t} {\mc X}_\mu - \sigma  {\mc X}_{d\mu\cdot v_t} + \sigma  \nabla_{v_t} {\mc X}_\mu  + (\partial_t \sigma) {\mc X}_\mu\\
= &\ \sigma \left( \nabla_{v_s} (J {\mc X}_\mu) + J {\mc X}_{d\mu \cdot v_s} + J \nabla_{v_s} {\mc X}_\mu + 2 \nabla_{v_t} {\mc X}_\mu - {\mc X}_{d\mu \cdot v_t}  \right) - R(v_s, v_t) v_t \\
&\ + ( \partial_s \sigma) J {\mc X}_\mu + (\partial_t \sigma) {\mc X}_\mu  + \nabla_{A, s} \left( (\nabla_{v_s} J ) v_t - (\nabla_{v_t} J ) v_s \right)\\
= & \ \sigma \left( \nabla_{v_s} (J {\mc X}_\mu) + J {\mc X}_{d\mu \cdot v_s} + J \nabla_{v_s} {\mc X}_\mu + 2 \nabla_{v_t} {\mc X}_\mu - {\mc X}_{d\mu \cdot v_t}  \right) \\
& - R(v_s, v_t) v_t + ( \partial_s \sigma) J {\mc X}_\mu + (\partial_t \sigma) {\mc X}_\mu  \\
&\ + L_1(v_s, v_s, v_t) + L_2 (\nabla_{A,s} v_s, v_t) + L_3 (v_s, \nabla_{A, s} v_t) \\
&\  + L_1 (v_s, v_t, v_s) + L_2 \left( \nabla_{A, s} v_t, v_s \right) + L_3 \left( v_t, \nabla_{A, s} v_s \right).
\end{split}
\eeqn

Therefore, since $u(\Xi)\subset K$, with abusive use of (small) $\epsilon$ and (big) $c_{K, \epsilon}$, we have
\beq
\begin{split}\label{eqna17}
\frac{1}{2} \Delta \left| v_s \right|^2 = &\ \left| \nabla_{A, s} v_s \right|^2 + \left| \nabla_{A, t} v_s \right|^2 + \left\langle \left( (\nabla_{A, s})^2 + (\nabla_{A, t})^2\right) v_s, v_s \right\rangle\\
\geq  &\ \left| \nabla_{A, s} v_s \right|^2 + \left| \nabla_{A, t} v_s \right|^2 - c_{K, \epsilon} |v_s|^4 \\
&\ + \sigma \left\langle \nabla_{v_s} (J {\mc X}_\mu) + J {\mc X}_{d\mu \cdot v_s} + J \nabla_{v_s} {\mc X}_\mu + 2 \nabla_{v_t} {\mc X}_\mu - {\mc X}_{d\mu \cdot v_t}, v_s \right\rangle\\
&\ - |d\sigma| |{\mc X}_\mu| |v_s| - \epsilon\left( \left| \nabla_{A, s} v_s \right|^2 + \left| \nabla_{A, s} v_t  \right|^2 \right) \\
\geq  &\ \left| \nabla_{A, s} v_s \right|^2 + \left| \nabla_{A, t} v_s \right|^2 - c_{K,\epsilon} |v_s|^4 - \sigma c_{K, \epsilon} |\mu| |v_s|^2\\
&\  + \sigma \left(\langle J {\mc X}_{d\mu \cdot v_s}, v_s \rangle - \langle J {\mc X}_{d\mu \cdot v_t}, v_t \rangle \right) \\
 &\ - |d\sigma| |{\mc X}_\mu| |v_s|- \left( \epsilon  \left| \nabla_{A, s} v_s \right|^2 + 2\epsilon \left| \nabla_{A, t} v_s  \right|^2 + 2\epsilon \sigma^2 \left| {\mc X}_\mu \right|^2 \right) \\
\geq &\ - \epsilon \sigma^2 |\mu(u)|^2 - c_{K, \epsilon} |v_s|^4  - c_{K, \epsilon} \frac{|d\sigma|^2 }{\sigma^2} |v_s|^2 \\
&\ + \sigma \left(\langle J {\mc X}_{d\mu \cdot v_s}, v_s \rangle - \langle J {\mc X}_{d\mu \cdot v_t}, v_t \rangle \right).
\end{split}
\eeq

Therefore, by \eqref{eqna9}, \eqref{eqna13} and \eqref{eqna17}, we have
\beq\label{eqna18}
\begin{split}
\frac{1}{2} \Delta e \geq  &\ \sigma^2 \left( \omega({\mc X}_\mu, J{\mc X}_\mu) -\left( \frac{|\Delta \sigma| }{ \sigma^2} + \epsilon\right)|\mu|^2 \right)  - c_{K, \epsilon} \left( |v_s|^4 + \frac{ |d\sigma|^2 }{\sigma^2} |v_s|^2 \right) \\
&\ + \sigma \left( |d\mu\cdot v_s|^2 + |d\mu\cdot v_t|^2 + \langle J {\mc X}_{d\mu \cdot v_s}, v_s \rangle - \langle J {\mc X}_{d\mu \cdot v_t}, v_t \rangle\right)\\
\geq &\ \sigma^2 \left( \omega ({\mc X}_\mu, J{\mc X}_\mu) - \epsilon |\mu|^2 \right) - c^{(2)} \sigma |\mu|^2  - c_{K, \epsilon}  \left( |v_s|^4 +  \left( c^{(1)} \right)^2 |v_s|^2\right) \\
&\ + \sigma \left( \frac{1}{2} |d\mu\cdot v_s|^2 +  \frac{1}{2} |d\mu \cdot v_t|^2 - {\bm n} |\mu(u)| |d\mu \cdot v_s| |v_s| - {\bm n} |\mu(u)| |d\mu \cdot v_t ||v_t| \right)\\
\geq &\ \sigma^2 \left( \omega({\mc X}_\mu, J{\mc X}_\mu)  - \epsilon |\mu|^2 \right)- c^{(2)} \sigma |\mu|^2  - c_{K, \epsilon}  \left( |v_s|^4 +  \left( c^{(1)} \right)^2 |v_s|^2\right) \\
&\  - \sigma {\bm n}^2 |\mu|^2 (|v_s|^2 + |v_t|^2).
\end{split}
\eeq
Here the second inequality follows from \eqref{eqna5}. Moreover, if $|\mu| \leq \updelta$, then for $\epsilon \leq {\bm m}$, the first term of the last line is nonnegative; if $|\mu|> \updelta$, then the first term is greater than or equal to $- \epsilon \updelta^{-2} \sigma^2 |\mu|^4$. In either case, there exist ${\bm c}', {\bm c}>0$ depending on $c_{K, \epsilon}$, ${\bm m}$, ${\bm n}$ and $\updelta$ such that
\beq\label{eqna19}
\begin{split}
\frac{1}{2} \Delta e \geq &\ \sigma^2 \big( \omega({\mc X}_\mu, J{\mc X}_\mu) - \epsilon |\mu|^2 \big) -{\bm c}' e^2 - \Big( c^{(2)} + c_{K, \epsilon} \big( c^{(1)} \big)^2 \Big) e\\
\geq &\ - {\bm c} e^2 - \Big( c^{(2)} + c_{K, \epsilon} \big( c^{(1)} \big)^2 \Big) e.
\end{split}
\eeq
This implies \eqref{eqna11}. \eqref{eqna12} follows by setting $c^{(1)}= c^{(2)} = 0$. \qed \end{proof}

\subsection{Removal of singularity at punctures}\label{appendixa5}

We prove the first part of Lemma \ref{lemma6}, which we restate as follows.
\begin{prop}\label{prop55}
Let $L_-, L_+$ be two $G$-Lagrangians of $X$ which intersect cleanly in $\bar{X}$. Let $(\Sigma, \partial \Sigma) = \left( {\mb D}^* \cap {\mb H}, {\mb D}^* \cap {\mb R} \right)$ and $\partial_\pm \Sigma = {\mb D}^* \cap {\mb R}_\pm$. Suppose ${\bf v}$ is a bounded solution to \eqref{eqn22} on $\left( \Sigma, \partial \Sigma  \right)$ with respect to a smooth area form $\nu \in \Omega^2 ( \Sigma)$. Then there exists a smooth gauge transformation $g: \Sigma \to G$ such that $g^* u$ extends continuously to $\{0\}$. 
\end{prop}

\begin{proof}
Identify $\Sigma$ with $[0, +\infty) \times [0, \pi]$ via $\Sigma \ni z \mapsto w = s + {\bm i} t= -  \log z$ and view the strip as a subset of ${\mb H}$. Suppose  $\nu = \sigma ds dt$, it is easy to check that \eqref{eqna9} is satisfied. Let $\wt{e}: [0, +\infty) \times [0,\pi ] \to {\mb R}_+ \cup \{0\}$ be the energy density function. By Lemma \ref{lemma54}, we have
\beqn
\Delta \wt{e} \geq - {\bm c} \left( 1 + \wt{e}^2 \right).
\eeqn
Here ${\bm c}$ depends on a choice of $G$-invariant metric on $X$. 

Now to derive pointwise decay of $\wt{e}$ as $s \to +\infty$, we have to extend $\wt{e}$ a bit beyond the boundary of $[0,+\infty) \times [0,\pi]$. For example, we use reflection to define
\beqn
\wt{e}(s, t) = \wt{e}(s, -t),\ t \in (-a, 0)
\eeqn
for $a > 0$ a small constant. Then $\wt{e}$ is extended to $[0, +\infty) \times (-a, 1]$. To see that the extension still satisfies \eqref{eqna11}, it suffices to show that $\partial_t \wt{e}(s, 0) = 0$. This is the place where we need the properties of metrics defined by Definition \ref{defn48}.  Choose a $(J, L_-, \mu)$-admissible metric $h_-$. 

First, by the boundary condition, $\partial_t | \mu(u) |^2 = 0$. On the other hand,
\beqn
\begin{split}
\frac{1}{2} \partial_t \langle v_s, v_s \rangle = &\ \langle \nabla_{A, t} v_s, v_s \rangle \\
= &\ \langle \nabla_{A, s} v_t , v_s \rangle + \langle \nabla_{A, t} v_s - \nabla_{A,s} v_t, v_s \rangle \\
= &\ \langle \nabla_{A, s} (J v_s), v_s \rangle + \langle {\mc X}_\mu, v_s \rangle \\
= &\  \left\langle \left( \nabla_{v_s} J \right) v_s, v_s \right\rangle + \left\langle J \left( \nabla_s v_s + \nabla_{v_s} {\mc X}_\phi \right), v_s \right\rangle + \langle {\mc X}_\mu, v_s \rangle.
\end{split}
\eeqn
Here the third equality follows from \eqref{eqna14} and the last follows from $\nabla_{A, s} J = \nabla_{v_s} J$. Then evaluating at $t = 0$, we see that in the last row, the first term vanishes because $\nabla J$ is skew-adjoint; the second term vanishes because $L_-$ is totally geodesic and $JTL_-$ is orthogonal to $TL_-$; the third term vanishes by the boundary condition.

Therefore \eqref{eqna11} holds on $[0, +\infty) \times (-a, 1]$, for the constant ${\bm c}$ associated with the metric $h_-$. By the mean value estimate (\cite[Page 12]{Salamon_lecture}) for any $B_{\frac{a}{2}}(w) \subset [0, +\infty) \times (-a, \pi]$,
\beqn
\lim_{s\to +\infty} \wt{e}(s, t) = 0,\ \forall t \in \left[ 0, \pi - \frac{a}{2} \right).
\eeqn
To achieve the estimate near the other boundary component, simply take a $(J, L_+, \mu)$-admissible metric $h_+$ and do the reflection along the other boundary. Since all metrics are equivalent, we see that $\wt{e}(s, t)$ converges to zero uniformly as $s \to +\infty$. 

On the other hand, it is easy to see that there exists a gauge transformation $g: [0, +\infty) \times [0,1] \to G$ which transforms $(A, u)$ into temporal gauge, i.e., 
\beqn
g^* (A, u) = (\psi dt, u).
\eeqn
Since $\sigma$ decays exponentially as $s \to +\infty$, by the equation $\partial_s \psi + \sigma \mu(u) = 0$ and the uniform boundedness of $|\mu(u)|$, we see that 
\beqn
\lim_{s \to +\infty} \left| {\mc X}_{\psi}(u(s, t)) \right| = 0.
\eeqn
Therefore $\left| \partial_t u (s, t) \right| \to 0$ as $s \to +\infty$. Let $\gamma_s(t) = u(s, t)$. Since $L_-$ and $L_+$ intersect cleanly in $X$, by the Po\'zniak's isoperimetric inequality Lemma \ref{lemma50}, we can prove that there exists $x\in L_- \cap L_+$ such that $\displaystyle \lim_{s \to \infty} u(s, t) = x$.
\qed \end{proof}

\subsection{Energy quantization of ${\mb H}$-vortices}\label{appendixa4}

\begin{prop}\label{prop56}
Let $L$ be a $G$-Lagrangian of $X$ and $K \subset X$ be a compact subset. Then there exists a constant $\epsilon_{K, L} > 0$ such that the following holds. Suppose $z_0 \in {\mb H}$ and $r>0$. Suppose ${\bf v} \in \wt{\mc M}(X, L; B_{2r}(z_0) \cap {\mb H}, \nu_0)$. Then if $u( B_{2r}(z_0) \cap {\mb H})\subset K$ and $E({\bf v}) < \epsilon_{K, L}$, then 
\beqn
\sup_{B_{r}(z_0) \cap {\mb H}} e({\bf v}) \leq \frac{8}{\pi r^{2}} E( {\bf v}, B_{2r}(z_0) \cap {\mb H}).
\eeqn
Here $e({\bf v}): B_{2r}(z_0) \cap {\mb H} \to {\mb R}_+ \cup \{0\}$ is the energy density function of ${\bf v}$ with respect to the standard metric on ${\mb H}$ and a $(J, L, \mu)$-admissible metric $h$.
\end{prop}

\begin{proof}
For the same reason as in the proof of Proposition \ref{prop55}, $e$ can be extended to $B_{2r}(z_0) \subset {\mb C}$ by reflection along the boundary of $B_{2r}(z_0) \cap {\mb H}$. This extension still satisfies 
\beqn
\Delta e \geq -{\bm c} e^2.
\eeqn
Then by \cite[Lemma 4.3.2]{McDuff_Salamon_2004}, the estimate holds. 
\qed \end{proof}

\begin{proof}[Proof of Theorem \ref{thm14}]
Choose $\epsilon_{X, L} = \epsilon_{K, L}$ for $K = X_{c_0}$ where $X_{c_0}$ is the one in \eqref{eqn23}. Suppose ${\bf v}$ is an ${\mb H}$-vortex and $E({\bf v}) < \epsilon_{X, L}$. Then by the mean value estimate in Proposition \ref{prop56}, we see that for any $z\in {\mb H}$ and $ r>0$, 
\beqn
e(z)\le \frac{8}{\pi r^{2}} E ( {\bf v}; B_r(z) ).
\eeqn
Let $r\to \infty$, we have $e(z)=0$. Thus $E({\bf v})=0$.
\qed \end{proof}

\subsection{Annulus lemma for vortices on strips}

Ziltener proved (\cite[Proposition 45]{Ziltener_book}) that for any annulus $A(r, R) = \{z\in {\mb C}\ |\ r\leq |z|\leq R\}$ and any small $\epsilon>0$, there exists a constant $E(r, \epsilon)$ such that for any vortex ${\bf v} = (A, u)$ on $A(r, R)$ with respect to the standard area form, if $E({\bf v}) \leq E(r, \epsilon)$, then 
\beqn
\begin{split}
E ( {\bf v}; A(ar, a^{-1} R) ) \leq &\ c a^{ - 2 + \epsilon}  E({\bf v}),\\
{\rm diam}_G ( u(A(ar, a^{-r}R)) ) \leq &\ c a^{-1 + \epsilon} \sqrt{ E({\bf v})}
\end{split}
\eeqn
for some constant $c>0$ and for any $a \in \big[ 2, \sqrt{ R/r} \big]$. Now we prove an analogue of this result on strips. Via the map $w = s + {\bm i}t = \log z$, we identify the strip $A^+(r, R) = A(r, R) \cap {\mb H}$ with 
\beqn
[p, q]\times [0, \pi] := [\log r, \log R] \times [0, \pi].
\eeqn
Let $\nu  = \sigma(s, t) ds dt$ be an area form satisfying \eqref{eqna9}. 

\begin{prop}\label{prop57}
There exists ${\bm a} = {\bm a}(L_-, L_+) > 0$, $\upepsilon = \upepsilon(L_-, L_+)>0$ satisfying the following condition. Given a smooth solution ${\bf v} = (u, \phi, \psi)$ to \eqref{eqna10} with boundary condition $u(A(r, R)\cap {\mb R}_\pm) \subset L_\pm$. Suppose $\sigma$ is bounded from below by a constant $\underline\sigma$. If $E({\bf v}) \leq  \upepsilon$, then for any $s \in \left[ \log 2, \frac{1}{2} ( q-p) \right]$, we have
\beq\label{eqna20}
E( {\bf v}; A^+ ( e^s r, e^{-s} R) ) \leq {\bm a} \exp \Big( - \frac{s \min\{1, \underline\sigma\} }{{\bm c} } \Big) E ( {\bf v}; A^+	(r, R) );
\eeq
\beq\label{eqna21}
{\rm diam}_G ( u( A^+( e^s r, e^{-s} R) )) \leq {\bm a} \exp \Big( -\frac{ s \min\{ 1, \underline\sigma \} }{2 {\bm c}  }\Big) \sqrt{E ( {\bf v}; A^+(r, R) ) }.
\eeq
Here ${\bm c} = {\bm c} (L_-, L_+) >0$ is the ${\bm c}(L_0, L_1)$ in Lemma \ref{lemma52} for $L_0 = L_+$, $L_1 = L_-$.
\end{prop}

\begin{proof} Let $\epsilon (s, t)$ be the energy density with respect to the cylindrical coordinates, so that
\beqn
E ({\bf v}; A^+(r, R) ) = \int_p^q \int_0^\pi \epsilon (s, t) ds dt.
\eeqn
By the estimate of Lemma \ref{lemma54} on the strip and Lemma \ref{lemma53}, we know that there exists $\upepsilon>0$ such that if $E ({\bf v}; A^+(r, R) ) \leq \upepsilon$, then for any $(s, t) \in [ p + \log 2, q - \log 2] \times [0, \pi]$, we have $\epsilon (s, t)\leq \updelta(L_-, L_+)$. Here $\updelta(L_-, L_+)$ is the one from Lemma \ref{lemma52}. (Notice that when applying Lemma \ref{lemma53}, we have to use $(J, L_\pm, \mu)$-admissible metrics to extend the energy density function beyond the boundary of the strip, and notice that the $r$ of Lemma \ref{lemma53} is uniformly bounded). We write $A = d + \phi(s, t) ds + \psi(s, t) dt$. Then for $s \in [p + 2, q - 2]$, we have $(\gamma_s, \psi_s):= ( u( s + {\bm i} \cdot), \psi(s, \cdot) ) \in {\mc P}_\updelta$. So we can define the local equivariant action
\beqn
{\mc A}(s) = {\mc A}_{loc}( \gamma_s, \psi_s). 
\eeqn
For $s \in [ \log 2, ( q - p)/2 ] $, we denote $E(s) = E ( {\bf v}; A^+ ( e^s r, e^{-s} R) )$. Then by the isoperimetric inequality (Lemma \ref{lemma52}), for ${\bm c} = {\bm c}(L_-, L_+)$, we have
\begin{multline*}
E(s) =  \big| {\mc A}( p + s ) - {\mc A}(q - s) \big|\\
\leq  {\bm c} \Big( \int_0^\pi \big(  | v_t( p + s, t) |^2 + | \mu(u (p + s, t)) |^2 \big) dt +  \int_0^\pi \big(  | v_t (q -s, t) |^2 +  | \mu(u( q -s, t)) |^2 \big) dt \Big) \\
\leq  \frac{ {\bm c}}{\min \{1, \underline\sigma\}}  \int_0^\pi \big( e(p+s, t) + e(q-s, t) \big) dt  = - \frac{ {\bm c} }{\min\{ 1, \underline\sigma \}}  \frac{d}{ds} E(s). 
\end{multline*}
Here $e(s, t) = |v_t(s, t)|^2 + \sigma(s, t)|\mu(u(s, t))|^2$. Abbreviate $\wt{\bm c} = {\bm c}/ \min\{1, \underline{\sigma}\}$. Then we have
\beqn
E(s) \leq E(2) \exp \Big( - \frac{ s-2}{\wt{\bm c}}\Big) \leq E ( {\bf v}; A^+(r, R) ) e^{2 /{\bm c}} \exp \Big( - \frac{s}{\wt{\bm c}}\Big).
\eeqn
Therefore \eqref{eqna20} holds.

To prove the estimate for the radius, apply Lemma \ref{lemma53} to $\epsilon $ again, for a choice of $r$ uniformly bounded from below. Then $e$ decays in a similar way as 
\beq\label{eqna22}
\sup_{[p + s, q - s]\times [0, \pi]} \epsilon = O ( \exp ( - s/\wt{\bm c}) ).
\eeq
Integrating over $[ p + s, q - s]$ gives the upper bound on the equivariant diameter.
\qed \end{proof}

Now we prove the following asymptotic property of ${\mb H}$-vortices.

\begin{proof}[Proof of Theorem \ref{thm12}] Let $q \to +\infty$ in \eqref{eqna22}, we obtain
\beqn
\epsilon(s, t) \leq {\bm a} \exp( -s/ {\bm c}).
\eeqn
Since $\epsilon(s, t) = e^{2s} e(s, t)$, we obtain \eqref{eqn27}. On the other hand, we could transform $(A, u)$ into temporal gauge, i.e., $A = d + \psi d t$ such that  
\beqn
\lim_{s \to +\infty} \psi( s, t) = 0.
\eeqn
Then the decay of $\epsilon$ implies that in this gauge, $u$ converges to a limit in $L_- \cap L_+$.
\qed \end{proof}

\section{Trees}\label{appendixb}

In this appendix we fix notions and notations of trees. 

In our convention, a tree ${\mc T}$ consists of a finite set of vertices $V({\mc T})$, a finite set of (finite) edges $E({\mc T}) \subset V({\mc T}) \times V({\mc T})$, and a finite set of {\bf semi-infinite edges} $E_\infty({\mc T})$. The semi-infinite edges are attached to vertices, by a map $\ud\iota: E_\infty({\mc T}) \to V({\mc T})$. A {\bf rooted tree} is a tree with a distinguished vertex, which is usually denoted by $v_\infty$. 

In this paper we mainly consider rooted trees. There are obvious notions of morphisms between rooted trees, and rooted subtrees. We index vertices by letters $\alpha, \beta, \gamma$, etc.. For a rooted tree ${\mc T}$, there is a canonical partial order $\leq$ in $V({\mc T})$ with $v_\infty$ the unique minimal element. Moreover, for notational purpose, we only consider edges with the correct orientation, i.e., for $v_\alpha, v_\beta\in V({\mc T})$, we write $\alpha \succ \beta$ if and only if $v_\alpha$, $v_\beta$ are adjacent and $\beta \leq \alpha$. For any $v_\alpha \in V({\mc T}) \smallsetminus \{ v_\infty\}$, we denote by $v_{\alpha'}\in V({\mc T})$ the unique vertex such that $\alpha \succ \alpha' \in E({\mc T})$. For $e\in E({\mc T}) \cup E_\infty({\mc T})$, denote by $e'\in V({\mc T})$ the end point of $e'$ which is closer to the root.

We regard a tree ${\mc T}$ as a 1-dimensional simplicial complex. Note that a semi-infinite edge only has one end combinatorially, but the point at infinity on the semi-infinite edge is regarded as a point of the simplicial complex.

\begin{defn} (\cite[Definition 1.1]{Fukaya_Oh}) \label{defn58}
\begin{enumerate}

\item A {\bf rooted ribbon tree} with $\ud k$ semi-infinite edges consists of a rooted tree $\ud{\mc T}$ with a topological embedding ${\mf i}: \ud{\mc T} \to \ov{{\mb D}}$ such that $\partial \ud{\mc T}:= {\mf i}^{-1}( \partial {\mb D}) $ consists of the $\ud k$ infinities of these semi-infinite edges. As a convention we always order the semi-infinite edges by $e_1, \ldots, e_{\ud k}$ which respects the cyclic ordering induced by the embedding. 

\item An isomorphism between two rooted ribbon trees $(\ud {\mc T}, {\mf i})$ and $(\ud {\mc T}', {\mf i}')$ consists of a rooted tree isomorphism $\rho: \ud {\mc T} \to \ud {\mc T}'$ together with an isotopy ${\mf i}_t: ( \ud {\mc T}, \partial \ud{\mc T})  \to ({\mb D}, \partial {\mb D})$ as embedding of pairs between ${\mf i}'\circ \rho$ and ${\mf i}$. Two ribbon trees with $\ud k$ semi-infinite edges are {\bf equivalent} if there is an isomorphism between them. An isotopy class of embeddings ${\mf i}$ for a rooted tree ${\mc T}$ is called a {\bf ribbon structure} on ${\mc T}$.

\item A {\bf based rooted tree} consists of a rooted tree ${\mc T}$ and a rooted subtree $\ud{\mc T}$ where the latter is equipped with a ribbon structure.

\item A {\bf colored rooted tree} is a rooted tree ${\mc T}$ together with an order-reversing map ${\mf s}: V({\mc T}) \to \{0, 1, \infty\}$ (called the {\bf coloring}) such that within every path of ${\mc T}$, ${\mf s}^{-1}(1)$ consists of at most one vertex.

\item A vertex $v_\alpha$ of a based colored rooted tree $({\mc T}, \ud{\mc T}, {\mf s})$ is {\bf stable} if one of the followings is true.
\begin{itemize}
\item $v_\alpha = v_\infty$;

\item $v_\alpha \in V({\mc T}) \smallsetminus V(\ud{\mc T})$ and $d_{{\mc T}}(v_\alpha) \geq 3$;

\item $v_\alpha \in V(\ud{\mc T})$, ${\mf s}(v_\alpha) \neq 1$ and $2 d_{\mc T}(v_\alpha) - d_{\ud{\mc T}}(v_\alpha) \geq 3$;

\item ${\mf s}(v_\alpha) = 1$ and $d_{\mc T}(v_\alpha) \geq 2$.
\end{itemize}
\end{enumerate}
\end{defn}

\begin{defn}\label{defn59}
Let $({\mc T}_0, \ud{\mc T}{}_0, {\mf s}_0)$ be a based colored rooted tree. A {\bf growth} is another based colored rooted tree $({\mc T}_1, \ud{\mc T}{}_1, {\mf s}_1)$ with a morphism $\rho: ({\mc T}_1, \ud{\mc T}{}_1, {\mf s}_1) \to ({\mc T}_0, \ud{\mc T}{}_0, {\mf s}_0)$, which is the composition of finitely many elementary growths of the following types.

$\mc{GR}.$ In this case $V({\mc T}_1) = V({\mc T}_0) \cup \{v_\alpha\}$ with $v_\alpha \in {\mf s}_1^{-1}(\infty) \smallsetminus V(\ud{\mc T}{}_1)$. $\rho$ contracts the edge $e_{\alpha \alpha'} \in E({\mc T}_1)$. $\mc{GR}$ corresponds to sphere bubbling downstairs.

$\ud{\mc{GR}}.$ In this case $V({\mc T}_1) = V({\mc T}_0) \cup \{ v_{\ud\alpha}\}$ with $v_{\ud\alpha} \in {\mf s}_1^{-1}(\infty) \cap V(\ud {\mc T}_1)$. $\rho$ contracts the edge $e_{\ud\alpha \ud\alpha'} \in E({\mc T}_1)$. $\ud{\mc{GR}}$ corresponds to disk bubbling downstairs.

$\mc{GS}.$ In this case $V({\mc T}_1) = V({\mc T}_0) \cup \{ v_{\alpha}\}$ with $v_\alpha \in {\mf s}_1^{-1}(1) \smallsetminus V(\ud{\mc T}{}_1)$. $\rho$ contracts the edge $e_{\alpha \alpha'} \in E({\mc T}_1)$. $\mc{GS}$ corresponds to the ${\mb C}$-vortex bubbling.

$\ud{\mc{GS}}.$ In this case $V({\mc T}_1) = V({\mc T}_0) \cup \{ v_{\ud\alpha}\}$ with $v_{\ud\alpha} \in {\mf s}_1^{-1}(0) \cap V(\ud{\mc T}{}_1)$. $\rho$ contracts the edge $e_{\ud\alpha \ud\alpha'}$. $\ud{\mc{GS}}$ corresponds to ${\mb H}$-vortex bubbling.

$\mc{GF}.$ In this case $V({\mc T}_1) = V({\mc T}_0) \cup \{ v_\alpha\}$ with $v_\alpha \in {\mf s}^{-1}(0) \smallsetminus V(\ud{\mc T}{}_1)$. $\rho$ contracts the edge $e_{\alpha \alpha'}$. $\mc{GF}$ corresponds to sphere bubbling upstairs.

$\mc{GF}.$ In this case $V({\mc T}_1) = V({\mc T}_0) \cup \{ v_{\ud\alpha}\}$ with $v_{\ud\alpha} \in {\mf s}^{-1}(0) \cap V(\ud{\mc T}{}_1)$. $\rho$ contracts the edge $e_{\ud\alpha \ud\alpha'}$. $\ud{\mc{GF}}$ corresponds to disk bubbling upstairs.

$\ud{\mc{GD}}.$ In this case $V({\mc T}_1) = V({\mc T}_0) \cup \{ v_{\ud\alpha_1}, \ldots, v_{\ud\alpha_s}\}$ and $\rho$ contracts the path $\ud\alpha \succ \ud\alpha_1 \succ \cdots \succ \ud\alpha_s \succ \ud\alpha'$ in $\ud{\mc T}{}_1$ to the edge $e_{\ud\alpha \ud\alpha'} \in E(\ud{\mc T}{}_0)$. $\ud{\mc{GD}}$ corresponds to the appearance of connecting disk bubbles, either upstairs or downstairs.
\end{defn}

\subsection{Combinatorial types of stable holomorphic spheres and disks}\label{appendixb2}

In this subsection we set up some convention of expressing holomorphic spheres or disks. 

\subsubsection{Holomorphic spheres}

Stable holomorphic spheres are modelled on ordinary trees. We will consider stable holomorphic spheres with a single marked point, which specifies a root of the tree. Therefore, rooted trees are what such objects are modelled on. Let ${\mc T} = (V({\mc T}), E({\mc T}), v_\infty)$ be a rooted tree. A stable holomorphic sphere in an almost K\"ahler manifold $(X, \omega, J)$ modelled on ${\mc T}$ is a collection of objects
\beqn
{\mc S} = \Big( (u_\alpha)_{v_\alpha \in V({\mc T})},\ (z_{\alpha \beta})_{e_{\alpha \beta} \in E({\mc T})} \Big)
\eeqn
where 
\begin{enumerate}
\item For each $v_\alpha \in V({\mc T})$, $u_\alpha: {\mb C} \to X$ is a holomorhpic map with finite energy (therefore extends to a holomorphic sphere with an evaluation $u_\alpha(\infty)$), such that $E(u_\alpha) = 0$ implies that $d_{\mc T}(v_\alpha) \geq 3$ or $v_\alpha = v_\infty$ and $d_{\mc T}(v_\alpha) \geq 2$.

\item For each $e_{\alpha \beta} \in E({\mc T})$, $z_{\alpha \beta} \in {\mb C}$ such that $u_\alpha(\infty) = u_\beta( z_{\alpha \beta})$ and for each $\beta \in V({\mc T})$, the collection of points $Z_\beta:= ( z_{\alpha  \beta} )_{\beta = \alpha'}$ are distinct.
\end{enumerate}
In this situation, we call the rooted tree ${\mc T}$ a {\bf branch}. 

\subsubsection{Holomorphic disks}

Stable holomorphic disks are modelled on based trees. We will consider stable holomorphic disks with a single boundary marked point, which specifies a root of the base. A stable $J$-holomorphic disk in $(X, L)$ modelled on a based rooted tree $({\mc T}, \ud{\mc T}, v_\infty)$ is a collection
\beqn
{\mc D} = \Big( (u_\alpha)_{v_\alpha \in  V({\mc T})},\ (z_{\alpha \beta})_{e_{\alpha \beta} \in E({\mc T})} \Big)
\eeqn
where
\begin{enumerate}
\item For each $v_{\ud\alpha} \in V(\ud{\mc T})$, $u_{\ud\alpha}: ({\mb H}, {\mb R}) \to (X, L)$ is a holomorphic map with finite energy (therefore extends to a holomorphic disk in $(X, L)$ with an evaluation $u_{\ud\alpha}(\infty)$); for each $v_\alpha \in V({\mc T}) \smallsetminus V(\ud{\mc T})$, $u_\alpha: {\mb C} \to X$ is a holomorphic map with finite energy (therefore has an evaluation $u_\alpha(\infty)$); they should satisfy the stability condition: for $v_\alpha \in V({\mc T}) \smallsetminus V(\ud{\mc T})$ and $E(u_\alpha) = 0$, $d_{\mc T}(v_\alpha) \geq 3$; if $v_{\ud \alpha} \in V(\ud{\mc T})\smallsetminus \{ v_\infty\}$ and $E(u_\alpha) = 0$, $2d_{\mc T}(v_{\ud\alpha}) - d_{\ud{T}}(v_{\ud\alpha}) \geq 3$; if $E(u_\infty) = 0$ then $2d_{\mc T}(v_\infty) - d_{\ud{\mc T}}(v_\infty) \geq 2$.

\item (Denote $\Sigma_\alpha = {\mb H}$ if $v_{\alpha} \in V(\ud{\mc T})$ and $\Sigma_\alpha = {\mb C}$ otherwise.) For each $e_{\alpha \beta}\in E(\ud{\mc T})$, $z_{\alpha \beta}\in \partial \Sigma_\beta$; for each $e_{\alpha \beta}\notin E(\ud{\mc T})$, $z_{\alpha \beta}\in {\rm Int} \Sigma_\beta$. They satisfy $u_\alpha(\infty) = u_\beta( z_{\alpha \beta})$ and for each $\beta \in V({\mc T})$, the collection of points $Z_\beta:= (z_{\alpha \beta})_{\alpha' = \beta}$ are distinct.
\end{enumerate}
In this situation, we call the based rooted tree ${\mc T}$ a {\bf based branch}.

\bibliography{symplectic_ref,physics_ref}
\bibliographystyle{amsplain}

\end{document}